\documentclass[a4paper]{amsart}

\usepackage[foot]{amsaddr}
\usepackage{amsthm,amssymb}
\usepackage{mathtools,thmtools}
\usepackage{hyperref}
\usepackage{bm,esint}
\usepackage{color}
\usepackage{cleveref}
\usepackage{cancel}
\usepackage{etoolbox}
\usepackage{caption}
\usepackage{subfigure}
\usepackage{tikz}
\usetikzlibrary{positioning}
\usetikzlibrary{arrows}

\title[Approximating rough functions with deep neural networks]{On the approximation of rough functions with deep neural networks}
\author{T.~De Ryck}
\email{tim.deryck@sam.math.ethz.ch}

\author{S.~Mishra}
\address[T. De Ryck and S. Mishra]{Seminar for Applied Mathematics, ETH Z\"urich, R\"amistrasse 101, 8092 Z\"urich, Switzerland}
\email{siddhartha.mishra@sam.math.ethz.ch}

\author{D.~Ray}
\address[D. Ray]{University of Southern California, Los Angeles, USA}
\email{deepray@usc.edu}

\newtheorem{theorem}{Theorem}[section]

\newtheorem{remark}[theorem]{Remark}

\newtheorem{lemma}[theorem]{Lemma}
\newtheorem{proposition}[theorem]{Proposition}

\numberwithin{equation}{section}
\numberwithin{theorem}{section}

\usepackage{aliases2}
\usepackage{tikz}
\usepackage{mathcomp}
\usepackage{bm}
\usepackage{multirow}
\usepackage[ruled,vlined]{algorithm2e}
\usepackage{bbm}
\numberwithin{equation}{section} \allowdisplaybreaks

\newcommand{\f}{ f}
\newcommand{\db}{d}

\newcommand{\Ro}{\mathbb{R}}

\newcommand{\pdeg}{p}
\newcommand{\Tau}{\mathcal{T}}

\newcommand{\func}{\mathcal{L}}
\newcommand{\funcnn}{\mathcal{L^\theta}}

\newcommand{\dset}{\mathbb{S}}

\newcommand{\act}{\mathcal{A}}
\newcommand{\W}{W}
\newcommand{\bb}{b}
\newcommand{\ofunc}{\mathcal{S}}

\renewcommand{\leq}{\leqslant}

\renewcommand{\geq}{\geqslant}

\newcommand{\ud}{\text{d}}

\newcommand{\norm}[1]{\left\lVert#1\right\rVert}
\newcommand{\abs}[1]{\lvert#1\rvert}

\begin{document}

\begin{abstract}
The essentially non-oscillatory (ENO) procedure and its variant, the ENO-SR procedure, are very efficient algorithms for interpolating (reconstructing) rough functions. We prove that the ENO (and ENO-SR) procedure are equivalent to deep ReLU neural networks. This demonstrates the ability of deep ReLU neural networks to approximate rough functions to high-order of accuracy. Numerical tests for the resulting trained neural networks show excellent performance for interpolating functions, approximating solutions of nonlinear conservation laws and at data compression.
\end{abstract}

\maketitle
\setcounter{tocdepth}{1}

\section{Introduction}\label{sec:intro}
Rough functions i.e, functions which are at most Lipschitz continuous and could even be discontinuous, arise in a wide variety of problems in physics and engineering. Prominent examples include (weak) solutions of nonlinear partial differential equations. For instance, solutions of nonlinear hyperbolic systems of conservation laws such as the compressible Euler equations of gas dynamics, contain shock waves and are in general discontinuous, \cite{DAF10}. Similarly, solutions to the incompressible Euler equations would well be only H\"older continuous in the turbulent regime, \cite{ES1}. Moreover, solutions of fully non-linear PDEs such as Hamilton-Jacobi equations are in general Lipschitz continuous, \cite{Evans}. Images constitute another class of rough or rather piecewise smooth functions as they are often assumed to be no more regular than functions of bounded variation on account of their sharp edges, \cite{IPBook}. 

Given this context, the efficient and robust numerical approximation of rough functions is of great importance. However, classical approximation theory has severe drawbacks when it comes to the interpolation (or approximation) of such rough functions. In particular, it is well known that standard linear interpolation procedures degrade to at best first-order of accuracy (in terms of the interpolation mesh width) as soon as the derivative of the underlying function has a singularity, \cite{ACDD2005} and references therein. This order of accuracy degrades further if the underlying function is itself discontinuous. Moreover approximating rough functions with polynomials can lead to spurious oscillations at points of singularity. Hence, the approximation of rough functions poses a formidable challenge.

Artificial neural networks, formed by concatenating affine transformations with pointwise application of nonlinearities, have been shown to possess universal approximation properties, \cite{HORNIK89,BARRON93,CYB88} and references therein. This implies that for any continuous (even for merely measurable) function, there exists a neural network that approximates it accurately. However, the precise architecture of this network is not specified in these universality results. Recently in \cite{YAROTSKY17}, Yarotsky was able to construct deep neural networks with ReLU activation functions and very explicit estimates on the size and parameters of the network, that can approximate Lipschitz functions to second-order accuracy. Even more surprisingly, in a very recent paper \cite{YZ1}, the authors were able to construct deep neural networks with alternating ReLU and Sine activation functions that can approximate Lipschitz (or H\"older continuous) functions to exponential accuracy.

The afore-mentioned results of Yarotsky clearly illustrate the power of deep neural networks in approximating rough functions. However, there is a practical issue in the use of these deep neural networks as they are mappings from the space coordinate $x \in D \subset \mathbb{R}^d$ to the output $f^{\ast}(x) \in \mathbb{R}$, with the neural network $f^{\ast}$ approximating the underlying function $f:D \to \mathbb{R}$. Hence, for every given function $f$, the neural network $f^{\ast}$ has to be \emph{trained} i.e, its weights and biases determined by minimizing a suitable loss function with respect to some underlying samples of $f$, \cite{DLBOOK}. Although it makes sense to train neural networks to approximate individual functions $f$ in high dimensions, for instance in the context of uncertainty quantification of PDEs, \cite{LMR1} and references therein, doing so for every low-dimensional function is unrealistic. Moreover, in a large number of contexts, the goal of approximating a function is to produce an interpolant $\tilde{f}$, given the vector $\{f(x_i)\}$ at sampling points $x_i \in D$ as input. Hence, one would like to construct neural networks that map the full input vector into an output interpolant (or its evaluation at certain sampling points). It is unclear if the function approximation results for neural networks are informative in this particular context. 

On the other hand, data-dependent interpolation procedures have been developed in the last decades to deal with the interpolation of rough functions. A notable example of these data dependent algorithms is provided by the essentially non-oscillatory (ENO) procedure. First developed in the context of reconstruction of non-oscillatory polynomials from cell averages in \cite{HEOC1987}, ENO was also adapted for interpolating rough functions in \cite{SHU89b} and references therein. Once augmented with a sub-cell resolution (SR) procedure of \cite{harten_1989}, it was proved in \cite{ACDD2005} that the ENO-SR interpolant also approximated (univariate) Lipschitz functions to second-order accuracy. Moreover, ENO was shown to satisfy a subtle non-linear stability property, the so-called \emph{sign property}, \cite{FMT}. Given these desirable properties, it is not surprising that the ENO procedure has been very successfully employed in a variety of contexts, ranging from the numerical approximation of hyperbolic systems of conservation laws \cite{HEOC1987} and Hamilton-Jacobi equations \cite{SO91} to data compression in image processing, \cite{HARTEN97,ACDD2005,ACDDM} and references therein. 

Given the ability of neural networks as well as ENO algorithms to approximate rough functions accurately, it is natural to investigate connections between them. This is the central premise of the current paper, where we aim to reinterpret ENO (and ENO-SR) algorithms in terms of deep neural networks. We prove the following results,
\begin{itemize}
    \item We prove that for any order, the ENO interpolation (and the ENO reconstruction) procedure can be cast as a suitable deep ReLU neural network.
    \item We prove that a variant of the piecewise linear ENO-SR (sub-cell resolution) procedure of \cite{harten_1989} can also be cast as a deep ReLU neural network. Thus, we prove that there exists a deep ReLU neural network that approximates piecewise smooth (say Lipschitz) functions to second-order accuracy. 
    \item The above theorems provide the requisite architecture for the resulting deep neural networks and we train them to obtain what we term as DeLENO (deep learning ENO) approximation procedures for rough functions. We test this procedure in the context of numerical methods for conservation laws and for data and image compression. 
\end{itemize}
Thus, our results reinforce the enormous abilities of deep ReLU neural networks to approximate functions, in particular rough functions and add a different perspective to many existing results on approximation with ReLU networks.

\section{Deep neural networks}\label{sec:DNN}

In statistics, machine learning, numerical mathematics and many other scientific disciplines, the goal of a certain task can often be reduced to the following. We consider a (usually unknown) function $\func:D \subset \mathbb{R}^m \to \mathbb{R}^n$ and we assume access to a (finite) set of labelled data $\dset \subset \{ (X,\func(X)):X \in D \}$, using which we wish to select an approximation $\hat{\func}$ from a parametrized function class $\{\funcnn:\theta\in\Theta\}$ that predicts the outputs of $\func$ on $D$ with a high degree of accuracy.

One possible function class is that of deep neural networks (DNNs). In particular, we consider multilayer perceptrons (MLPs) in which the basic computing units (neurons) are stacked in multiple layers to form a feedforward network. The input is fed into the \textit{source layer} and flows through a number of \textit{hidden layers} to the \textit{output layer}.  An example of an MLP with two hidden layers is shown in Figure \ref{fig:MLP}. 

In our terminology, an MLP of depth $L$ consists of an input layer, $L-1$ hidden layers and an output layer. We denote the vector fed into the input layer by $X=Z^0$. 
The $l$-th layer (with $n_l$ neurons) receives an input vector $Z^{l-1} \in \Ro^{n_{l-1}}$ and transforms it into the vector $Z^{l} \in \Ro^{n_{l}}$ by first applying an affine linear transformation, followed by a component-wise (non-linear) activation function $\mathcal{A}^l$,
\begin{equation}\label{eqn:1HL}
Z^{l} = \act^l(W^l Z^{l-1} + b^l), \quad W^l \in \Ro^{n_{l} \times n_{l-1}}, \ b^l \in \Ro^{n_{l}}, \quad 1 \leq l \leq L,
\end{equation}
with $Z^l$ serving as the input for the $(l+1)$-th layer. For consistency, we set $n_0 = m$ and $n_L = n$. In \eqref{eqn:1HL}, $W^l$ and $b^l$ are respectively known as the weights and biases associated with the $l$-th layer. The parameter space $\Theta$ then consists of all possible weights and biases. A neural network is said to be deep if $L\geq3$ and such a deep neural network (DNN) is denoted as a ReLU DNN if the activation functions are defined by the very popular rectified linear (ReLU) function,
\begin{equation}\label{eqn:relu}
\mathcal{A}^l(Z) = (Z)_+ = \max(0,Z) \quad\textrm{for}\quad 1\leq l\leq L-1\quad\textrm{and}\quad \act^L(Z)=Z.
\end{equation}
\begin{figure}[!h]
\begin{center}
\includegraphics[width=0.8\textwidth]{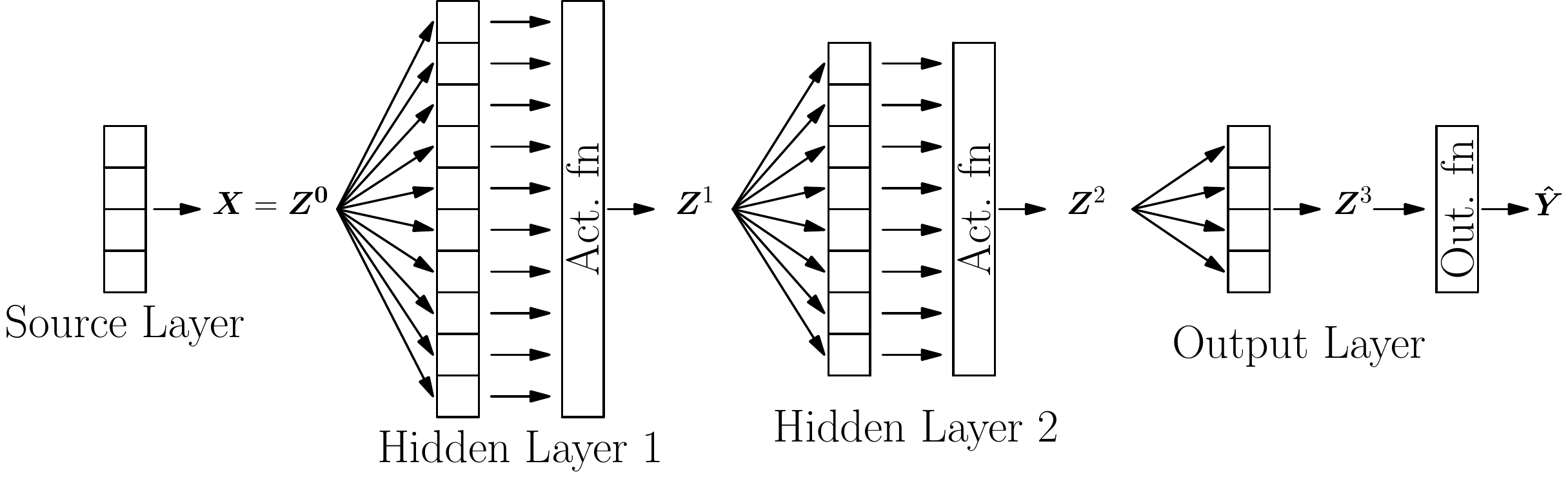} 
\caption{An MLP with 2 hidden layers. The source layer transmits the signal $X$ to the first hidden layer. The final output of the network is $\hat{Y}$.}
\label{fig:MLP}
\end{center}
\end{figure}
Depending on the nature of the problem, the output of the ANN may have to pass through an output function $\ofunc$ to convert the signal into a meaningful form. In classification problems, a suitable choice for such an output function would be the \textit{softmax} function
\begin{equation}
    \ofunc(x):\mathbb{R}^n\to \mathbb{R}^n:x\mapsto \left(  \frac{e^{x_1}}{\sum_{j=1}^n e^{x_j}},\ldots,\frac{e^{x_n}}{\sum_{j=1}^n e^{x_j}}\right).
\end{equation}
This choice ensures that the final output vector $\hat{Y}=\ofunc(Z^L)$ satisfies $\sum_{j=1}^n \hat{Y}_j = 1$ and $0 \leq \hat{Y}_j \leq 1$ for all $1\leq j \leq n$, which allows 
$\hat{Y}_j$ to be viewed as the probability that the input $Z^0$ belongs to the $j$-th class. Note that the class predicted by the network is $\text{arg} \max_j \hat{Y}_j$. For regression problems, no additional output function is needed. 

\begin{remark}\label{rem:argmax}
It is possible that multiple classes have the largest probability. In this case, the predicted class can be uniquely defined as $\min \mathrm{arg} \max _j\{\hat{Y}_j\}$, following the usual coding conventions. Also note that the softmax function only contributes towards the interpretability of the network output and has no effect on the predicted class, that is,
\[\min \mathrm{arg} \max _j\{\hat{Y}_j\} = \min \mathrm{arg} \max _j\{Z^L_j\}.\]
This observation will be used at a later stage. 
\end{remark}

The expressive power of ReLU neural networks, in particular their capability of approximating rough functions, has already been demonstrated in literature \cite{petersen2018optimal, YZ1}. In practice however, there is a major issue in this approach when used to approximate an unknown function $f:D\subseteq \mathbb{R} \to \mathbb{R}$ based on a finite set $\mathbb{S}\subset\{(x,f(x)):x\in D\}$, as for each individual function $f$ a new neural network has to be found. This network (or an approximation) is found by the process of \textit{training} the network. The computational cost is significantly higher than that of other classical regression methods in low dimensions, which makes this approach rather impractical, at least for functions in low dimensions. This motivates us to investigate how one can obtain a neural network that takes input in $D$ and produces an output interpolant, or rather its evaluation at certain sample points. Such a network primarily depends only on the training data $\mathbb{S}$ and can be reused for each individual function, thereby drastically reducing the computational cost. Instead of creating an entirely novel data-dependent interpolation procedure, we base ourselves in this paper on the essentially non-oscillatory (ENO) interpolation framework of \cite{HEOC1987, harten_1989}, which we introduce in the next section.  


\section{ENO framework for interpolating rough functions}

In this section, we explore the \textit{essentially non-oscillatory} (ENO) interpolation framework \cite{HEOC1987, harten_1989}, on which we will base our theoretical results of later sections. Although the ENO procedure has its origins in the context of the numerical approximation of solutions of hyperbolic conservation laws, this data-dependent scheme has also proven its use for the interpolation of rough functions \cite{SHU89b}. 

\subsection{ENO interpolation}\label{sec:ENO}
We first focus on the original ENO procedure, as introduced in \cite{HEOC1987}. This procedure can attain any order of accuracy for smooth functions, but reduces to first-order accuracy for functions that are merely Lipschitz continuous. In particular, the ENO-$p$ interpolant is $p$-th order accurate in smooth regions and suppresses the appearance of spurious oscillations in the vicinity of points of discontinuity. In the following, we describe the main idea behind this algorithm. 

Let $f$ be a function on $\Omega = [c,d]\subset \mathbb{R}$ that is at least $p$ times continuously differentiable. 
We define a sequence of nested uniform grids $\{\Tau^k\}_{k=0}^K$ on $\Omega$, where
\begin{equation}\label{eqn:interp_partitions}
\Tau^k = \{x_i^k\}_{i=0}^{N_k}, \: I_i^k = [x^k_{i-1}, x^k_i], \: x_i^k = c + i h_k, \: h_k = \frac{(d-c)}{N_k}, \: N_k = 2^k N_0,
\end{equation}
for $0\leq i \leq N_k$, $0\leq k\leq K$ and some positive integer $N_0$. 
Furthermore we define $f^k=\{f(x):x\in \Tau^k\}$, $f^k_i=f(x_i^k)$ and we let $f^k_{-\pdeg+2}, ...,f^k_{-1}$ and $f^k_{N_k+1},...f^k_{N_k + \pdeg-2}$ be suitably prescribed ghost values. We are interested in finding an interpolation operator $\mathcal{I}^{h_k}$ such that
\begin{equation*}
    \mathcal{I}^{h_k}f(x)=f(x) \textrm{ for } x\in\Tau^k \quad \textrm{and} \quad \norm{\mathcal{I}^{h_k}f-f}_{\infty}=O(h_k^p) \textrm{ for } k\to \infty.
\end{equation*}
In standard approximation theory, this is achieved by defining $\mathcal{I}^{h_k}f$ on $I_i^k$ as the unique polynomial $p^k_i$ of degree $p-1$ that agrees with $f$ on a chosen set of $p$ points, including $x_{i-1}^k$ and $x_i^k$. The linear interpolant ($\pdeg=2$) can be uniquely obtained using the stencil $\{ x^k_{i-1},x^k_i \}$. However, there are several candidate stencils to choose from when $\pdeg > 2$. The ENO interpolation procedure considers the stencil sets 
\begin{equation*}
    \mathcal{S}^r_i = \{x^k_{i-1-r+j}\}_{j=0}^{p-1}, \quad 0\leq r \leq p-2,
\end{equation*}
where $r$ is called the (left) stencil shift. The smoothest stencil is then selected based on the local smoothness of $f$ using Newton's undivided differences. These are inductively defined in the following way. Let $\Delta_j^0 = f^k_{i+j}$ for $-p+1\leq j \leq p-2$ and $0\leq i \leq N_k$. We can then define 
\begin{equation*}
    \Delta^s_{j} = \begin{cases}
    \Delta^{s-1}_{j}-\Delta^{s-1}_{j-1} & \text{for $s$ odd}\\
    \Delta^{s-1}_{j+1}-\Delta^{s-1}_{j} & \text{for $s$ even.}
    \end{cases}
\end{equation*}
Algorithm \ref{alg:eno_int_select} describes how the stencil shift $r$ can be obtained using these undivided differences. Note that $r$ uniquely defines the polynomial $p_i^k$. We can then write the final interpolant as 
\begin{equation*}
     \mathcal{I}^{h_k}f(x)=\sum_{i=1}^{N_k}p^k_i(x)\mathbbm{1}_{[x_{i-1}^k,x_i^k)}(x).
\end{equation*}
This interpolant can be proven to be \textit{total-variation bounded} (TVB), which guarantees the disappearance of spurious oscillations (e.g. near discontinuities) when the grid is refined. This property motivates the use of the ENO framework over standard techniques for the interpolation of rough functions. 

In many applications, one is only interested in predicting the values of $f^{k+1}$ given $f^k$. In this case, there is no need to calculate $\mathcal{I}^{h_{k}}f$ and evaluate it on $\Tau^{k+1}$. Instead, one can use Lagrangian interpolation theory to see that there exist fixed coefficients $C_{r_,j}^p$ such that
\begin{equation}\label{eqn:eno_int}
\begin{split}
    \mathcal{I}^{h_{k}}f(x^{k+1}_{2i-1}) &= \sum_{j=0}^{p-1} C_{r_i^k,j}^p f^k_{i-r_i^k+j} \: \textrm{ for }\:  1\leq i \leq N_k \quad \textrm{and}\quad\\  \mathcal{I}^{h_{k}}f(x^{k+1}_{2i})&=f^{k+1}_{2i}=f^k_i \: \textrm{ for }\:  0\leq i \leq N_k,
\end{split}
\end{equation}
where $r_i^k$ is the stencil shift corresponding to the smoothest stencil for interval $I_i^k$. The coefficients $C_{r,j}^p$ are listed in Table \ref{tab:coef_int} in Appendix \ref{app:ENO-coeff}. 

\begin{remark}
ENO was initially introduced by \cite{HEOC1987} for high-order accurate piecewise polynomial reconstruction, given cell averages of a function. This allows the development of high-order accurate numerical methods for hyperbolic conservation laws, the so-called ENO schemes. ENO reconstruction can be loosely interpreted as ENO interpolation applied to the primitive function and  is discussed in Appendix \ref{sec:ENO-reconstruction}. 
\end{remark}

\begin{remark}
The prediction of $f^{k+1}$ from $f^k$ can be framed in the context of multi-resolution representations of functions, which are useful for data compression \cite{HARTEN97}. As we will use ENO interpolation for data compression in Section \ref{sec:num-results}, we refer to Appendix \ref{app:multi-res} for details on multi-resolution representations. 
\end{remark}

\begin{algorithm}
\caption{ENO interpolation stencil selection}
\KwIn{ENO order $\pdeg$, input array $\bf{\Delta^0}$ $= \{f^k_{i+j}\}_{j=-p+1}^{\pdeg-2}$, for any $0 \leq i \leq N_k$.} 
\KwOut{Stencil shift $r$.}
\textit{Evaluate Newton undivided differences:} 

\For{$j=1$ \KwTo $\pdeg-1$}{
${\bf{\Delta^j}} = {\bf{\Delta^{j-1}}}[2:\textrm{end}] - {\bf{\Delta^{j-1}}}[1:\textrm{end}-1]$
}
\textit{Find shift: } 

$r=0$

\For{$j=2$ \KwTo $\pdeg-1$}{
\If{$|{\bf{\Delta^j}}[\pdeg-2-r] | < |{\bf{\Delta^j}}[\pdeg-1-r] |$}{ $r = r+1$}
}
\Return $r$
\label{alg:eno_int_select}
\end{algorithm}

\subsection{An adapted second-order ENO-SR algorithm}\label{sec:enosr-alg}

Even though ENO is able to interpolate rough functions without undesirable side effects (e.g. oscillations near discontinuities), there is still room for improvement. By itself, the ENO interpolation procedure degrades to first-order accuracy for piecewise smooth functions i.e, functions with a singularity in the second derivative. However, following \cite{harten_1989}, one can use \textit{sub-cell resolution} (SR), together with ENO interpolation, to obtain a second-order accurate approximation of such functions. We propose a simplified variant of the ENO-SR procedure from \cite{ACDD2005} and prove that it is still second-order accurate. In the following, we assume $f$ to be a continuous function that is two times differentiable except at a single point $z$ where the first derivative has a jump of size $[f'] = f^{\prime}(z+) - f^{\prime}(z-) $. We use the notation introduced in Section \ref{sec:ENO}.

The first step of the adapted second-order ENO-SR algorithm is to label intervals that might contain the singular point $z$ as \textit{bad} ($B$), other intervals get the label \textit{good} ($G$). We use second-order differences
\begin{equation}\label{eqn:second-diff}
    \Delta^2_hf(x):=f(x-h)-2f(x)+f(x+h)
\end{equation}
as smoothness indicators. The rules of the ENO-SR detection mechanism are the following: 
\begin{enumerate}
    \item The intervals $I_{i-1}^k$ and $I_i^k$ are labelled $B$ if
    \begin{equation*}
        \abs{\Delta_{h_k}^2f(x^k_{i-1})} > \max_{n=1,2,3}\abs{\Delta_{h_k}^2f(x^k_{i-1\pm n})}.
    \end{equation*}
    \item Interval $I_i^k$ is labelled $B$ if
    \begin{equation*}
        \abs{\Delta_{h_k}^2f(x^k_{i})} > \max_{n=1,2}\abs{\Delta_{h_k}^2f(x^k_{i+n})} \quad \textrm{and} \quad \abs{\Delta_{h_k}^2f(x^k_{i-1})} > \max_{n=1,2}\abs{\Delta_{h_k}^2f(x^k_{i-1-n})}.
    \end{equation*}
    \item All other intervals are labelled $G$. 
\end{enumerate}

\begin{figure}[!htbp]
  \centering
  \subfigure{\includegraphics[width=0.49\textwidth]{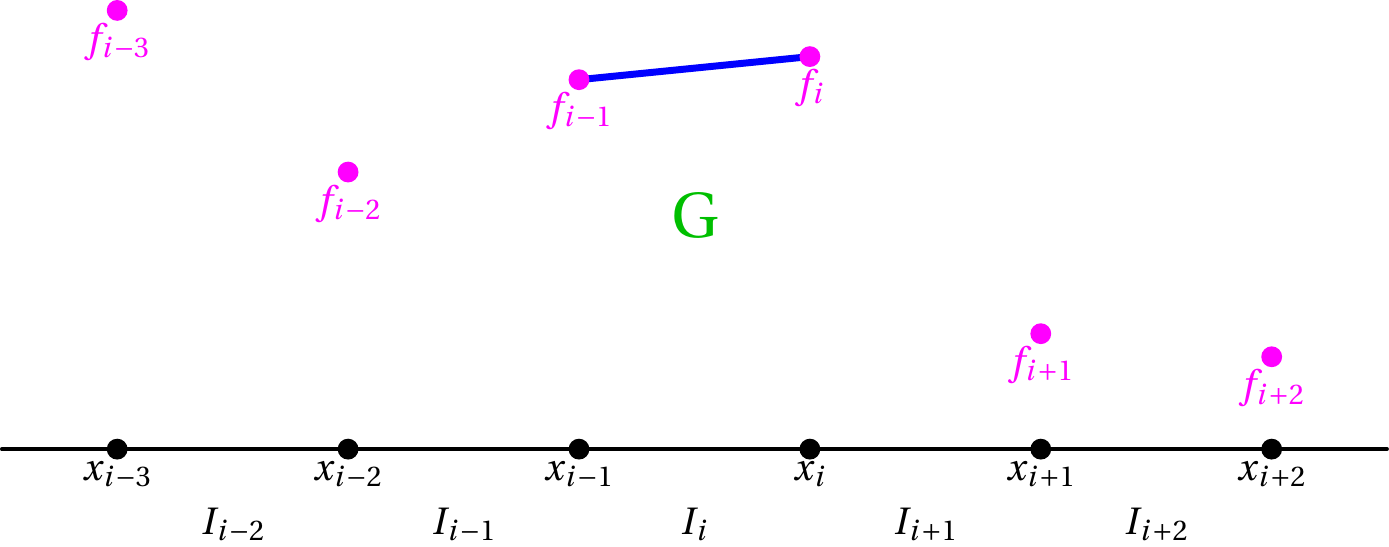}}
  \subfigure{\includegraphics[width=0.49\textwidth]{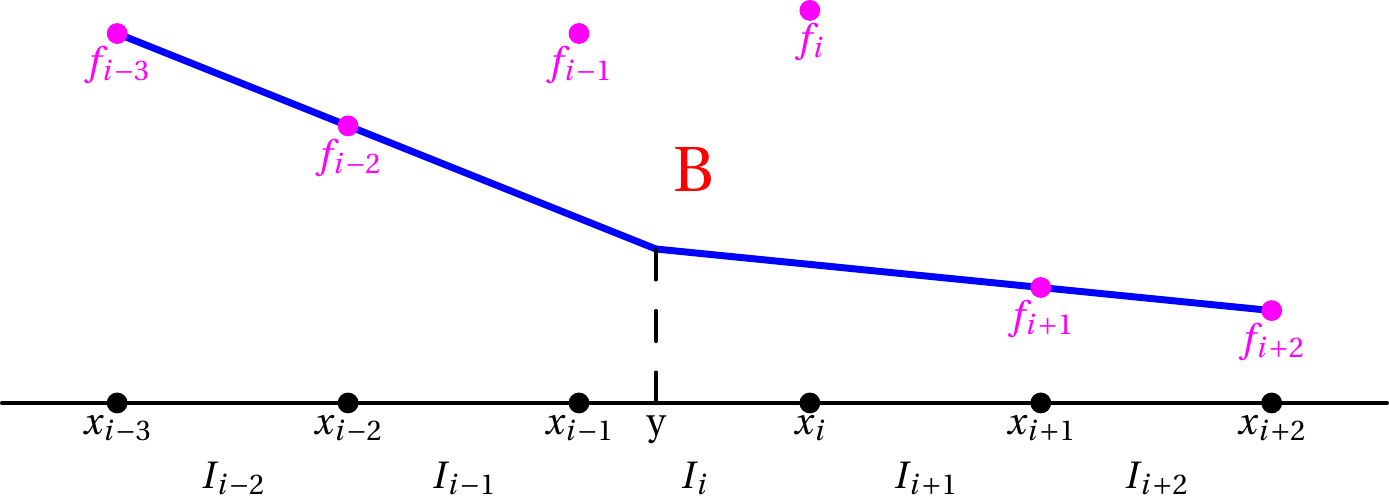}}
  \caption{Visualization of the second-order ENO-SR algorithm in the case where an interval is labelled as good (left) and bad (right). The superscript $k$ was omitted for clarity.}\label{fig:enosr}
\end{figure}

Note that neither detection rule implies the other and that an interval can be labelled $B$ by both rules at the same time. In the following, we will denote by ${p_i^k:[c,d]\rightarrow \mathbb{R}}$ the linear interpolation of the endpoints of $I_i^k$.
The rules of the interpolation procedure are stated below, a visualization of the algorithm can be found in Figure \ref{fig:enosr}. 
\begin{enumerate}
    \item If $I_i^k$ was labelled as $G$, then we take the linear interpolation on this interval as approximation for $f$, 
    \begin{equation*}
        \mathcal{I}_i^{h_k}f(x) = p_i^k(x).
    \end{equation*}
    \item If $I_i^k$ was labelled as $B$, we use $p^k_{i-2}$ and $p^k_{i+2}$ to predict the location of the singularity. If both lines intersect at a single point $y$, then we define 
    \begin{equation*}
        \mathcal{I}_i^{h_k}f(x) = p^k_{i-2}(x)\mathbbm{1}_{[c,\max\{y,c\})}(x)+p^k_{i+2}(x)\mathbbm{1}_{[\min\{y,d\},d]}(x).
    \end{equation*}
    The relation between this intersection point $y$ and the singularity $z$ is quantified by Lemma \ref{lem:L3-cohen}.
    If the two lines do not intersect, we treat $I_i^k$ as a good interval and let $\mathcal{I}_i^{h_k}f(x) = p_i^k(x)$. 
\end{enumerate}
The theorem below states that our adaptation of ENO-SR is indeed second-order accurate.
\begin{theorem}\label{thm:ENOSR2}
Let $f$ be a globally continuous function with a bounded second derivative on $\mathbb{R}\backslash \{z\}$ and a discontinuity in the first derivative at a point $z$. The adapted ENO-SR interpolant $\mathcal{I}^hf$ satisfies
\begin{equation*}
    \norm{f-\mathcal{I}^hf}_{\infty} \leq Ch^2 \sup_{\mathbb{R}\backslash \{z\}}\abs{f''}
\end{equation*}
for all $h>0$, with $C>0$ independent of $f$. 
\end{theorem}
\begin{proof}
The proof is an adaptation of the proof of Theorem 1 in \cite{ACDD2005} and can be found in Appendix \ref{app:proof-ENOSR2}. 
\end{proof}

\section{ENO as a ReLU DNN}\label{sec:ENO-DNN}
As mentioned in the introduction, we aim to recast the ENO interpolation algorithm from Section \ref{sec:ENO} as a ReLU DNN. Our first approach to this end begins by noticing that the crucial step of the ENO procedure is determining the correct stencil shift. Given the stencil shift, the retrieval of the ENO interpolant is straightforward. ENO-$p$ can therefore be interpreted as a classification problem, with the goal of mapping an input vector (the evaluation of a certain function on a number of points) to one of the $p-1$ classes (the stencil shifts). 
We now present one of the main results of this paper. The following theorem states that the stencil selection of $p$-th order ENO interpolation can be exactly obtained by a ReLU DNN for every order $p$. The stencil shift can be obtained from the network output by using the default output function for classification problems (cf. Remark \ref{rem:argmax}).

\begin{theorem}\label{thm:ENO-DNN}
There exists a ReLU neural network consisting of $p+ \left\lceil \log_2 \binom{p-2}{\lfloor \frac{p-2}{2}\rfloor} \right\rceil $ hidden layers,
that takes input $\bf{\Delta^0}$ $= \{f^k_{i+j}\}_{j=-p+1}^{\pdeg-2}$ and leads to exactly the same stencil shift as the one obtained by Algorithm \ref{alg:eno_int_select}. 
\end{theorem}
\begin{proof}
We first sketch an intuitive argument why there exists a ReLU DNN that, after applying $\min\argmax$, leads to the ENO-$p$ stencil shift. Algorithm \ref{alg:eno_int_select} maps every input stencil $\bf{\Delta^0}$ $\in [c,d]^{2p-2}$ to a certain stencil shift $r$. A more careful look at the algorithm reveals that the input space $[c,d]^{2p-2}$ can be partioned into polytopes such that the interior of every polytope is mapped to one of the $p-1$ possible stencil shifts. Given that every ReLU DNN is a continuous, piecewise affine function (e.g. \cite{arora2016understanding}), one can construct for every $i\in\{0, \ldots, p-2\}$ a ReLU DNN $\phi_i: [c,d]^{2p-2}\to \mathbb{R}$ that is equal to 0 on the interior of every polytope corresponding to stencil shift $i$ and that is strictly smaller than 0 on the interior of every polytope not corresponding to stencil shift $i$. It is then clear that 
\begin{equation}\label{eq:eno-short-proof}
    \min\argmax\{\phi_0, \ldots, \phi_{p-2}\}-1
\end{equation}
corresponds to the ENO stencil shift on the interiors of all polytopes. Thanks to the minimum in \eqref{eq:eno-short-proof}, it also corresponds to the unique stencil shift from Algorithm \ref{sec:enosr-alg} on the faces of the polytopes, where multiple $\phi_i$ are equal to zero. The claim then follows from the fact that the mapping ${\bf{\Delta^0}}$ $\mapsto (\phi_0({\bf{\Delta^0}}), \ldots, \phi_{p-2}({\bf{\Delta^0}}))$ can be written as a ReLU DNN. 

In what follows, we present a more constructive proof that sheds light on the architecture that is needed to represent ENO and the sparsity of the corresponding weights. In addition, a technique to replace the Heaviside function (as in Algorithm \ref{alg:eno_int_select}) is used. Recall that we look for the ENO stencil shift $r:=r_i^k$ corresponding to the interval $I_i^k$. Let $k\in\mathbb{N}$ and define $\Delta_j^0 = f^k_{i+j}$ for $-p+1\leq j \leq p-2$ and $0\leq i \leq N_k$, where $f^k_{-p+1},\ldots, f^k_{-1}$ and $f^k_{N_k+1},\ldots, f^k_{N_k+p-2}$ are suitably defined ghost values. Following Section \ref{sec:ENO}, we define $\Delta^s_{j}=\Delta^{s-1}_{j}-\Delta^{s-1}_{j-1}$ for $s$ odd and $\Delta^s_{j}=\Delta^{s-1}_{j+1}-\Delta^{s-1}_{j}$ for $s$ even, and with $\bf{\Delta^s}$ we denote the vector consisting of all $\Delta^s_j$ for all applicable $j$. In what follows, we use $Y^l$ and $Z^l$ to denote the values of the $l$-th layer of the neural network before and after activation, respectively. We use the notation $X^l$ for an auxiliary vector needed to calculate $Y^l$. 

\textbf{Step 1. } Take the input to the network to be 
\begin{equation*}
    Z^0 = [\Delta_{-p+1}^0,\ldots,\Delta_{p-2}^0]\in\mathbb{R}^{2(p-1)}.
\end{equation*}
These are all the candidate function values considered in Algorithm \ref{alg:eno_int_select}. 

\textbf{Step 2. } We want to obtain all quantities $\Delta_{j}^s$ that are compared in Algorithm \ref{alg:eno_int_select}, as shown in Figure \ref{fig:ndd}. We therefore choose the first layer (before activation) to be 
\begin{equation*}
    Y^1 = \begin{bmatrix} Y_{\Delta} \\ -Y_{\Delta} \end{bmatrix} \in \mathbb{R}^{2M} \quad \textrm{where} \quad Y_{\Delta} = \begin{bmatrix} \Delta_{0}^2 \\\Delta_{-1}^2  \\ \vdots \end{bmatrix} \in \mathbb{R}^{M}
\end{equation*}
is the vector of all the terms compared in Algorithm \ref{alg:eno_int_select} and $M=\frac{p(p-1)}{2}-1$. Note that every undivided difference is a linear combination of the network input. Therefore one can obtain $Y^1$ from $Z^0$ by taking a null bias vector and weight matrix $W^1\in\mathbb{R}^{2M\times (2p-2)}$. After applying the ReLU activation function, we obtain
\begin{equation*}
    Z^1 = \begin{bmatrix} (Y_{\Delta})_+ \\ (-Y_{\Delta})_+ \end{bmatrix}.
\end{equation*}

\textbf{Step 3. } We next construct a vector $X^2\in\mathbb{R}^{L}$, where $L=\frac{(p-2)(p-1)}{2}$, that contains all the quantities of the if-statement in Algorithm \ref{alg:eno_int_select}. This is ensured by setting,
\begin{equation*}
    X^2 = \begin{bmatrix} \abs{\Delta_{-1}^2}-\abs{\Delta_{0}^2}\\ \abs{\Delta_{0}^3}-\abs{\Delta_{1}^3} \\ \abs{\Delta_{-1}^3}-\abs{\Delta_{0}^3} \\\vdots \end{bmatrix}. 
\end{equation*}
Keeping in mind that $\abs{a}=(a)_++(-a)_+$ for $a\in\mathbb{R}$ we see that there is a matrix $\widetilde{W}^2\in\mathbb{R}^{L\times 2M}$ such that $X^2=\widetilde{W}^2 Z^1$. 
We wish to quantify for each component of $X^2$ whether it is strictly negative or not (cf. the if-statement of Algorithm \ref{alg:eno_int_select}). For this reason, we define the functions
$H_1:\mathbb{R}\rightarrow \mathbb{R}$ and $H_2:\mathbb{R}\rightarrow\mathbb{R}$ by
\begin{equation*}
    H_1(x) = \begin{cases} 0 &x\leq -1 \\ x+1 &-1 < x < 0 \\ 1 &x \geq 0 \end{cases} \quad \textrm{and} \quad H_2(x) = \begin{cases} -1 &x\leq 0 \\ x-1 &0 < x < 1 \\ 0 &x \geq 1 \end{cases}. 
\end{equation*}
The key property of these functions is that $H_1$ and $H_2$ agree with the Heaviside function on $x>0$ and $x<0$, respectively. When $x=0$ the output is respectively $+1$ or $-1$. Now note that $H_1(x) = (x+1)_{+} - (x)_{+}$ and $H_2(x) =  (x)_{+} - (x-1)_{+}-1$. 
This motivates us to define
\begin{equation*}
    Y^2 = \begin{bmatrix} X^2+1 \\ X^2 \\ X^2-1 \end{bmatrix}\in\mathbb{R}^{3L},
\end{equation*}
which can be obtained from $Z^1$ by taking weight matrix $W^2\in\mathbb{R}^{3L\times 2M}$ and bias vector $b^2\in\mathbb{R}^{3L}$, 
\begin{equation*}
    W^2=\left(\begin{bmatrix} 1\\1\\1 \end{bmatrix}\otimes \mathbb{I}_{L}\right)\cdot\widetilde{W}^2 \quad \textrm{and} \quad b^2_j=\begin{cases} 1 &1 \leq j \leq L \\ 0 &L+1 \leq j \leq 2L \\ -1 &2L+1 \leq j \leq 3L \end{cases}
\end{equation*}
where $\mathbb{I}_{L}$ denotes the $L\times L$ unit matrix. 
After activation we obtain $Z^2=(Y^2)_+=(W^2Z^1+b^2)_+$.

\textbf{Step 4. } We first define $X^3 \in \mathbb{R}^{2L}$ by
\begin{equation*}
    X^3_j = \begin{cases} H_1(X^2_j) = Z^2_j-Z^2_{L+j}& 1\leq j \leq L \\ H_2(X^2_{j-L}) = Z^2_j-Z^2_{L+j}-1 & L + 1 \leq j \leq 2L. \end{cases}
\end{equation*}
This is clearly for every $j$ an affine transformation of the entries of $Z^2$. For this reason there exist a matrix $\widetilde{W}^3\in\mathbb{R}^{2L\times 3L}$ and a bias vector $\Tilde{b}^3\in\mathbb{R}^{2L}$ such that $X^3 = \widetilde{W}^3 Z^2+\Tilde{b}^3$. 
\begin{figure}
\centering
\begin{minipage}{.48\textwidth}
  \centering
  \includegraphics[width=\linewidth]{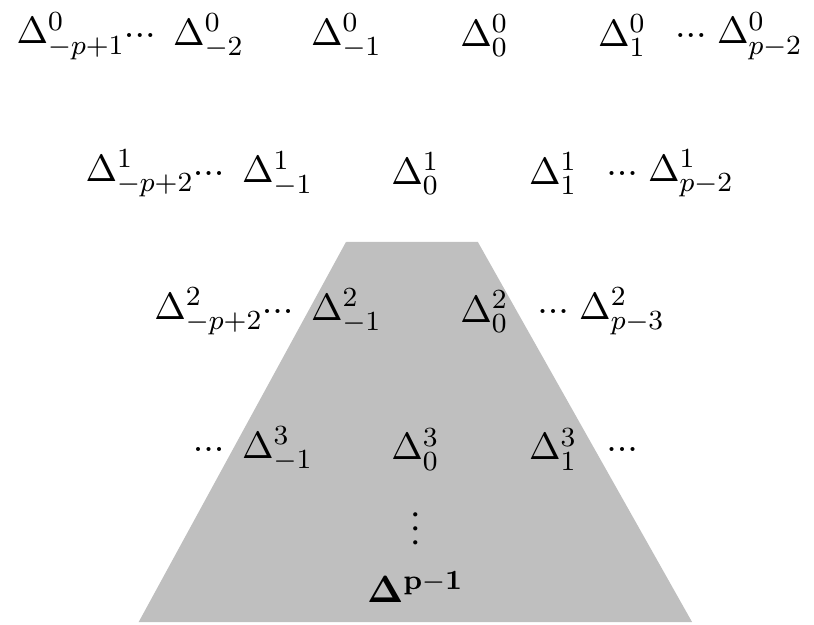}
  \caption{Only undivided differences in the shaded region are compared in Algorithm \ref{alg:eno_int_select}.}
  \label{fig:ndd}
\end{minipage}%
\quad
\begin{minipage}{.48\textwidth}
  \centering
  \includegraphics[width=\linewidth]{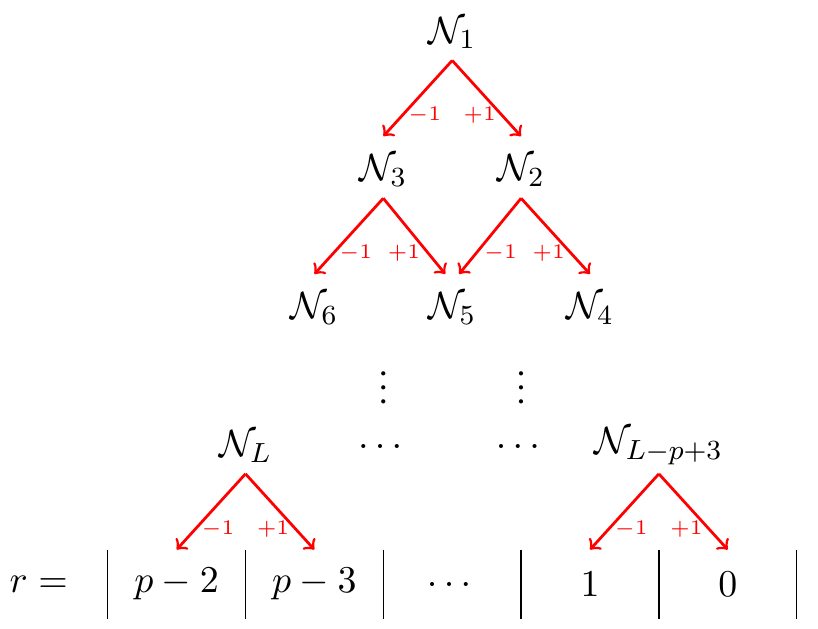}
  \caption{Arrangement of $N_1,\ldots,N_L$ into directed acyclic graph.}
  \label{fig:ndd_pascal}
\end{minipage}
\end{figure}
In order to visualize the next steps, we arrange the elements of $X^3$ in a triangular directed acyclic graph, shown in Figure \ref{fig:ndd_pascal}, where every node $\mathcal{N}_j$ corresponds to the tuple $(X^3_j,X^3_{j+L}) = ( H_1(X^2_j),  H_2(X^2_j))$. 
We note that this tuple is either of the form $( +1,  H_2(X^2_j))$ or $( H_1(X^2_j),  -1)$. Algorithm \ref{alg:eno_int_select} is equivalent to finding a path from the top node to one of the bins on the bottom.  Starting from $\mathcal{N}_1$, we move to the closest element to the right in the row below (i.e. $\mathcal{N}_2$) if $\mathcal{N}_1$ is of the form $( +1,  H_2(X^2_j))$. If $\mathcal{N}_1$ is of the form $( H_1(X^2_j),  -1)$, we move to the closest element to the left in the row below (i.e. $\mathcal{N}_3$). If $\mathcal{N}_1$ is of the form $(+1,-1)$, then it is not important in which direction we move. Both paths lead to a suitable ENO stencil shift. Repeating the same procedure at each row, one ends up in one of the $p-1$ bins at the bottom representing the stencil shift $r$. 

There are $2^{p-2}$ paths from the top to one of the bins at the bottom. In order to represent the path using a $(p-2)$-tuple of entries of $X^3$, one needs to choose between $H_1(X^2_j)$ and $H_2(X^2_j)$ at every node of the path, leading to $2^{p-2}$ variants of each path. At least one of these variants only takes the values $+1$ and $-1$ on the nodes and is identical to the path described above; this is the variant we wish to select. Counting all variants, the total number of paths is $2^{2p-4}$. 

Consider a path $\mathcal{P}=(X^3_{j_1}, \ldots , X^3_{j_{p-2}})$ that leads to bin $r$. We define for this path a weight vector $W \in \{-1,0,1\}^{2L}$ whose elements are set as
\begin{equation*}
    W_j = \begin{cases} +1 &\textrm{if $X^3_j=+1$ and $j=j_s$ for some $1\leq s \leq p-2$} \\ -1 &\textrm{if $X^3_j=-1$ and $j=j_s$ for some $1\leq s \leq p-2$} \\ 0 &\textrm{otherwise.} \end{cases}
\end{equation*}
For this particular weight vector and for any possible $X^3 \in \mathbb{R}^{2L}$ we have ${W \cdot X^3}\leq p-2$, with equality achieved if and only if the entries of $X^3$ appearing in $\mathcal{P}$ are assigned the precise values used to construct $W$. 
One can construct such a weight vector for each of the $2^{2p-4}$ paths. We next construct the weight matrix $\widehat{W}^3\in \mathbb{R}^{2^{2p-4}\times 2L}$ in such a way that the first $2^{p-2}\cdot \binom{p-2}{0}$ rows correspond to the weight vectors for paths reaching $r=0$, the next $2^{p-2}\cdot \binom{p-2}{1}$ for paths reaching $r=1$ et cetera. We also construct the bias vector $\hat{b}^3\in \mathbb{R}^{2^{2p-4}}$ by setting each element to $p-2$ and we define $\hat{X}^3 = \widehat{W}^3 X^3 + \hat{b}^3=\widehat{W}^3 (\widetilde{W}^3 Z^2+\Tilde{b}^3) + \hat{b}^3$. By construction, $\hat{X}^3_j = 2p-4$ if and only if path $j$ corresponds to a suitable ENO stencil shift, otherwise $0\leq \hat{X}^3_j<2p-4$. 

\textbf{Step 5.} Finally we define the final output vector by taking the maximum of all components of $\hat{X}^3$ that correspond to the same bin, 
\begin{equation*}
    \hat{Y}_j = \max\left\{\hat{X}^3\left(1+2^{p-2}\cdot\sum_{k=0}^{j-2} \binom{p-2}{k}\right), \ldots, \hat{X}^3\left(2^{p-2}\cdot\sum_{k=0}^{j-1}  \binom{p-2}{k}\right)\right\},
\end{equation*}
for $j=1,\ldots, p-1$ and where $\hat{X}^3(j):=\hat{X}^3_j$. Note that $\hat{Y}_j$ is the maximum of $2^{p-2}\cdot \binom{p-2}{j-1}$ real positive numbers. Using the observation that $\max\{a,b\}=(a)_++(b-a)_+$ for $a,b\geq 0$, one finds that the calculation of $\hat{Y}$ requires $
p-2+ \left\lceil \log_2 \binom{p-2}{\lfloor \frac{p-2}{2}\rfloor} \right\rceil $ additional hidden layers. 
By construction, it is true that $\hat{Y}_j=2p-4$ if and only if the $(j-1)$-th bin is reached. Furthermore, $\hat{Y}_j<2p-4$ if the $(j-1)$-th bin is not reached. 
The set of all suitable stencil shifts $R$ and the  unique stencil shift $r$ from Algorithm \ref{alg:eno_int_select} are then respectively given by
\begin{equation}\label{eqn:ENO-shift}
    R = \mathrm{argmax}_j \hat{Y}_j - 1 \quad \textrm{and} \quad r = \min R = \min \mathrm{argmax}_j \hat{Y}_j - 1,
\end{equation}
where for classification problems, $\min \mathrm{argmax}$ is the default output function to obtain the class from the network output (see Remark \ref{rem:argmax}).
\end{proof}

\begin{remark}
The neural network constructed in the above theorem is local in the sense that for each cell, it provides a stencil shift. These local neural networks can be concatenated to form a single neural network that takes as its input, the vector $f^k$ of sampled values and returns the vector of interpolated values that approximates $f^{k+1}$. The global neural network combines the output stencil shift of each local neural network with a simple linear mapping \eqref{eqn:eno_int}. 
\end{remark}

Although the previous theorem provides a network architecture for every order $p$, the obtained networks are excessively large for small $p$. We therefore present alternative constructions for ENO interpolation of orders $p=3,4$. 

Algorithm \ref{alg:eno_int_select} for $p=3$ can be exactly represented by the following ReLU network with a single hidden layer, whose input is given by $X = (\Delta_{-2}^0,\Delta^0_{-1}, \Delta^0_0, \Delta^0_1)^\top$. 
The first hidden layer is identical to the one described in the original proof of Theorem \ref{thm:ENO-DNN} for $p=4$, with a null bias vector and $\W^1 \in \Ro^{4 \times 4}$,
\begin{equation}\label{eqn:DeLENO3int-L1}
    \W^1 = \begin{pmatrix}
0 & 1 & -2 & 1 \\
1 & -2& 1 & 0 \\
0 & -1 & 2 & -1 \\
-1 & 2& -1 & 0 \\
\end{pmatrix}, \quad \bb^1 = \begin{pmatrix} 0\\0\\0\\0\end{pmatrix}.
\end{equation}
The weights and biases of the output layer are
\begin{equation}\label{eqn:DeLENO3int-L2}
    W^2 = \begin{pmatrix}
-1 & 1 & -1 & 1\\
1 & -1 & 1 & -1\\
\end{pmatrix}, \quad \bb^2 = \begin{pmatrix} 0\\0\end{pmatrix}.
\end{equation}
The resulting network output is
\begin{equation*}
    \hat{Y} = \begin{pmatrix}\abs{\Delta^2_{-1}}-\abs{\Delta^2_0} \\ \abs{\Delta^2_{0}}-\abs{\Delta^2_{-1}}\end{pmatrix},
\end{equation*}
from which the ENO stencil shift can then be determined using \eqref{eqn:ENO-shift}.

For $p=4$, Algorithm \ref{alg:eno_int_select} can be represented by following ReLU network with 3 hidden layers, whose input is given by $X = (\Delta_{-3}^0,\Delta_{-2}^0,\Delta^0_{-1}, \Delta^0_0, \Delta^0_1,\Delta_{2}^0)^\top$.
The first hidden layer is identical to the one described in the  original proof of Theorem \ref{thm:ENO-DNN} for $p=4$, with a null bias vector and $\W^1 \in \Ro^{10 \times 6}$, 
\begin{equation}\label{eqn:DeLENO4int-L1}
    \W^1 = \begin{pmatrix} \widetilde{W}^1 \\ -\widetilde{W}^1 \end{pmatrix} \in \Ro^{10 \times 6} \quad \textrm{where}\quad \widetilde{W}^1 = \begin{pmatrix} 0 & 0 & 1 & -2 & 1 & 0 \\ 0 & 1 & -2 & 1 & 0 & 0 \\ 0 & 0 & -1 & 3 & -3 & 1 \\ 0 & -1 & 3 & -3 & 1 & 0 \\  -1 & 3 & -3 & 1 & 0 & 0\end{pmatrix}.
\end{equation}
The second hidden layer has a null bias vector and the weight matrix
\begin{equation}\label{eqn:DeLENO4int-L2}
    \W^2 = \begin{pmatrix} \widetilde{W}^2 \\ -\widetilde{W}^2 \end{pmatrix} \in \Ro^{6 \times 10} \quad \textrm{where}\quad \widetilde{W}^2 = \begin{pmatrix} 1&1 \end{pmatrix}\otimes \begin{pmatrix} -1&1&0&0&0\\ 0&0&-1&1&0\\ 0&0&0&-1&1 \end{pmatrix}.
\end{equation}
Note that $\widetilde{W}^2 \in \Ro^{3 \times 10}$ is as in the original proof of Theorem \ref{thm:ENO-DNN} for $p=4$. 
The third hidden layer and the output layer both have a null bias vector and their weights are respectively given by,
\begin{equation}\label{eqn:DeLENO4int-L3}
    \W^3 = \begin{pmatrix}
1 & 1 & 1 & 0 & 0 & 1\\
1 & 0 & 1 & 0 & 1 & 1\\
-1 & 1 & 0 & 1 & 0 & -1\\
0 & 1  & 0 & 1 & 1 & 1
\end{pmatrix} \quad \text{and} \quad  \W^4 = \begin{pmatrix}
1 & 0 & 0 & 0\\
0 & 1 & 1 & 0\\
0 & 0 & 0 & 1
\end{pmatrix}.
\end{equation}
After an elementary, yet tedious case study, one can show that the shift can again be determined using \eqref{eqn:ENO-shift}.

\begin{remark}
Similarly, one can show that the stencil selection algorithm for ENO reconstruction (Algorithm \ref{alg:eno_fd_select} in Appendix \ref{sec:ENO-reconstruction}) for $p=2$ can be exactly represented by a ReLU DNN with one hidden layer of width 4. The input and output dimension are 3 and 2, respectively. For $p=3$, Algorithm \ref{alg:eno_fd_select} can be shown to correspond to a ReLU DNN with three hidden layers of dimensions $(10,6,4)$. Input and output dimension are 5 and 4, respectively. 
\end{remark}

\begin{remark}\label{rem:ENO-regression}
After having successfully recast the ENO stencil selection as a ReLU neural network, it is natural to investigate whether there exists a pure ReLU neural network (i.e. without additional output function) input $f^k$ and output $(\mathcal{I}^{h_k}f)^{k+1}$, as in the setting of \eqref{eqn:eno_int} in Section \ref{sec:ENO}. Since ENO is a discontinuous procedure and a pure ReLU neural network is a continuous function, a network with such an output does not exist. It remains however interesting to investigate to which extent we can approximate ENO using ReLU neural networks. This is the topic of Section 4.4 of the thesis \cite{TDR_THESIS}, where it is shown that there exists a pure ReLU DNN that mimics ENO to some extent such that some of the desirable properties of ENO are preserved. 
\end{remark}

\section{ENO-SR as a ReLU DNN}\label{sec:ENO-SR-DNN}
The goal of this section is to recast the second-order ENO-SR procedure from Section \ref{sec:enosr-alg} as a ReLU DNN, similar to what we did for ENO in Section \ref{sec:ENO-DNN}. Just like ENO, the crucial step of ENO-SR is the stencil selection, allowing us to interpret ENO-SR as a classification problem. In this context, we prove the equivalent of Theorem \ref{thm:ENO-DNN} for ENO-SR-2. Afterwards, we interpret ENO-SR as a regression problem (cfr. Remark \ref{rem:ENO-regression}) and investigate whether we can cast ENO-SR-2 as a pure ReLU DNN, i.e. without additional output function. In the following, we assume $f$ to be a continuous function that is two times differentiable except at a single point $z$ where the first derivative has a jump of size $[f'] = f^{\prime}(z+) - f^{\prime}(z-) $. 

\subsection{ENO-SR-2 stencil selection as ReLU DNN}

We will now prove that a second-order accurate prediction of $f^{k+1}$ can be obtained given $f^k$ using a ReLU DNN, where we use notation as in Section \ref{sec:ENO}. Equation \eqref{eqn:eno_int} shows that we only need to calculate $\mathcal{I}^{h_k}_{i}f(x^{k+1}_{2i-1})$ for every $1\leq i\leq N_k$. The proof we present can be directly generalized to interpolation at points other than the midpoints of the cells, e.g. retrieving cell boundary values for reconstruction purposes. 
From the ENO-SR interpolation procedure it is clear that for every $i$ there exists $r_i^k\in\{-2,0,2\}$ such that $\mathcal{I}^{h_k}_{i}f(x^{k+1}_{2i-1})=p^k_{i+r_i^k}(x^{k+1}_{2i-1})$. 
Analogously to what was described in Section \ref{sec:ENO-DNN}, this gives rise to a classification problem. Instead of considering the stencil shifts as the output classes of the network, one can also treat the different cases that are implicitly described in the ENO-SR interpolation procedure in Section \ref{sec:enosr-alg} as classes. 
This enables us to construct a ReLU neural network such that the stencil shift $r^k_i$ can be obtained from the network output by using the default output function for classification problems (cf. Section \ref{sec:ENO} and Remark \ref{rem:argmax}). 

\begin{theorem}\label{thm:eno-sr-ann}
There exists a ReLU neural network with input $f^k$ that leads to output $(r_1^k, \ldots, r_{N_k}^k)$ as defined above. 
\end{theorem}
\begin{proof}
Instead of explicitly constructing a ReLU DNN, we will prove that we can write the output vector as a composition of functions that can be written as pure ReLU DNNs with linear output functions. Such functions include the rectifier function, absolute value, maximum and the identity function. The network architecture of a possible realisation of the network of this proof can be found after the proof. 
Furthermore we will assume that the discontinuity is not located in the first four or last four intervals. This can be achieved by taking $k$ large enough, or by introducing suitably prescribed ghost values. We also assume without loss of generality that $x_i^k=i$ for $0\leq i\leq N_k$.

The input of the DNN will be the vector $X^0\in \mathbb{R}^{N_k+1}$ with $X^0_{i+1}=f(x^k_i)$ for all $0\leq i \leq N_k$. Using a simple affine transformation, we can obtain $X^1\in\mathbb{R}^{N_k-1}$ such that $X^1_{i}=\Delta^2_{h_k} f(x^k_{i})$ for all $1\leq i \leq N_k-1$. 
We now define the following quantities, 
\begin{equation}\label{eqn:M_i-bis}
        M_i = \max_{n=1,2,3}\abs{\Delta_{h_k}^2f(x^k_{i\pm n})}=\max_{n=1,2,3}\abs{X^1_{i\pm n}},\quad        N_i^{\pm} = \max_{n=1,2}\abs{\Delta_{h_k}^2f(x^k_{i\pm n})}=\max_{n=1,2}\abs{X^1_{i\pm n}},
\end{equation}
where $4\leq i\leq N_k-4$. Next, we construct a vector $X^2\in\mathbb{R}^{N_k}$ such that every entry corresponds to an interval. For $1\leq i\leq N_k$, we want $X^2_i>0$ if and only if the interval $I_{i}^k$ is labelled as $B$ by the adapted ENO-SR detection mechanism. We can achieve this by defining
\begin{equation}\label{eqn:X^2-bis}
\begin{split}
        X^2_{i} = (\min\{\abs{X^1_{i}}- N_i^+,\abs{X^1_{i-1}} - N_{i-1}^-\})_+ +(\abs{X^1_i}-M_i)_+ + (\abs{X^1_{i-1}}-M_{i-1})_+
\end{split}
\end{equation}
for $5\leq i\leq N_k-4$. Furthermore we set $X^2_1 = X^2_2=X^2_3 =X^2_4 = X^2_{N_k-3}= X^2_{N_k-2}= X^2_{N_k-1}= X^2_{N_k}=0$. Note that the first term of the sum will be strictly positive if $I_i^k$ is labelled bad by the second rule of the detection mechanism and one of the other terms will be strictly positive if $I_i^k$ is labelled bad by the first rule. Good intervals $I_i^k$ have $X^2_i=0$. 

Now define $n_{i,l} = l + 4(i-1)$ for $1\leq i \leq N_k$ and $1\leq l \leq 4$. Using this notation, $i$ refers to the interval $I_i^k$. We denote by ${p_i^k:[c,d]\rightarrow \mathbb{R}:x\mapsto a_ix+b_i}$ the linear interpolation of the endpoints of $I_i^k$, where we write $a_i$ and $b_i$ instead of $a_i^k$ and $b_i^k$ to simplify notation. Define $X^3\in \mathbb{R}^{4N_k}$ in the following manner: 
\begin{equation}\label{eqn:X^3-bis}
\begin{split}
        X^3_{n_{i,1}} &= X^2_i, \qquad\qquad\qquad \mathrm{  }X^3_{n_{i,3}} = \left(\abs{b_{i-2}-b_{i+2}}-x^{k+1}_{2i-1}\abs{a_{i-2}-a_{i+2}}\right)_+,\\
        X^3_{n_{i,2}} &= \abs{a_{i-2} - a_{i+2}}, \qquad X^3_{n_{i,4}} = \left(-\abs{b_{i-2}-b_{i+2}}+x^{k+1}_{2i-1}\abs{a_{i-2}-a_{i+2}}\right)_+,
\end{split}
\end{equation}
for $5\leq i \leq N_k-4$. We set $X^3_{n_{i,l}}=0$ for $1 \leq l \leq 4$ and $1\leq i \leq 4$ or $N_k-3\leq i \leq N_k$. 
We can now define the output $\hat{Y}\in\mathbb{R}^{N_k}$ of the ReLU neural network by
\begin{equation}\label{eqn:Y_i-bis}
    \hat{Y}_i = \min \mathrm{argmin}_{1\leq l \leq 4} X^3_{n_{i,l}}. 
\end{equation}
where we used the notation $\hat{Y}_i$ for the predicted class instead of the network output to simplify notation. It remains to prove that $r_i^k$ can be obtained from $\hat{Y}_i$. Note that $\hat{Y}_i=1$ if and only if $I_i^k$ was labelled $G$. Therefore $\hat{Y}_i=1$ corresponds to $r_i^k=0$. If $\hat{Y}_i=2$, then $I_i^k$ was labelled $B$ and the interpolants $p_{i-2}^k$ and $p_{i+2}^k$ do not intersect, leading to $r_i^k=0$ according to the interpolation procedure. Next, $\hat{Y}_i=3,4$ corresponds to the case where $I_i^k$ was labelled $B$ and the interpolants $p_{i-2}^k$ and $p_{i+2}^k$ do intersect. This intersection point is seen to be $y=\frac{b_{i+2}-b_{i-2}}{a_{i-2}-a_{i+2}}$. If $\hat{Y}_i=3$, then $x^{k+1}_{2i-1}$ is right of $y$ and therefore $r_i^k=2$. Analogously, $\hat{Y}_i=4$ corresponds to $r^k_i=-2$, which concludes the proof. 


\end{proof}

Now that we have established that our adaptation of the second-order ENO-SR algorithm can be written as a ReLU DNN augmented with a discontinuous output function, we can present a possible architecture of a DNN that calculates the output $\hat{Y}_i$ from $f^k$. The network we present has five hidden layers, of which the widths vary from 6 to 20, and an output layer of 4 neurons. The network is visualized in Figure \ref{fig:flowchart_adapted}. 
\begin{figure}
    \centering
    \includegraphics[height=0.95\textheight]{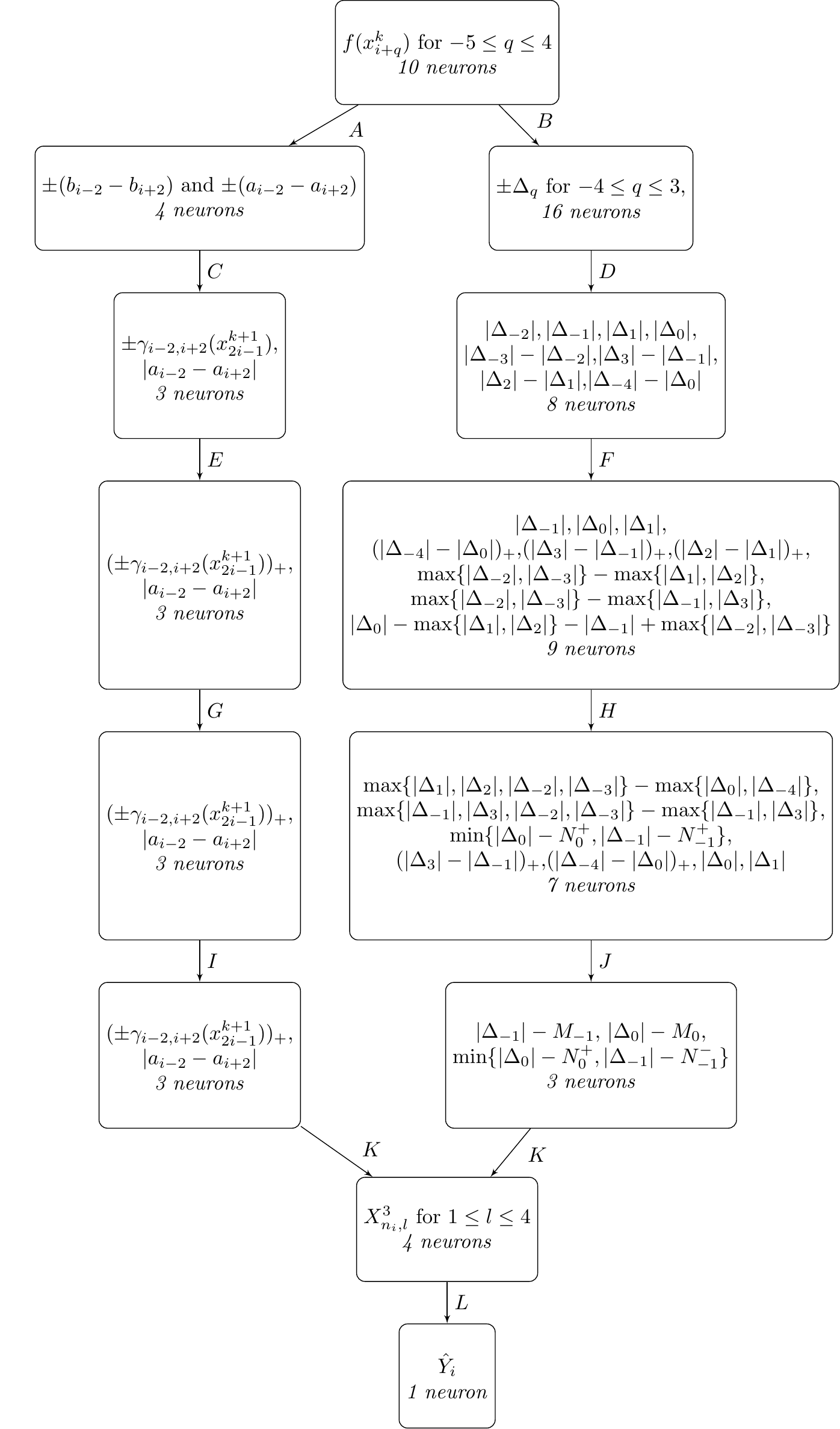}
    \caption{Flowchart of a ReLU DNN to calculate $\hat{Y}_i$ from $f^k$.}
    \label{fig:flowchart_adapted}
\end{figure}
We now give some more explanation about how each layer in Figure \ref{fig:flowchart_adapted} can be calculated from the previous layer, where we use the same notation as in the proof of Theorem \ref{thm:ENO-DNN}. In addition, we define and note that
\begin{subequations}
\begin{align}
    \gamma_{i-2,i+2}(z)&=\abs{b_{i-2}-b_{i+2}}-z\abs{a_{i-2}-a_{i+2}}\label{eqn:gammamn},\\
    \max\{x,y\} &= x + (y-x)_+\label{eqn:max2}.
\end{align}
\end{subequations}
\textbf{\textit{A.B.}} It is easy to see that all quantities of the first layer are linear combinations of the input neurons. 
\textbf{\textit{C.}}  Application of $\abs{x}=(x)_++(-x)_+$ and definition (\ref{eqn:gammamn}) on $\pm(b_{i-2}-b_{i+2})$ and $\pm(a_{i-2}-a_{i+2})$. 
\textbf{\textit{D.}} Straightforward application of the identity $\abs{x}=(x)_++(-x)_+$ on $\pm\Delta_q$, followed by taking linear combinations. 
\textbf{\textit{E.G.I.}} Passing by values. 
\textbf{\textit{F.}} The first six quantities were passed by from the previous layer. The other ones are applications of equation (\ref{eqn:max2}), where the order of the arguments of the maximums is carefully chosen. 
\textbf{\textit{H.}} Equation (\ref{eqn:max2}) was used, where we use that $\min\{x,y\}=-\max\{-x,-y\}$. 
\textbf{\textit{J.}} Application of equation (\ref{eqn:M_i-bis}). 
\textbf{\textit{K.}} The result follows from combining definitions (\ref{eqn:X^2-bis}) and (\ref{eqn:X^3-bis}).
\textbf{\textit{L.}} As can be seen in definition (\ref{eqn:Y_i-bis}), $\hat{Y}_i$ is obtained by applying the output function $\min\mathrm{argmin}$ on the output layer. 

\begin{remark}
The second-order ENO-SR method as proposed in \cite{ACDD2005} can also be written as a ReLU DNN, but it leads to a neural network that is considerably larger than the one presented above. 
\end{remark}

\subsection{ENO-SR-2 regression as ReLU DNN}

After having successfully recast the ENO-SR stencil selection as a ReLU neural network, it is natural to investigate whether there exists a ReLU neural network with output $(\mathcal{I}^{h_k}f)^{k+1}$, as in the setting of \eqref{eqn:eno_int} in Section \ref{sec:ENO}. Since ENO-SR interpolation is a discontinuous procedure and a ReLU neural network is a continuous function, a network with such an output does not exist. It is however interesting to investigate to which extent we can approximate ENO-SR using ReLU neural networks. In what follows, we design an approximate ENO-SR method, based on the adapted ENO-SR-2 method of Section \ref{sec:enosr-alg}, and investigate its accuracy. 

We first introduce for $\epsilon\geq0$ the function $H_{\epsilon}:\mathbb{R}\to\mathbb{R}$, defined by
\begin{equation}
    H_{\epsilon}(x) = \begin{cases} 0 &x\leq 0\\ x/\epsilon &0<x\leq\epsilon \\ 1 &x > \epsilon.\end{cases}
\end{equation}
Note that $H_0$ is nothing more than the Heaviside function. Using this function and the notation of the proof of Theorem \ref{thm:eno-sr-ann}, we can write down a single formula for $\mathcal{I}^{h_k}_{i}f(x^{k+1}_{2i-1})$, 
\begin{equation}\label{eqn:ENOSRa-formula}
\begin{split}
    \mathcal{I}^{h_k}_{i}f(x^{k+1}_{2i-1}) = (1-\alpha)p^k_i(x^{k+1}_{2i-1}) + \alpha\left((1-\beta)p^k_{i+2}(x^{k+1}_{2i-1})+\beta p^k_{i-2}(x^{k+1}_{2i-1})\right), \\
    \textrm{ where } \alpha=H_{0}(\min\{X^3_{n_{i,1}},X^3_{n_{i,2}}\}), \quad \beta = H_{0}(X^3_{n_{i,3}}),
\end{split}
\end{equation}
for $1\leq i \leq N_k$. Observe that this formula cannot be calculated using a pure ReLU DNN. Nevertheless, we will base ourselves on this formula to introduce an approximate ENO-SR algorithm that can be exactly written as a ReLU DNN.  

The first step is to replace $H_0$ by $H_{\epsilon}$ with $\epsilon > 0$, since 
\begin{equation}
    H_{\epsilon}(x) =  \frac{1}{\epsilon}(x)_+-\frac{1}{\epsilon}(x-\epsilon)_+,
\end{equation}
which clearly can be calculated using a ReLU neural network.
The now remaining issue is that the multiplication of two numbers cannot be exactly represented using a ReLU neural network. Moreover, as we aim for a network architecture that is independent of the accuracy of the final network that approximates ENO-SR-2, we cannot use the approximate multiplication networks in the sense of \cite{YAROTSKY17}. We therefore introduce an operation that resembles the multiplication of bounded numbers in another way.  
For $\lambda>0$, we denote by $\star$ the operation on $[0,1]\times[-\lambda,\lambda]$ defined by 
\begin{equation}\label{def:star}
    x \star y := (y+\lambda x-\lambda)_+-(-y+\lambda x-\lambda)_+.
\end{equation}
Like $H_{\epsilon}$, this operation can be cast as a simple ReLU DNN.  
We compare $x\star y$ with $x\cdot y$ for fixed $x\in[0,1]$ and $\lambda>0$ in Figure \ref{fig:star}. Next, we list some properties of $\star$ that are of great importance for the construction of our approximation. 

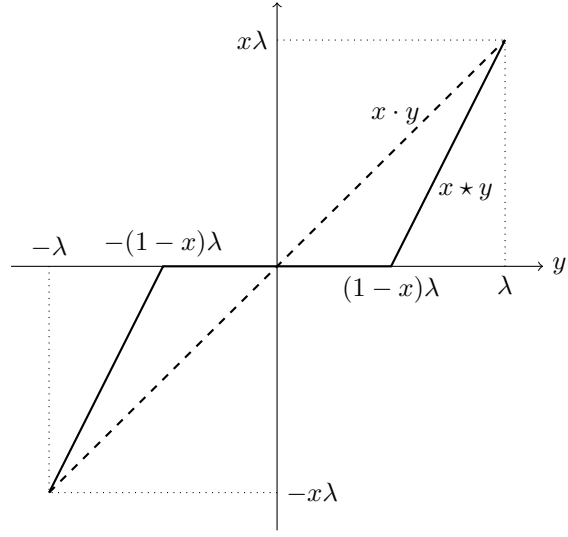
\begin{figure}
    \centering
    \begin{tikzpicture}

\draw[->] (-3.5,0) -- (3.5,0) node[right] {$y$};
\draw[->] (0,-3.5) -- (0,3.5) node[right] {};
\draw[dashed,thick] (-3,-3)--(3,3);
\draw[thick] (-3,-3)--(-1.5,0)--(1.5,0)--(3,3);
\node[left] at (2,2) {$x\cdot y$};
\node[right] at (2,1) {$x\star y$};
\draw[dotted] (0,3)--(3,3)--(3,0);
\draw[dotted] (0,-3)--(-3,-3)--(-3,0);
\node[below] at (1.5,0) {$(1-x)\lambda$};
\node[below] at (3,0) {$\lambda$};
\node[above] at (-1.5,0) {$-(1-x)\lambda$};
\node[above] at (-3,0) {$-\lambda$};
\node[left] at (0,3) {$x\lambda$};
\node[right] at (0,-3) {$-x\lambda$};
\end{tikzpicture}
    \caption{Plot of $x\star y$ and $x\cdot y$ for fixed $x\in[0,1]$ and $\lambda>0$.}
    \label{fig:star}
\end{figure}

\begin{lemma}\label{lem:star}
For $\lambda>0$, the operation $\star$ satisfies the following properties: 
\begin{enumerate}
    \item For all $x\in\{0,1\}$ and $y\in[-\lambda,\lambda]$ it holds true that $x \star y = xy$.
    \item For all $x\in[0,1]$ and $y\in[0,\lambda]$ we have $0\leq x \star y \leq xy$.
    \item For all $x\in[0,1]$ and $y\in[-\lambda,0]$ we have $xy\leq x \star y \leq 0$.
    \item There exist $x\in[0,1]$ and $y_1,y_2\in[-\lambda,\lambda]$ such that \[\min\{y_1,y_2\}\leq (1-x)\star y_1 + x \star y_2\leq \max\{y_1,y_2\}\] does not hold. 
    \item For all $x\in[0,1]$ and $y_1,y_2\in[-\lambda,\lambda]$ it holds true that \[\min\{y_1,y_2\}\leq y_1 + x \star (y_2-y_1)\leq \max\{y_1,y_2\}.\]
\end{enumerate}
\end{lemma}

\begin{proof}
Properties 1,2 and 3 follow immediately from the definition and can also be verified on Figure \ref{fig:star}. For property 4, note that $(1/2)\star (\lambda/2)+(1/2)\star (\lambda/2)=0$. Property 5 is an application of properties 2 and 3. 
\end{proof}

For the moment, we assume that there exists $\lambda>0$ such that all quantities that we will need to multiply, lie in the interval $[-\lambda,\lambda]$. In view of the third property in Lemma \ref{lem:star}, directly replacing all multiplications in \eqref{eqn:ENOSRa-formula} by the operation $\star$ will lead to a quantity that is no longer a convex combination of $p^k_{i-2}(x^{k+1}_{2i-1}), p^k_{i}(x^{k+1}_{2i-1})$ and $p^k_{i+2}(x^{k+1}_{2i-1})$. 
We therefore introduce the \textit{approximate ENO-SR} prediction $\hat{f}_{i,\epsilon}^{k+1}$ of $f(x^{k+1}_i)$ by setting $\hat{f}_{2i,\epsilon}^{k+1}=f_{2i}^{k+1}$ for $0\leq i \leq N_k$ and
\begin{equation}\label{eqn:def-approx-enosr}
\begin{split}
    \hat{f}_{2i-1,\epsilon}^{k+1} = p^k_i(x^{k+1}_{2i-1}) + &\alpha\star \left(p^k_{i+2}(x^{k+1}_{2i-1})- p^k_i(x^{k+1}_{2i-1}) +\beta\star \left( p^k_{i-2}(x^{k+1}_{2i-1})-p^k_{i+2}(x^{k+1}_{2i-1})\right)\right), 
    \\
    \textrm{ where } &\alpha=H_{\epsilon}(\min\{X^3_{n_{i,1}},X^3_{n_{i,2}}\}), \quad \beta = H_{\epsilon}(X^3_{n_{i,3}}),
\end{split}
\end{equation}
for $1\leq i \leq N_k$. The fourth property of Lemma \ref{lem:star} ensures that the two convex combinations in \eqref{eqn:ENOSRa-formula} are replaced by convex combinations (with possibly different weights). 
The theorem below quantifies the accuracy of the approximate ENO-SR predictions for $\epsilon>0$. 

\begin{theorem}\label{thm:approx-enosr-bound}
Let $f:[c,d]\to[-1,1]$ be a globally continuous function with a bounded second derivative on $\mathbb{R}\backslash \{z\}$ and a discontinuity in the first derivative at a point $z$. For every $k$, the approximate ENO-SR predictions $\hat{f}^{k+1}_{i,\epsilon}$ satisfy for every $0\leq i\leq N_{k+1}$ and $\epsilon\geq 0$ that
\begin{equation}\label{eqn:approx-enosra}
    \abs{\mathcal{I}^{h_k}_{i}f(x^{k+1}_{2i-1})-\hat{f}^{k+1}_{i,\epsilon}} \leq Ch_k^2 \sup_{[c,d]\backslash \{z\}}\abs{f''} + \frac{3}{2}\epsilon,
\end{equation}
where $\mathcal{I}^{h_k}_{i}f(x^{k+1}_{2i-1})$ is the ENO-SR-2 prediction.
\end{theorem}
\begin{proof}
The proof can be found in Appendix \ref{app:acc-approx-enosr}.
\end{proof}

We see that our approximation is second-order accurate up to an additional constant error, which can be made arbitrarily small. Finally, the following theorem states that the constructed approximation can indeed be represented by a ReLU DNN, i.e. there exists a pure ReLU DNN that satisfies bound \eqref{eqn:approx-enosra}. 

\begin{theorem}\label{thm:approx-enosr-dnn}
Let $f:[c,d]\to[-1,1]$ be a globally continuous function with a bounded second derivative on $\mathbb{R}\backslash \{z\}$ and a discontinuity in the first derivative at a point $z$. For every $\epsilon>0$, there exists a pure ReLU neural network with input $f^k\in[-1,1]^{N_k+1}$ and output $\hat{f}^{k+1}_{2i-1,\epsilon}$ for every $1\leq i \leq N_k$, $0\leq k \leq K$.
\end{theorem}
\begin{proof}
Most of the work was already done in Theorem \ref{thm:eno-sr-ann} and the discussion preceding Theorem \ref{thm:approx-enosr-bound}. Indeed, we have already established that $p^k_{i-2}(x^{k+1}_{2i-1}), p^k_{i}(x^{k+1}_{2i-1})$, $p^k_{i+2}(x^{k+1}_{2i-1}),X^3_1,X^3_2$ and $X^3_3$, as well as the operation $\star$ can be represented using pure ReLU networks. It only remains to find a bound for all second arguments of the operation $\star$ in  \eqref{eqn:def-approx-enosr}. Since the codomain of $f$ is $[-1,1]$, one can calculate that $p^k_{i-2}(x^{k+1}_{2i-1}), p^k_{i}(x^{k+1}_{2i-1})$ and $p^k_{i+2}(x^{k+1}_{2i-1})$ lie in $[-4,4]$. Using Lemma \ref{lem:star}, we then find that 
\[p^k_{i+2}(x^{k+1}_{2i-1})- p^k_i(x^{k+1}_{2i-1}) +\beta\star \left( p^k_{i-2}(x^{k+1}_{2i-1})-p^k_{i+2}(x^{k+1}_{2i-1})\right)\in[-16,16]\]
for all $\beta\in[0,1]$. 
We can thus use the operation $\star$ with $\lambda = 16$ in  \eqref{def:star}. 
\end{proof}

We now present the network architecture of a ReLU neural network that computes the approximate ENO-SR prediction \eqref{eqn:def-approx-enosr}. The network we propose is visualized in Figure \ref{fig:flowchart_total} and consists of eight hidden layers with widths 23, 13, 14, 12, 8, 7, 6 and 3. 
\begin{figure}
    \centering
    \includegraphics[scale=0.9]{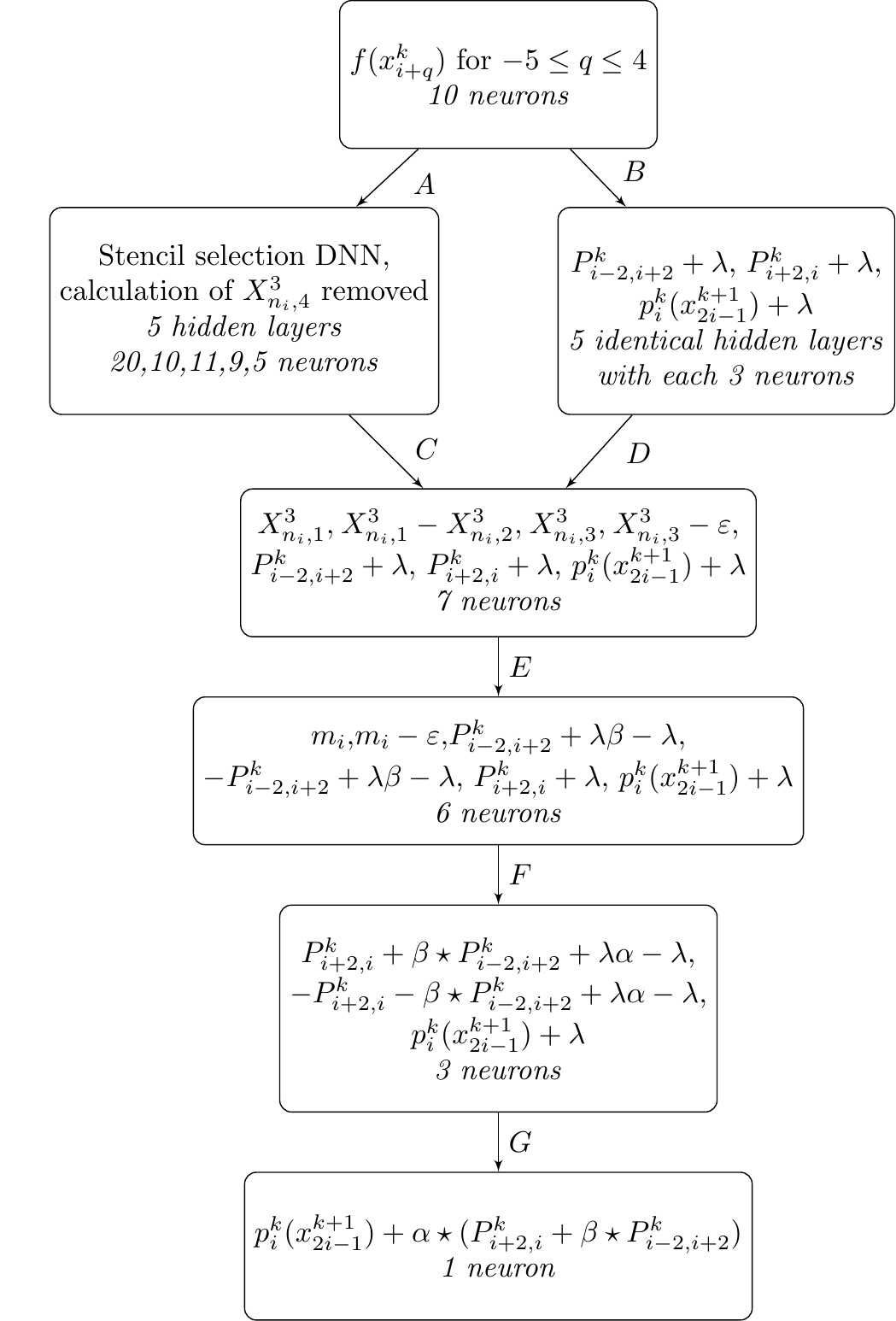}
    \caption{Flowchart of a ReLU DNN to calculate $\hat{f}^{k+1}_{2i-1,\epsilon}$ from $f^k$. }
    \label{fig:flowchart_total}
\end{figure}
In the figure, the following notation was used, 
\begin{equation}\label{eqn:flowchart_total}
\begin{split}
    m_i &= \min\{X^3_{n_i,1},X^3_{n_i,2}\},\\
    P^k_{m,n} &= p^k_{m}(x^{k+1}_{2i-1})- p^k_n(x^{k+1}_{2i-1}),
\end{split}
\end{equation}
for $1\leq i,m,n \leq N_k$. We now give some more explanation about how all the layers can be calculated from the previous layer in Figure \ref{fig:flowchart_total}. \textbf{\textit{A.B.}} All quantities of the first layer are linear combinations of the input neurons, where we also refer to Figure \ref{fig:flowchart_adapted}. From the proof of Theorem \ref{thm:approx-enosr-dnn}, it follows that we can take $\lambda=16$. \textbf{\textit{C.D.}} Linear combinations. \textbf{\textit{E.}} We refer to \eqref{eqn:def-approx-enosr} and \eqref{eqn:flowchart_total} for the definitions of $\beta$ and $m_i$, respectively. \textbf{\textit{F.}} We refer to \eqref{def:star} and \eqref{eqn:def-approx-enosr} for the definitions of $\star$ and $\alpha$, respectively. \textbf{\textit{G.}} From \eqref{eqn:def-approx-enosr} it follows that the value of the output layer is indeed equal to the approximate second-order ENO-SR prediction $\hat{f}^{k+1}_{2i-1,\epsilon}$. 

\section{Numerical results}\label{sec:num-results}
From sections \ref{sec:ENO-DNN} and \ref{sec:ENO-SR-DNN}, we know that there exist deep ReLU neural networks, of a specific architecture, that will mimic the ENO-$p$ and the second-order ENO-SR-2 algorithms for interpolating rough functions. In this section, we investigate whether we can \textit{train} such networks in practice and we also investigate their performance on a variety of tasks for which the ENO procedure is heavily used. We will refer to these trained networks as \textit{DeLENO} (Deep Learning ENO) and \textit{DeLENO-SR} networks. More details on the training procedure can be found in Section \ref{sec:training}, the performance is discussed in Section \ref{sec:performance} and illustrated by various applications at the end of this section. 

\subsection{Training}\label{sec:training}
The \textit{training} of these networks involves finding a parameter vector $\theta$ (the weights and biases of the network) that approximately minimizes a certain \textit{loss function}  $\mathcal{J}$ which measures the error in the network's predictions. To achieve this, we have access to a finite data set $\dset = \{ (X ^i,\func(X^i))\}_i \subset D \times \func(D)$, where $\func:D\subset\mathbb{R}^m\to\mathbb{R}^n$ is the unknown function we try to approximate using a neural network $\func^\theta$. 

For classification problems, each $Y^i=\func(X^i)$ is an $n$-tuple that indicates to which of the $n$ classes $X^i$ belongs. The output of the network $\hat{Y}^i=\func^\theta(X^i)$ is an approximation of $Y^i$ in the sense that $\hat{Y}^i_j$ can be interpreted as the probability that $X^i$ belongs to class $j$. A suitable loss function in this setting is the \textit{cross-entropy function} with regularization term
\begin{equation}\label{eqn:reg_cross_entropy}
\mathcal{J}(\theta ; \dset,\lambda) = - \frac{1}{\# \dset}\sum \limits_{(X^i,Y^i) \in \dset} \ \sum \limits_{j=1}^{n} Y^i_j \log(\hat{Y}^i_j)+ \lambda \mathcal{R}(\theta). 
\end{equation}
The cross-entropy term measures the discrepancy between the probability distributions of the true outputs and the predictions. It is common to add a regularization term $\lambda \mathcal{R}(\theta)$ to prevent overfitting of the data and thus improve the generalization capabilities of the network \cite{DLBOOK}. The network hyperparameter $\lambda > 0$ controls the extent of regularization. Popular choices of $\mathcal{R}(\theta)$ include the sum of some norm of all the weights of the network. To monitor the generalization capability of the network, it is useful to split $\dset$ into a training set $\mathbb{T}$ and a validation set $\Vset$ and minimize $\mathcal{J}(\theta ; \mathbb{T},\lambda)$ instead of $\mathcal{J}(\theta ; \dset,\lambda)$. The validation set $\Vset$ is used to evaluate the generalization error. 
The accuracy of network $\funcnn$ on $\Tset$ is measured as 
\begin{equation}\label{eqn:accuracy}
\Tset_{acc} = 
\#\left\{(X,Y) \in \Tset \ | \ \hat{Y} = \funcnn(X), \ \ \text{arg} \max \limits_{1\leq j\leq n} \hat{Y}_j= \text{arg} \max \limits_{1\leq j\leq n} Y_j \right\}/ \# \Tset,
\end{equation}
with a similar expression for $\Vset_{acc}$. 

For regression problems, $\hat{Y}^i$ is a direct approximation of $Y^i$, making the \textit{mean squared error} with regularization term
\begin{equation}\label{eqn:mse-def}
    \mathcal{J}(\theta ; \dset,\lambda) =  \frac{1}{\# \dset}\sum \limits_{(X^i,Y^i) \in \dset} \norm{Y^i-\hat{Y}^i}^2+ \lambda \mathcal{R}(\theta), 
\end{equation}
an appropriate loss function. As before, the data set $\dset$ can be split into a training set $\Tset$ and a validation set $\Vset$, in order to minimize $ \mathcal{J}(\theta ; \Tset,\lambda)$ and estimate the MSE of the trained network by $ \mathcal{J}(\theta ; \Vset,\lambda)$. 

The loss functions are minimized either with a mini-batch version of the stochastic gradient descent algorithm with an adpative learning rate or with the popular ADAM optimizer \cite{adam}. We use batch sizes of $1024$, unless otherwise specified. 

 \subsubsection{Training DeLENO-p}\label{sec:train-deleno}

We want to construct a suitable training data set $\mathbb{S}$ to train DeLENO-$p$ for interpolation purposes. Thanks to the results of Section \ref{sec:ENO-DNN}, we are guaranteed that for certain architectures it is theoretically possible to achieve an accuracy of 100\%. For any order, this architecture is given by Theorem \ref{thm:ENO-DNN} and its proof. For small orders $\pdeg=3,4$ we use the alternative network architectures described at the end of Section \ref{sec:ENO-DNN}, as they are of smaller size. The network will take an input from $\mathbb{R}^m$, $m=2\pdeg-2$, and predicts the stencil shift $r$. We generate a data set $\dset$ of size 460,200-200$m$ using Algorithm \ref{alg:eno_int_select} with inputs given by,
\begin{itemize}
\item A total of 400,000 samples $X \in \Ro^m$, with each component $X_j$ randomly drawn from the uniform distribution on the interval $[-1,1]$. 
\item The set
\[
\{ (u_l, ..., u_{l+m})^\top | \ 0\leq l \leq N-m, \ \ 0\leq q \leq 39, \ \ N = \{100,200,300,400,500\} \}
\]
where $u_l$ is defined as
\[
u_l:=\sin\left((q+1) \pi \frac{l}{N} \right), \quad 0\leq l \leq N.
\]
\end{itemize}

The input data needs to be appropriately scaled before being fed into the network, to ensure faster convergence during training. We use the following scaling for each input $X$,
\begin{equation}\label{eqn:scale}
\text{Scale}(X) = \begin{cases} \frac{2X - (b + a)}{b-a}& \quad \text{if } X \neq 0 \\
                 				   (1,...,1)^\top \in \Ro^m & \quad \text{otherwise}
                                    \end{cases}, \quad a = \min_j (X_j) , \:  b = \max_j (X_j),
\end{equation}
which scales the input to lie in the box $[-1,1]^m$.
\begin{remark}
When the input data is scaled using formula \eqref{eqn:scale}, then Newton's undivided differences are scaled by a factor $2(b-a)^{-1}$ as well. Therefore scaling does not alter the stencil shift obtained using Algorithm \ref{alg:eno_int_select} or \ref{alg:eno_fd_select}.
\end{remark}

The loss function $\mathcal{J}$ is chosen as \eqref{eqn:reg_cross_entropy}, with an $L_2$ penalization of the network weights and $\lambda = 7.8\cdot10^{-6}$. 
The network is retrained using 5 times, with the weights and biases initialized using a random normal distribution for each of the retrainings. The last $20 \%$ of $\dset$ is set aside to be used as the validation set $\Vset$. For each $p$, we denote by DeLENO-$p$ the network with the highest accuracy $\Vset_{acc}$ at the end of the training. 
The training of the DeLENO reconstruction networks was performed entirely analogously, with the only difference that now we set $m=2p-1$.

\subsubsection{Training DeLENO-SR}\label{sec:train-delenosr}

Next, we construct a suitable training data set $\mathbb{S}$ to train second-order DeLENO-SR for use as an interpolation algorithm. Recall that ENO-SR is designed to interpolate continuous functions $f$ that are two times differentiable, except at isolated points, $z\in\mathbb{R}$ where the first derivative has a jump of size $[f']$. Locally, these functions can be viewed as piecewise linear functions. Based on this observation, we create a data set using functions of the form 
\begin{equation}
    f(x)=a(x-z)_-+b(x-z)_+,
\end{equation}
where $a,b,z\in\mathbb{R}$. For notational simplicity we assume that the $x$-values of the stencil that serves as input for the ENO-SR algorithm (Section \ref{sec:enosr-alg}) are $0,1,\ldots,9$. The interval of interest is then $[4,5]$ and the goal of ENO-SR is to find an approximation of $f$ at $x=4.5$. We generate 100,000 samples, where we choose $a,b,z$ in the following manner, 
\begin{itemize}
    \item The parameters $a$ and $b$ are drawn from the uniform distribution on the interval $[-1,1]$. Note that any interval that is symmetric around 0 could have been used, since the data will be scaled afterwards. 
    \item For 25,000 samples, $z$ is drawn from the uniform distribution on the interval $[4,5]$. This simulates the case where the discontinuity is inside the interval of interest. 
    \item For 75,000 samples, $z$ is drawn from the uniform distribution on the interval $[-9,9]$, which also includes the case in which $f$ is smooth on the stencil. 
\end{itemize}

The network architecture is described in Section \ref{sec:ENO-SR-DNN}. The network will take an input from $\mathbb{R}^{10}$ and predicts the stencil shift $r$. The training of DeLENO-SR was performed in a very similar fashion to the training of DeLENO-$p$ (Section \ref{sec:train-deleno}), only this time we retrained the DeLENO-SR network 5 times for 5000 epochs each. Furthermore we used 8-fold cross-validation on a data set of 20,000 samples to select the optimal regularization parameter, resulting in the choice $\lambda=1\cdot10^{-8}$. 

\begin{remark}
Note that the detection mechanism of the ENO-SR interpolation method (Section \ref{sec:enosr-alg}) labels an interval as bad when $\alpha-\beta>0$ for some numbers $\alpha,\beta\in\mathbb{R}$. This approach causes poor approximations in practice due to numerical errors. When for example $\alpha=\beta$, rounding can have as a consequence that $\mathrm{round}(\alpha-\beta)>0$, leading to an incorrect label. This deteriorates the accuracy of the method and is very problematic for the training. Therefore we used in our code the alternative detection criterion $\alpha-\beta>\epsilon$, where for example $\epsilon=10^{-10}$. 
\end{remark}

\subsection{Performance}\label{sec:performance}

In the previous sections, we have proven the existence of ReLU neural networks that approximate ENO(-SR) well, or can even exactly reproduce its output. However, it might be challenging to obtain  these networks by training on a finite set of samples.  Fortunately, Table \ref{tab:eno_param} demonstrates that this is not the case for the DeLENO(-SR) stencil selection networks. For both interpolation and reconstruction, the classification accuracy \eqref{eqn:accuracy} is nearly 100$\%$. A comparison between the trained weights and biases in Appendix \ref{sec:trained-weights} and their theoretical counterparts of (\ref{eqn:DeLENO3int-L1}-\ref{eqn:DeLENO4int-L3}) reveals that there are multiple DNNs that can represent ENO. Moreover, this indicates that the weights of two DNNs (i.e. the theoretical and trained DNNs) can be very different even though the output is approximately the same (Table \ref{tab:eno_param}). This is in agreement with the result from \cite{petersen2018topological} that the function that maps a family of weights to the function computed by the associated network is not inverse stable. 

\begin{table}[!htpb]
\centering
\subtable[DeLENO interpolation]{
\begin{tabular}{|c|c|c|c|}
\hline
$\pdeg$ & \multicolumn{1}{c|}{hidden layer sizes} &  \multicolumn{1}{c|}{$\Tset_{acc}$} & \multicolumn{1}{c|}{$\Vset_{acc}$} \\ \hline
3   & 4                                             &     $99.36 \%$       &      $99.32 \%$                    \\ \hline
4   & 10,6,4                                   &     $99.22 \%$       &      $99.14 \%$                     \\ \hline
SR & 20,11,12,10,6 & $99.74\%$ & $99.81\%$ \\\hline
\end{tabular}
\label{tab:eno_int_param}
}
\subtable[DeLENO reconstruction]{
\begin{tabular}{|c|c|c|c|}
\hline
$\pdeg$ & \multicolumn{1}{c|}{hidden layer sizes}  & \multicolumn{1}{c|}{$\Tset_{acc}$} & \multicolumn{1}{c|}{$\Vset_{acc}$} \\ \hline
2   & 4                                  &     $99.96 \%$       &      $99.97 \%$                    \\ \hline
3   & 10, 6, 4                                   &     $99.65 \%$       &      $99.65 \%$                     \\ \hline
\end{tabular}
\label{tab:eno_rec_param}
}
\caption{Shape of DeLENO-$\pdeg$ and DeLENO-SR networks with their accuracies for the interpolation and reconstruction problem.}\label{tab:eno_param}
\end{table}

Next, we investigate the order of accuracy of the DeLENO-SR regression network, for the functions
\begin{align}\label{eq:function_ordercomparison}
    f_1(x) &= -2\left(x-\frac{\pi}{6}\right)\mathbbm{1}_{[0,\frac{\pi}{6})}(x)+\left(x-\frac{\pi}{6}\right)^2\mathbbm{1}_{[\frac{\pi}{6},1]}(x),\\
    f_2(x) &= \sin(x). 
\end{align}
Note that the first derivative of $f_1$ has a jump at $\frac{\pi}{6}$. In Figure \ref{fig:error-enosr-reg}, the order of accuracy of second-order ENO-SR-2 and DeLENO-SR-2 is compared with those of ENO-3 and DeLENO-3 for both the piecewise smooth function $f_1$ and the smooth function $f_2$. 
\begin{figure}[!htbp]
  \centering
  \subfigure[Error for $f_1$]{\includegraphics[width=0.49\textwidth]{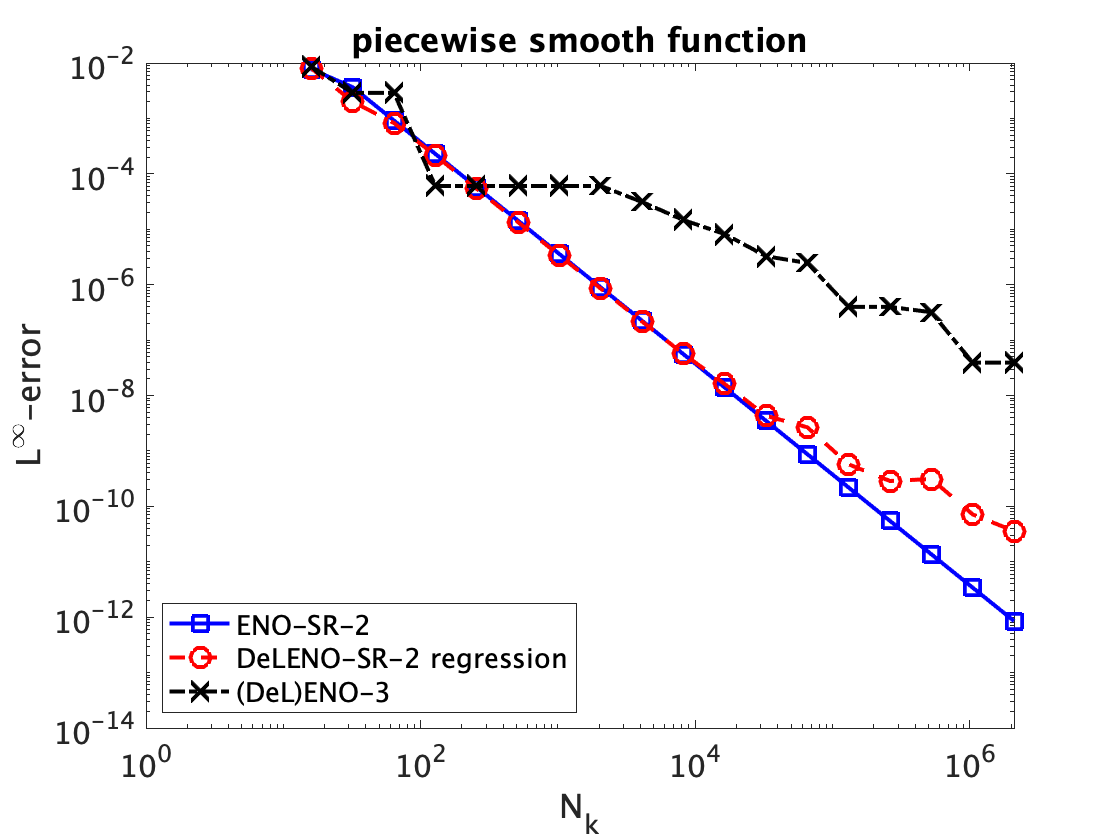}}
  \subfigure[Error for $f_2$]{\includegraphics[width=0.49\textwidth]{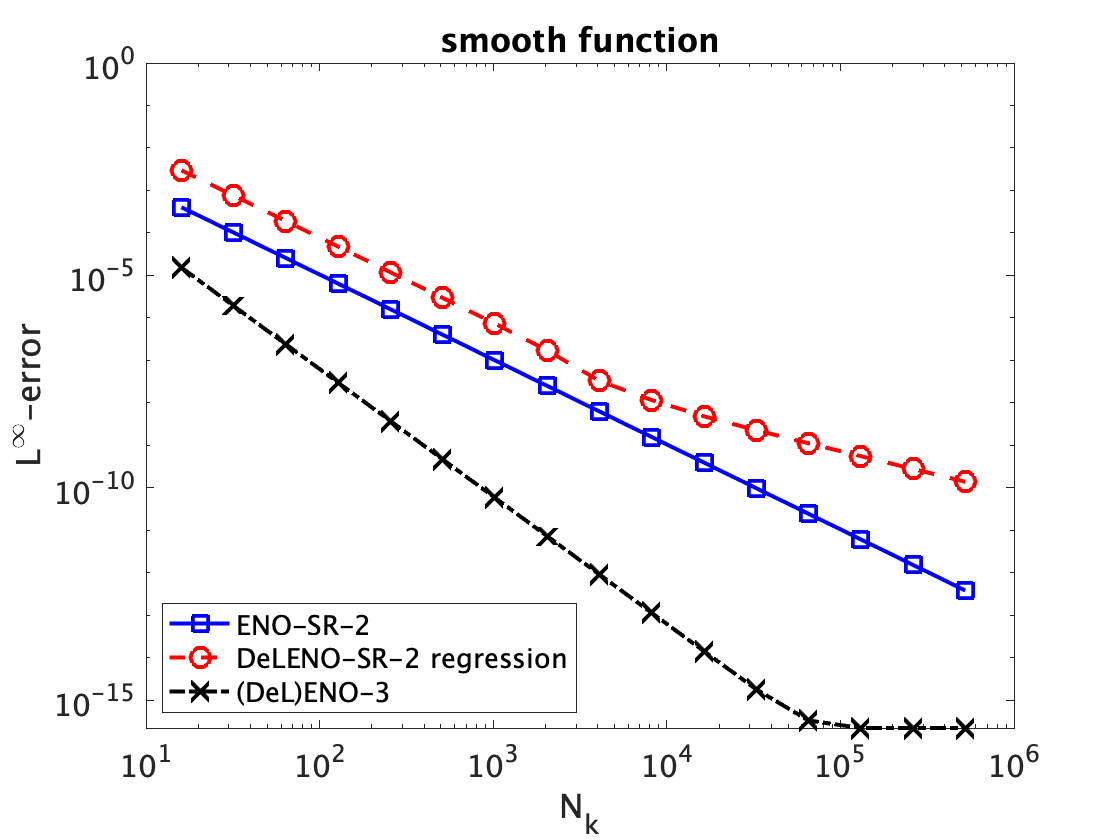}}
  \caption{Plots of the approximation error of the DeLENO-SR-2 regression network for the piecewise smooth function $f_1$ and the smooth sine function $f_2$. The approximation errors of ENO-SR-2 and (DeL)ENO-3 are shown for comparison. }\label{fig:error-enosr-reg}
\end{figure}
In both cases, ENO-3 and DeLENO-3 completely agree, which is not surprising given the high classification accuracies listed in Table \ref{tab:eno_param}. (DeL)ENO-3 is third-order accurate for the smooth function and only first-order accurate for a more rough function, in agreement with expectations. 
For both $f_1$ and $f_2$, DeLENO-SR-2 is second-order accurate on coarse grids, but a deterioration to first-order accuracy is seen on very fine grids. This deterioration is an unavoidable consequence of the error that the trained network makes and the linear rescaling of the input stencils \eqref{eqn:scale}. A more detailed discussion of this issue can be found in Section 6.2 of \cite{TDR_THESIS}. 
Furthermore, although the DeLENO-SR regression network is initially second-order accurate for the smooth function, the approximation error does not agree with that of ENO-SR. This is in line with Theorem \ref{thm:approx-enosr-bound}, where we proved that there exists a network that is a second-order accurate approximation of ENO-SR-2 except for an error term that can be arbitrarily small, yet is fixed. A second factor that might contribute is the fact that DeLENO-SR-2 was trained on piecewise linear functions, which can be thought of as a second-order accurate approximation of a smooth function, therefore leading to a higher error. 

\subsection{Applications}
Next, we apply the DeLENO algorithms in the following examples.
\subsubsection{Function approximation}
We first demonstrate the approximating ability of the DeLENO interpolation method using the function
\begin{equation}\label{eqn:f_int}
q(x) = \begin{cases}
          -x & \quad \text{if } x < 0.5,\\
          3\sin(10 \pi x) & \quad \text{if } 0.5 < x < 1.5,\\
          - 20(x-2)^2 & \quad \text{if } 1.5 < x < 2.5,\\
          3 & \quad \text{if } 2.5<x,\\
	\end{cases}
\end{equation}
which consists of jump discontinuities and smooth high-frequency oscillations. We discretize the domain $[0,3]$ and generate a sequence of nested grids of the form \eqref{eqn:interp_partitions} by setting $N_0 = 16$ and $K=4$. We use the data on the grid $\Tau^k$, and interpolate it onto the grid $\Tau^{k+1}$ for $0\leq k < K$. As shown in Figure \ref{fig:fint_p34}, the interpolation with ENO-$4$ and DeLENO-$4$ is identical on all grids, for this particular function.
\begin{figure}[!htbp]
\begin{center}
\subfigure[$\Tau^0$ to $\Tau^1$, $\pdeg=3$]{\includegraphics[width=0.33\textwidth]{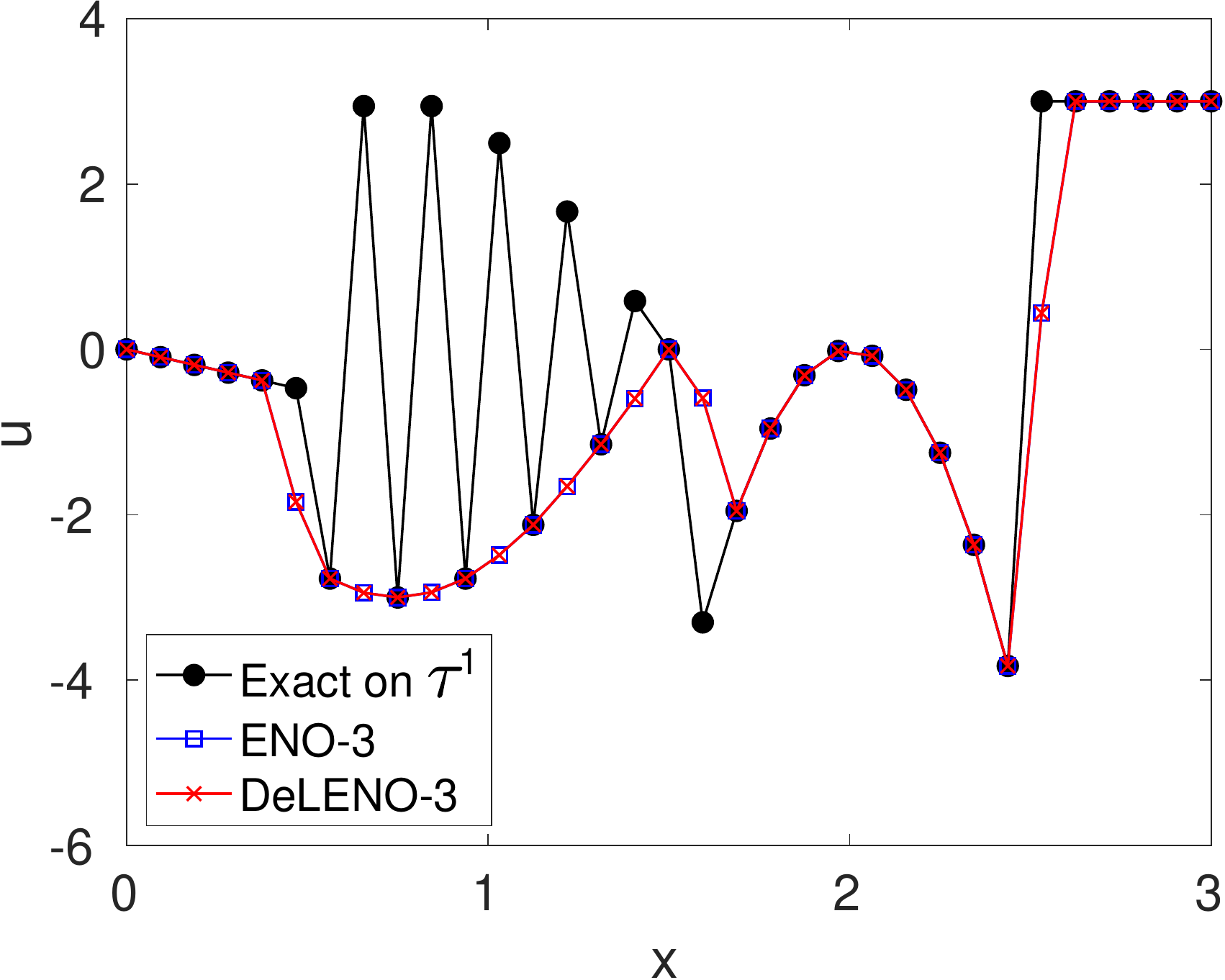}}
\subfigure[$\Tau^0$ to $\Tau^1$, $\pdeg=4$]{\includegraphics[width=0.33\textwidth]{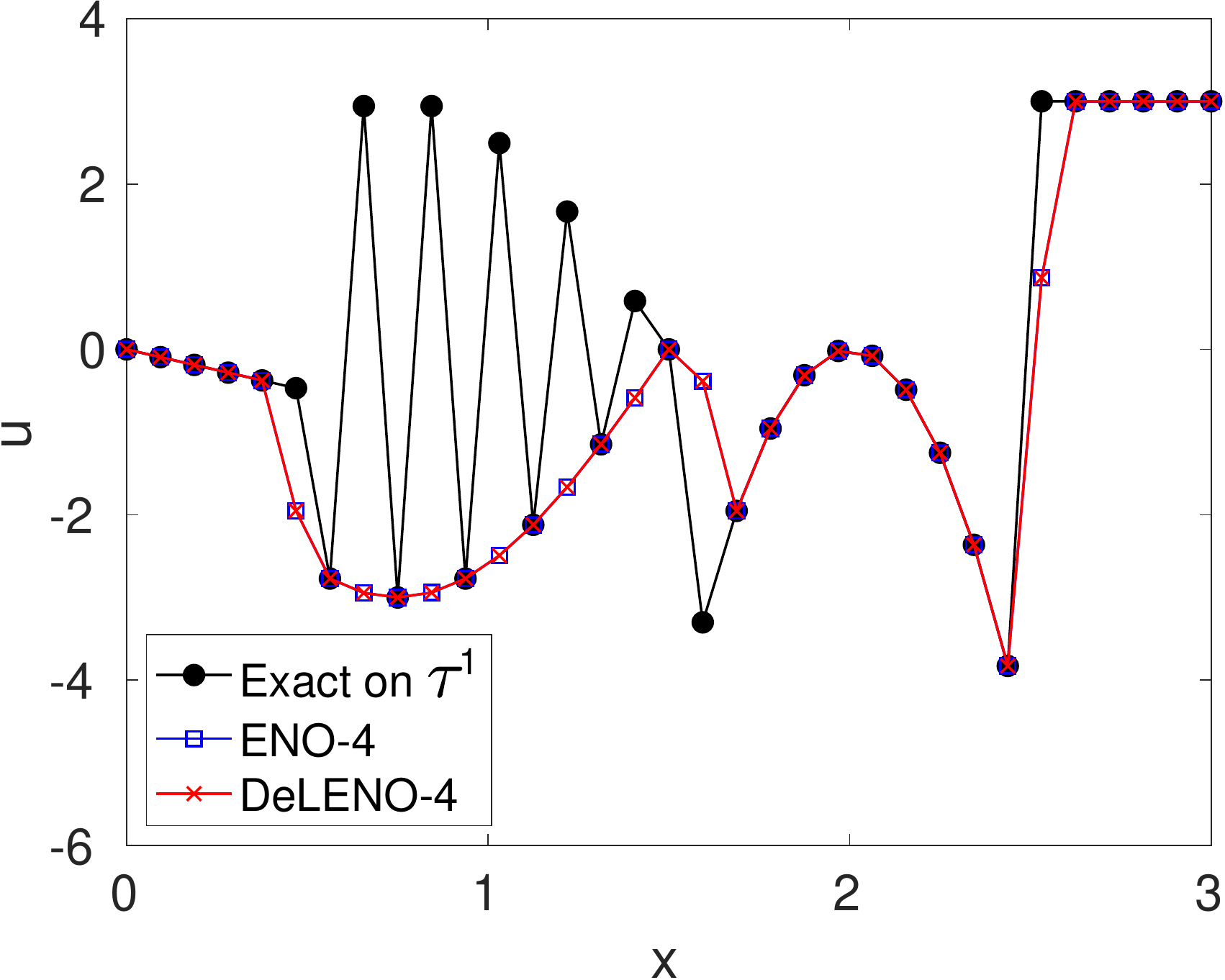}}\\
\subfigure[$\Tau^1$ to $\Tau^2$, $\pdeg=3$]{\includegraphics[width=0.33\textwidth]{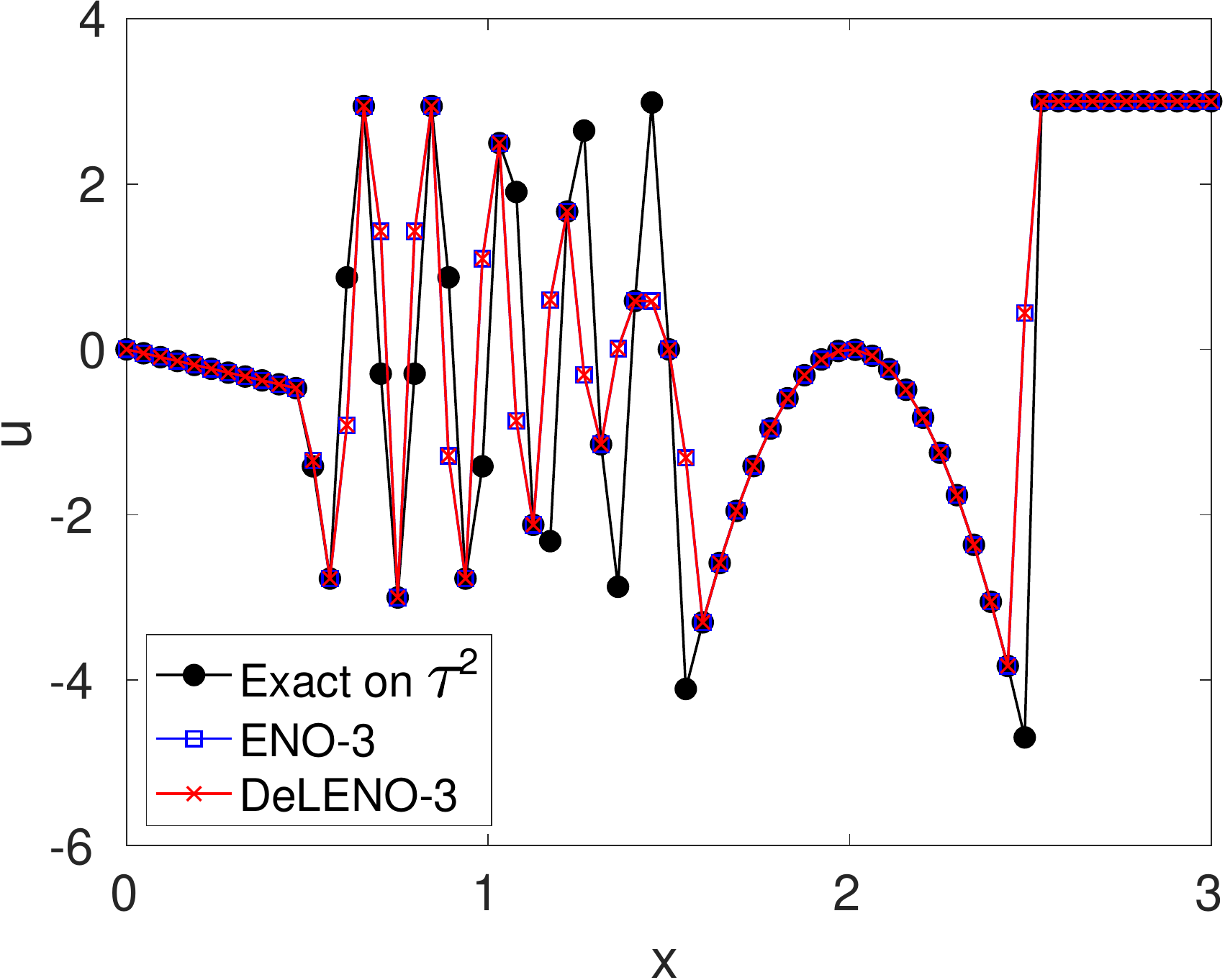}}
\subfigure[$\Tau^1$ to $\Tau^2$, $\pdeg=4$]{\includegraphics[width=0.33\textwidth]{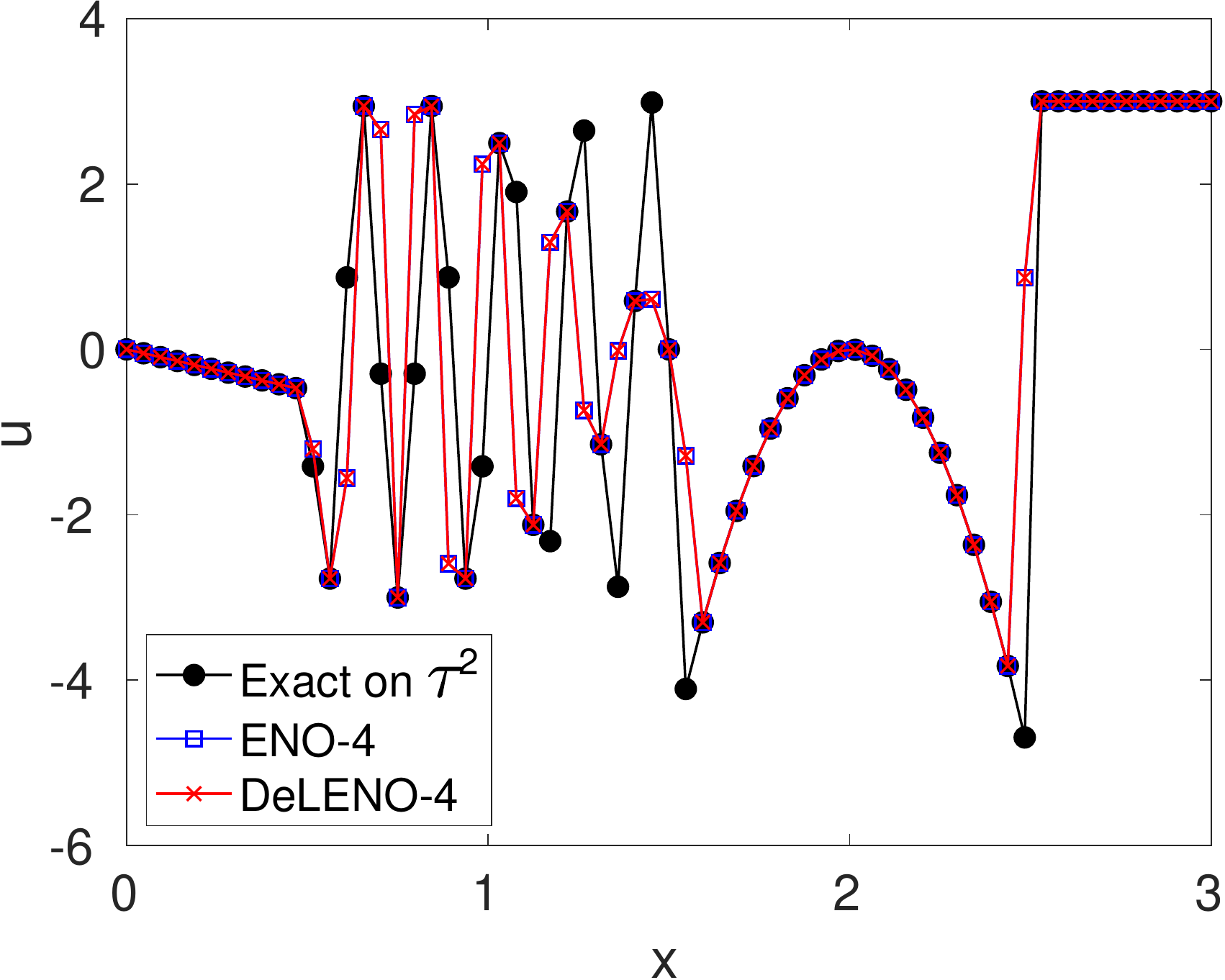}}\\
\subfigure[$\Tau^2$ to $\Tau^3$, $\pdeg=3$]{\includegraphics[width=0.33\textwidth]{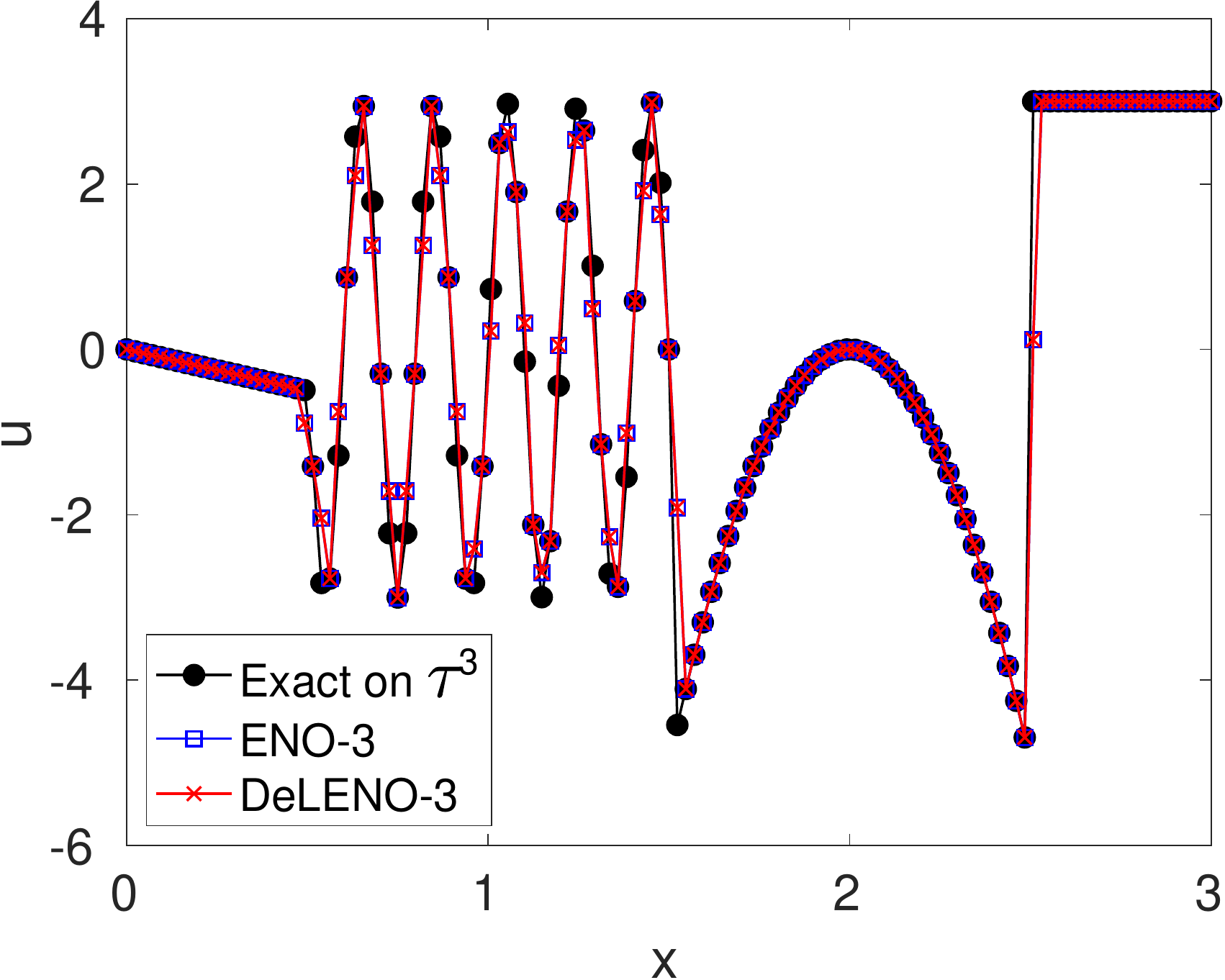}}
\subfigure[$\Tau^2$ to $\Tau^3$, $\pdeg=4$]{\includegraphics[width=0.33\textwidth]{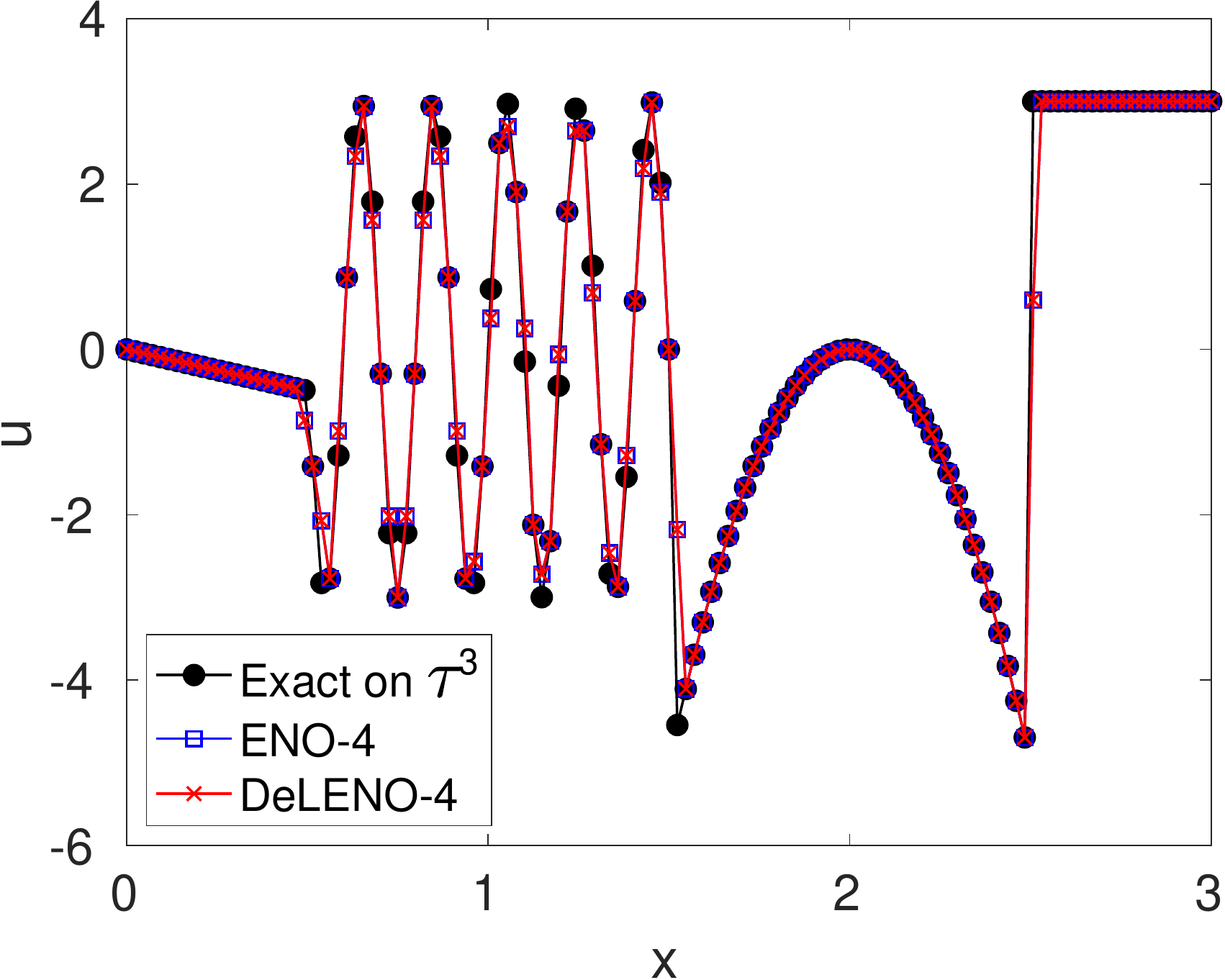}}\\
\subfigure[$\Tau^3$ to $\Tau^4$, $\pdeg=3$]{\includegraphics[width=0.33\textwidth]{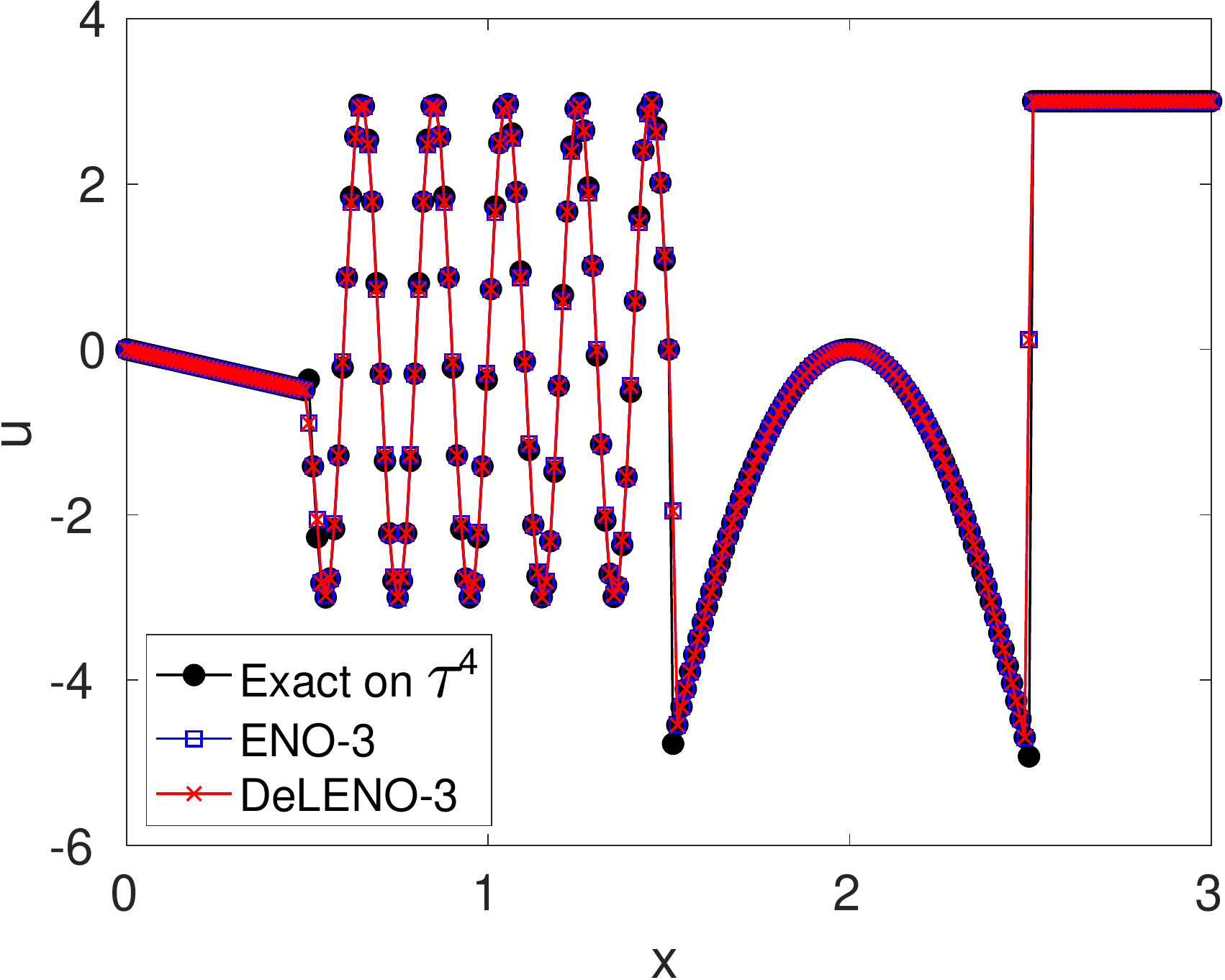}}
\subfigure[$\Tau^3$ to $\Tau^4$, $\pdeg=4$]{\includegraphics[width=0.33\textwidth]{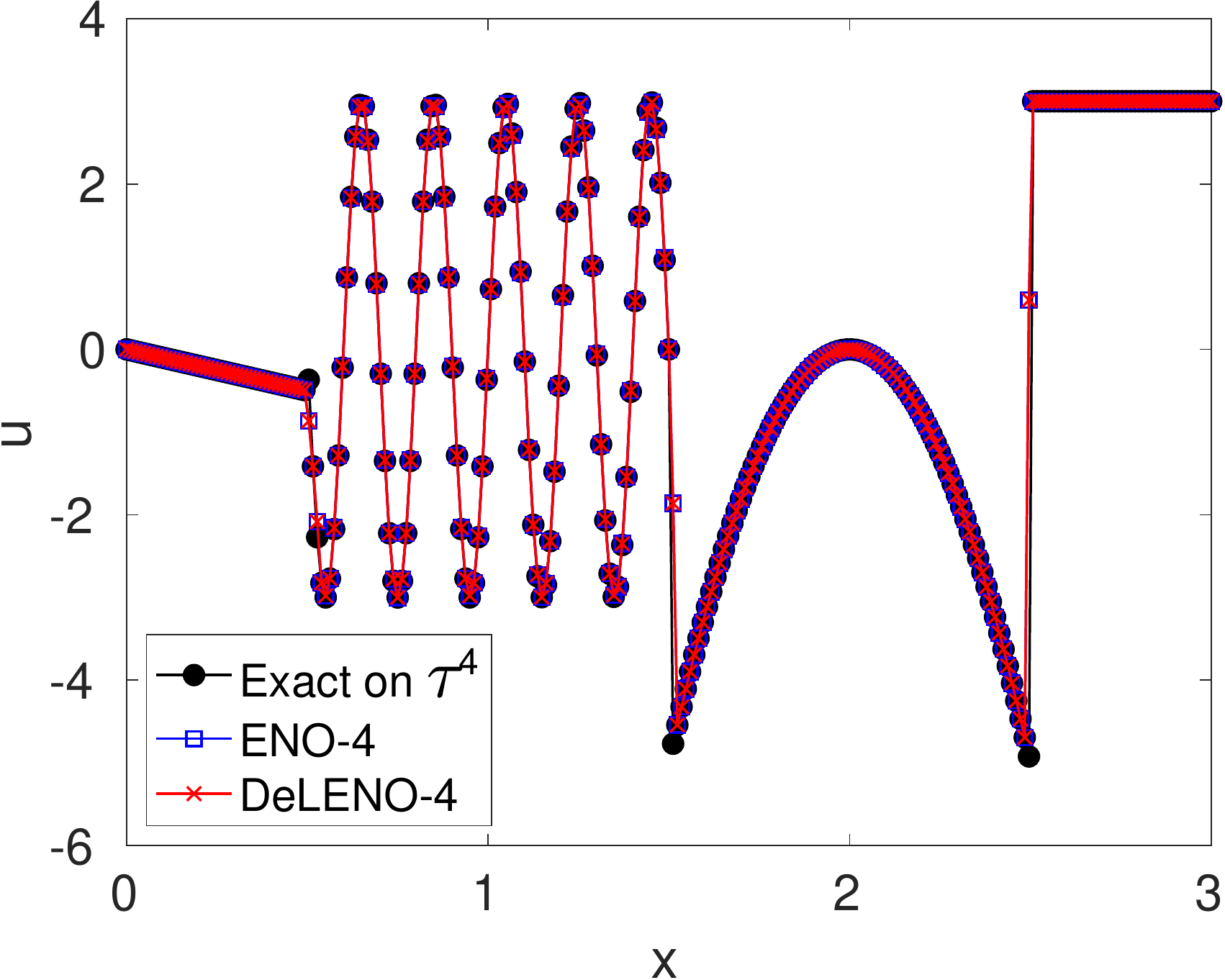}}
\caption{Interpolating the function \eqref{eqn:f_int} using ENO and DeLENO. }
\label{fig:fint_p34}
\end{center}
\end{figure}

\subsubsection{Data compression}\label{sec:datacompression}

We now apply the multi-resolution representation framework  of Appendix \ref{app:multi-res} to use DeLENO to compress the function \eqref{eqn:f_int}. We construct a nested sequence of meshes on $[0,3]$ by choosing $N_0 = 9$ and $K=5$ in \eqref{eqn:partitions}. We use Algorithm \ref{alg:encode} to obtain the multi-resolution representation of the form \eqref{eqn:cmr_rep} and decode the solution using Algorithm \ref{alg:decode} to obtain the approximation $\widehat{\q}^K$. The compression thresholds needed for the encoding procedure are set using \eqref{eqn:threshold}. 

Figure \ref{fig:1d_comp_soln_dk} provides a comparison of the results obtained using different values for the threshold parameters $\epsilon$, and shows the non-zero coefficients $\widehat{\db}^k$ for each mesh level $k$. A higher value of $\epsilon$ can truncate a larger number of $\widehat{\db}^k$ components, as is evident for $\pdeg=3$. However, there is no qualitative difference between $\widehat{\q}^K$ obtained for the two $\epsilon$ values considered. Thus, it is beneficial to use the larger $\epsilon$, as it leads to a sparser multi-resolution representation without deteriorating the overall features. The solutions obtained with ENO and DeLENO are indistinguishable. We refer to Table \ref{tab:1D_comp_err} for the errors of the two methods. 

\begin{figure}[!htbp]
\begin{center}
\subfigure[$\epsilon=0.5, p=3$]{\includegraphics[width=0.32\textwidth]{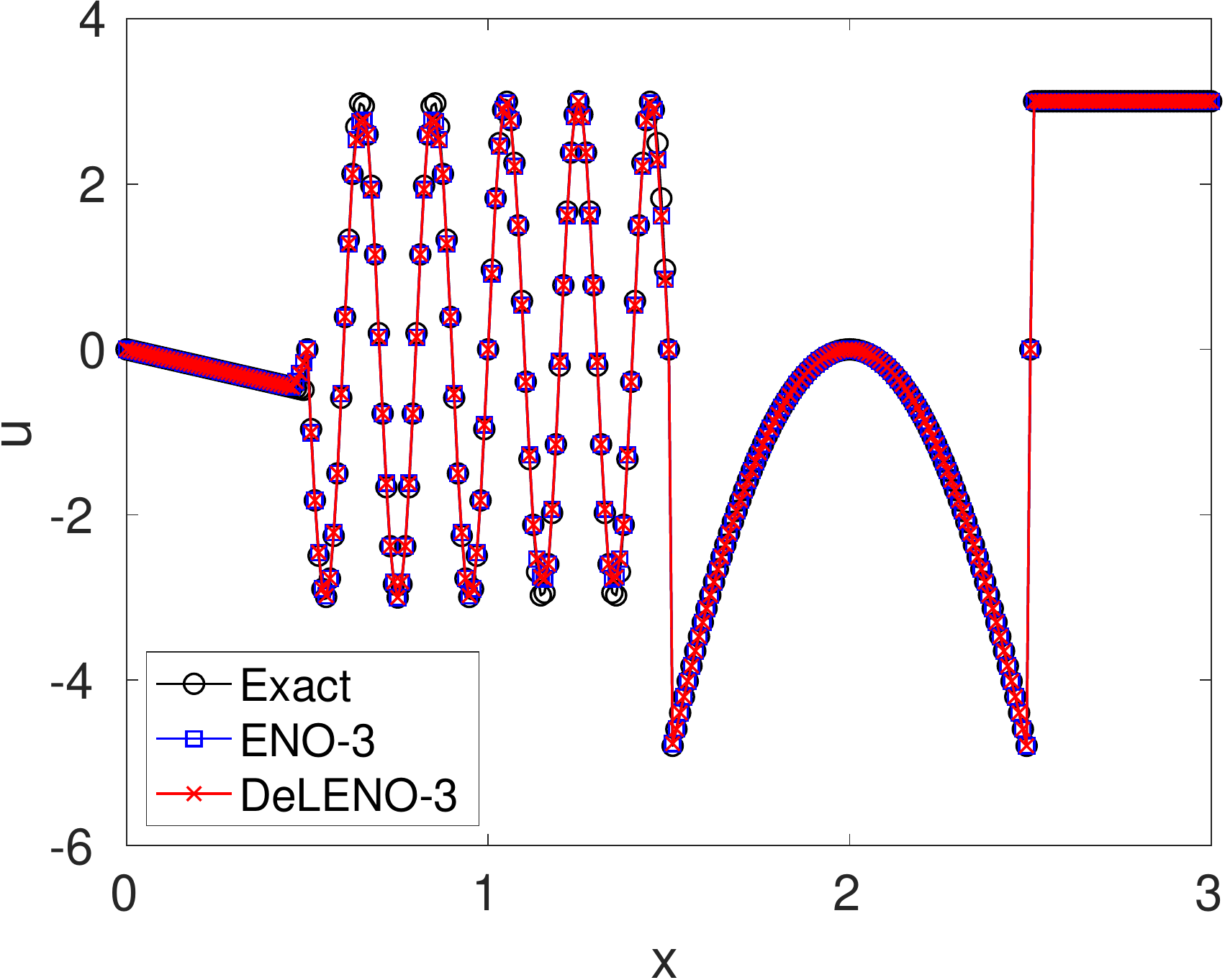}}
\subfigure[$\epsilon=0.5, p=3$]{\includegraphics[width=0.32\textwidth]{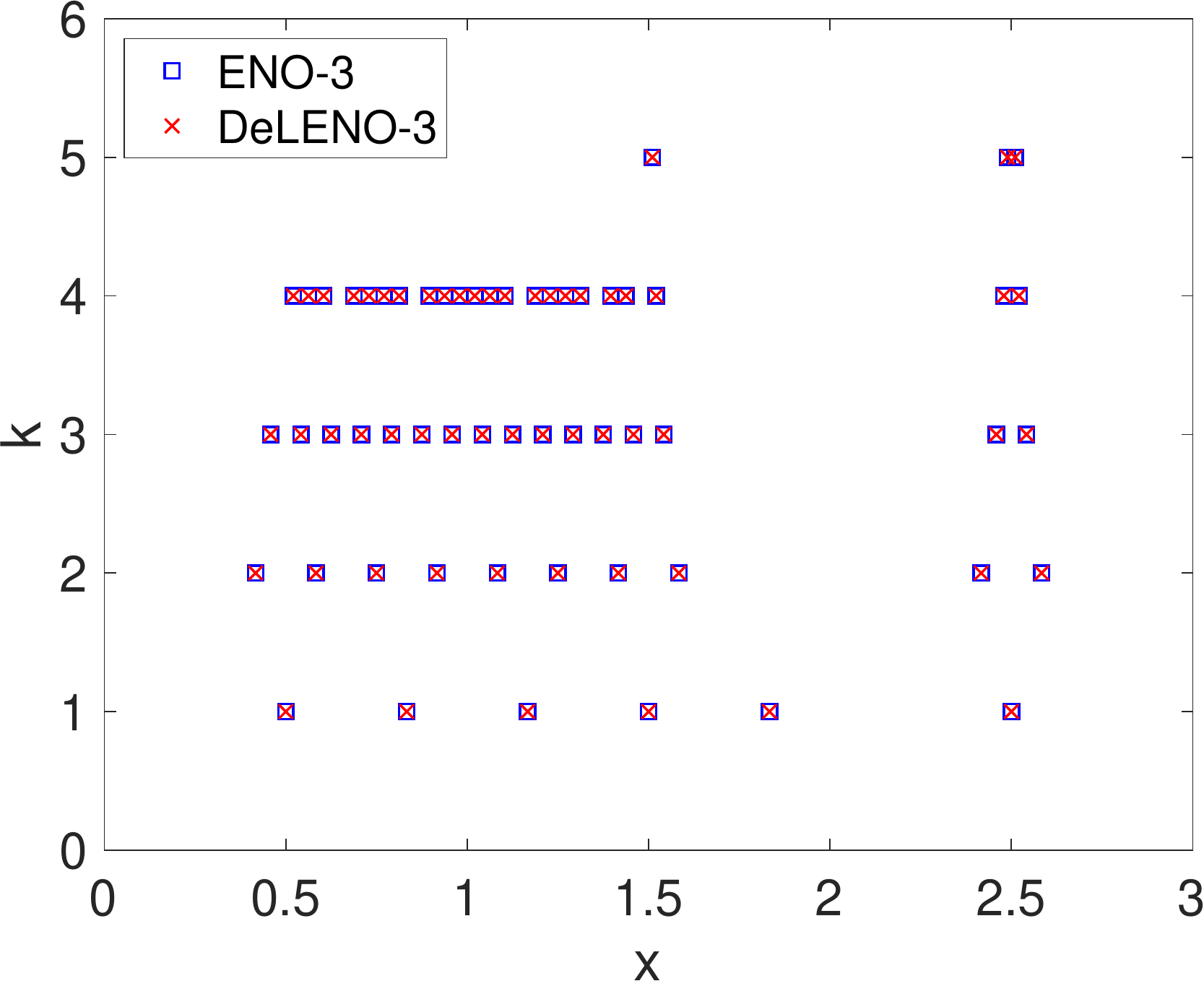}}\\
\subfigure[$\epsilon=0.5, p=4$]{\includegraphics[width=0.32\textwidth]{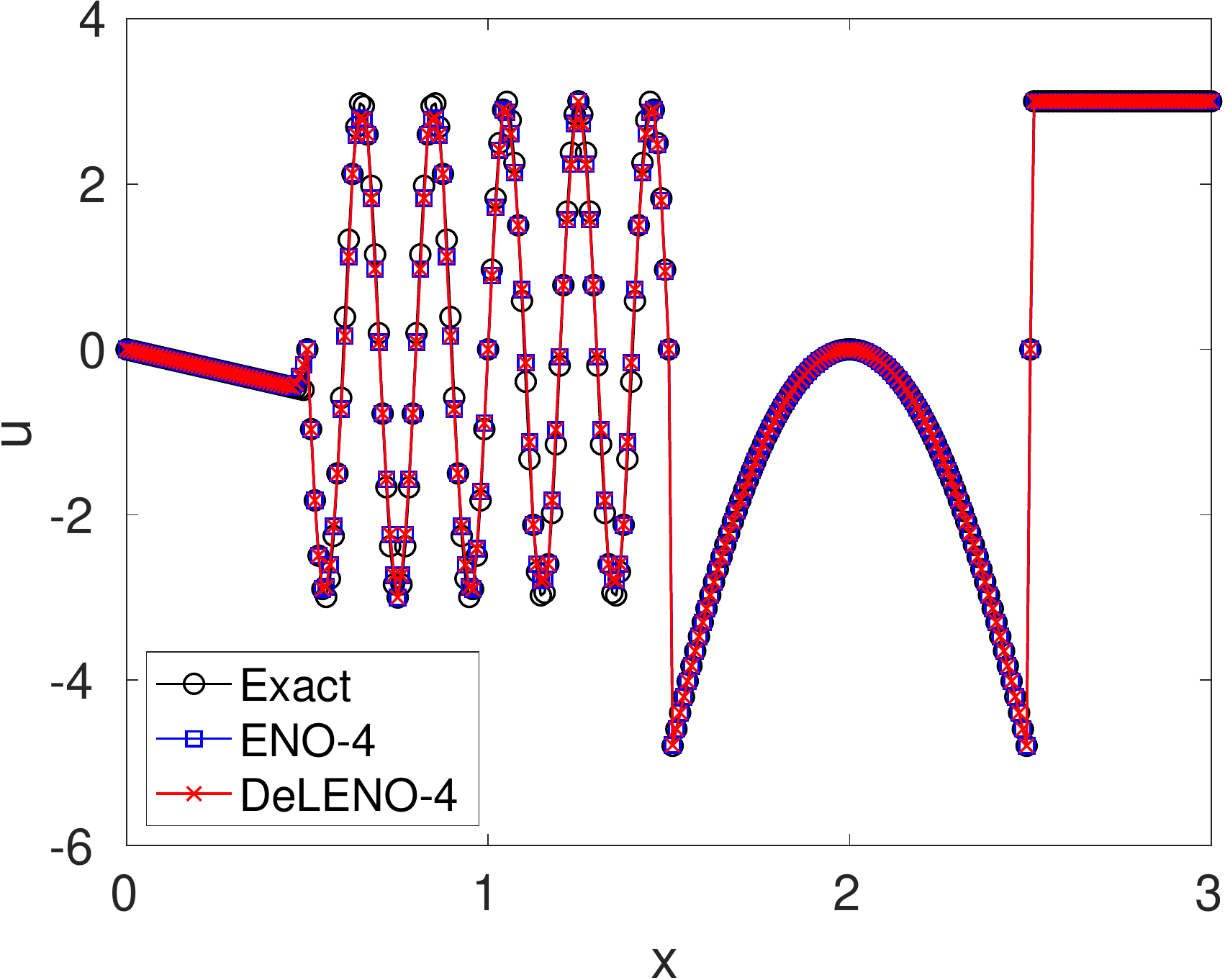}}
\subfigure[$\epsilon=0.5, p=4$]{\includegraphics[width=0.32\textwidth]{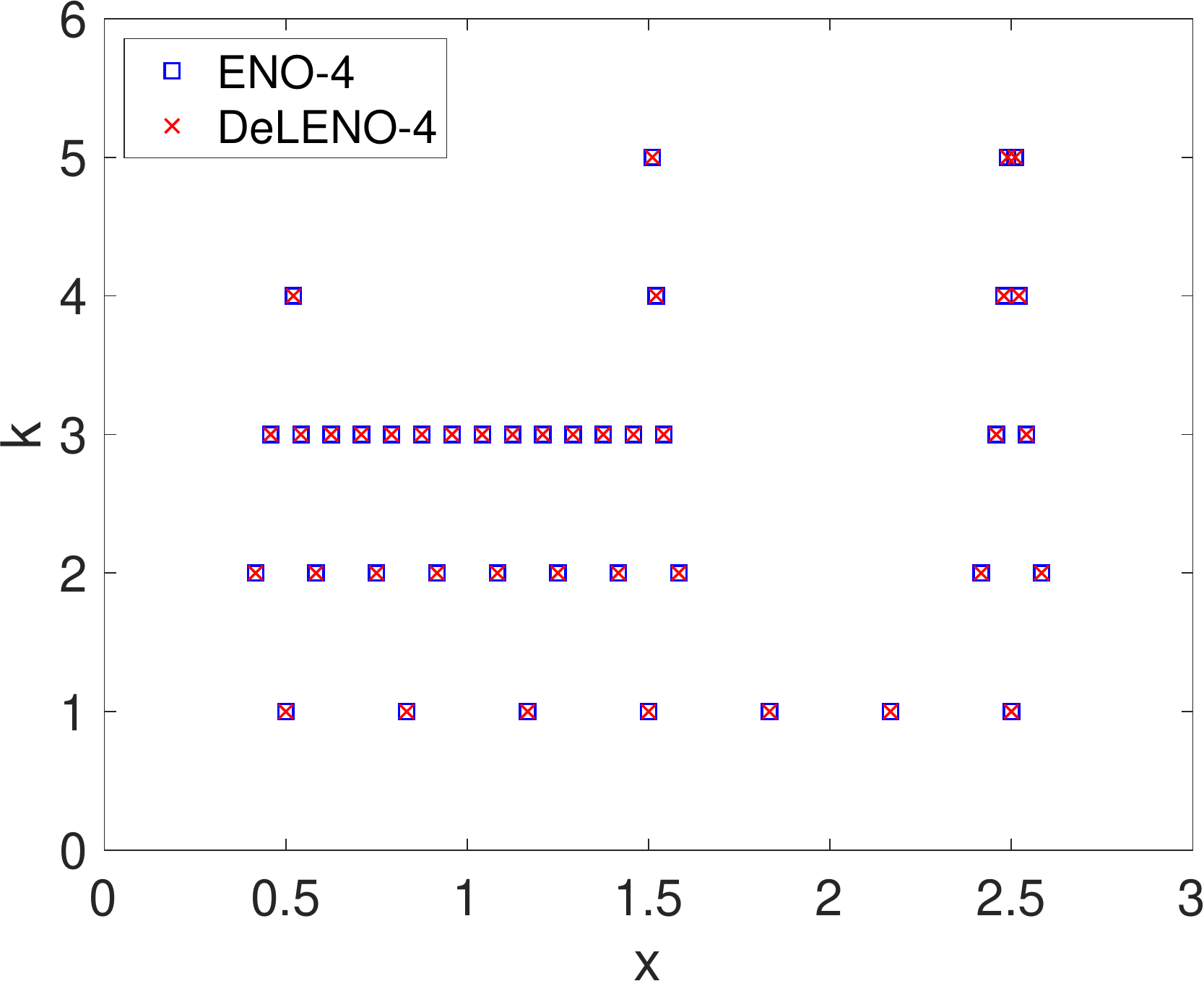}}\\
\subfigure[$\epsilon=1, p=3$]{\includegraphics[width=0.32\textwidth]{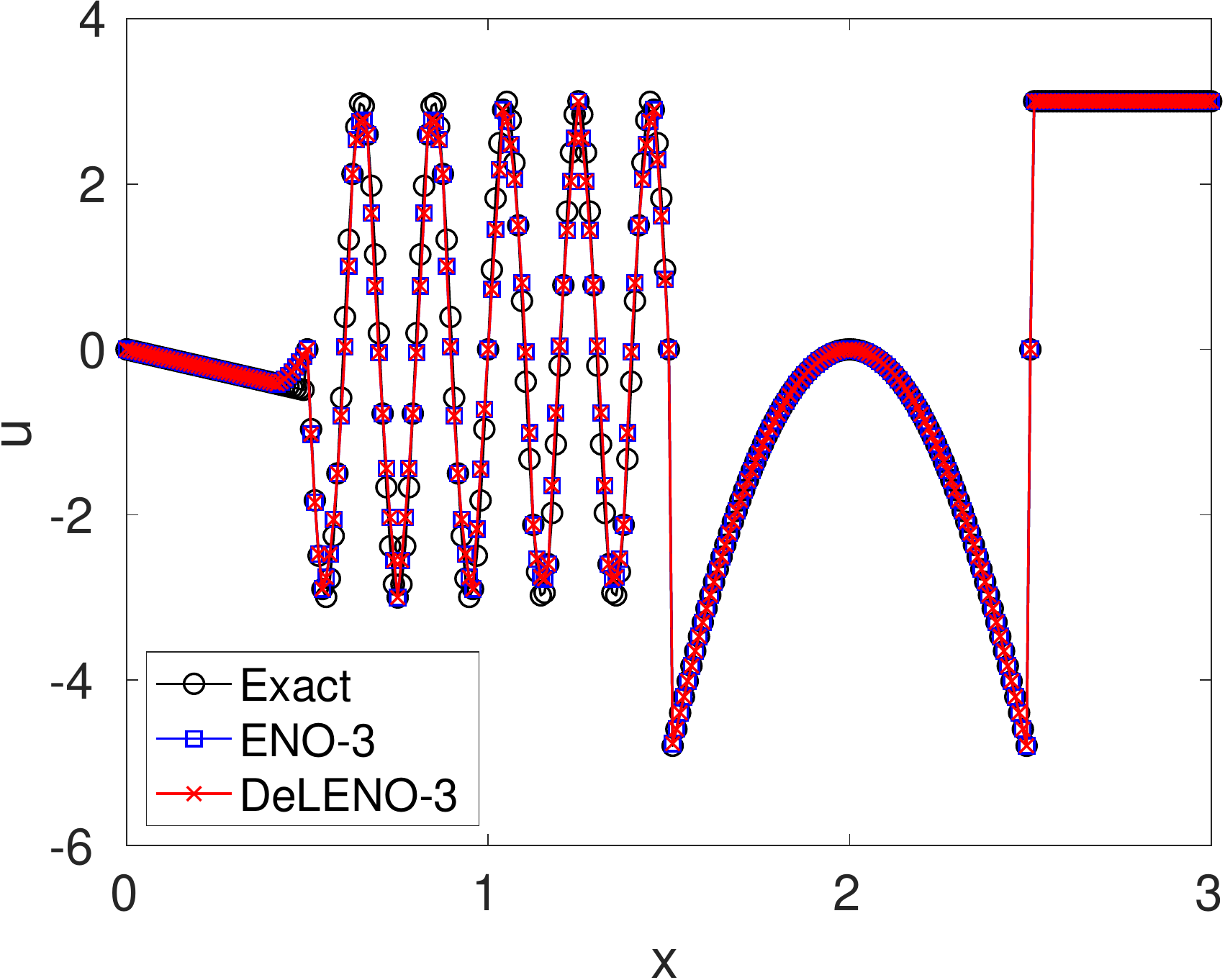}}
\subfigure[$\epsilon=1, p=3$]{\includegraphics[width=0.32\textwidth]{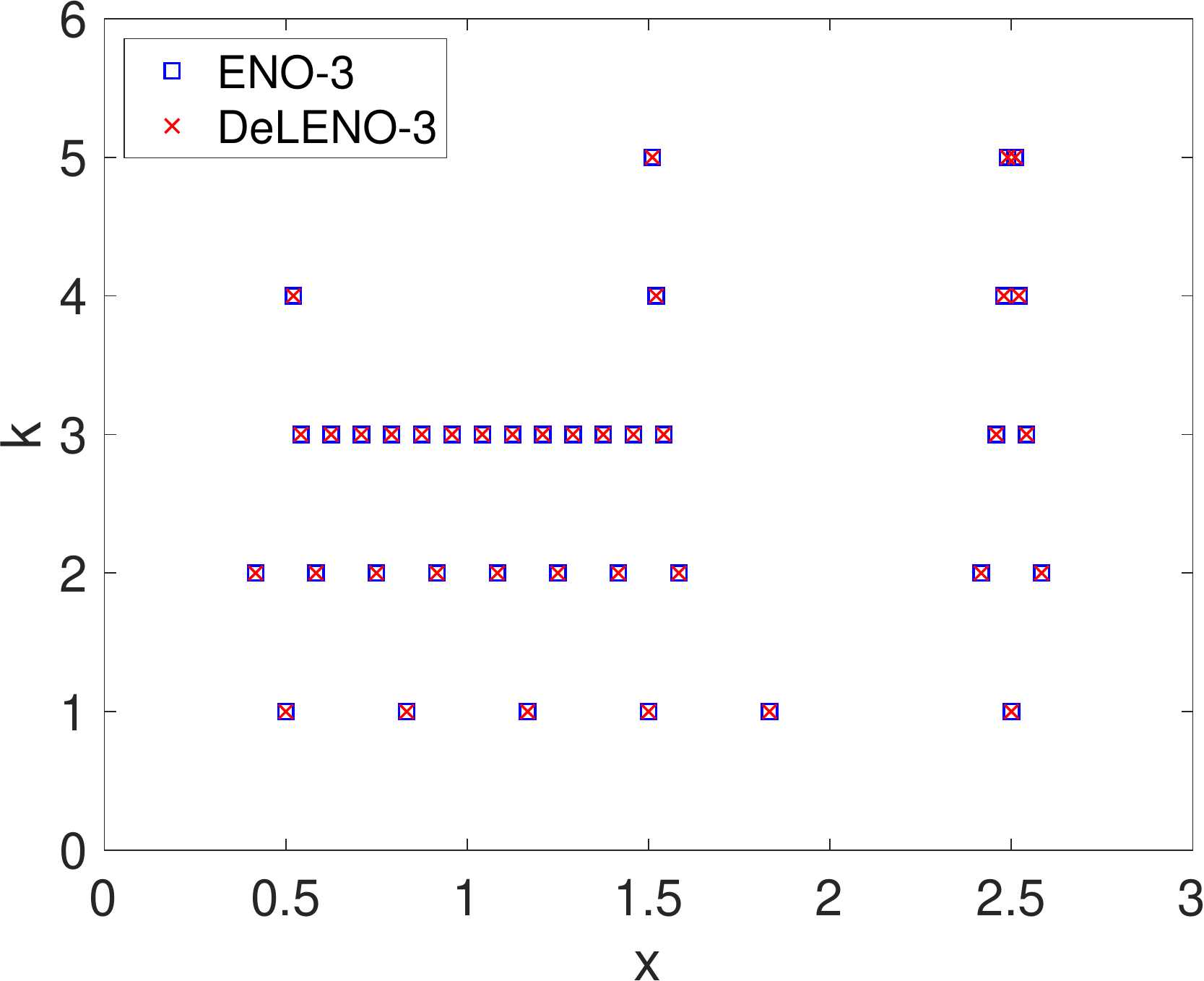}}\\
\subfigure[$\epsilon=1, p=4$]{\includegraphics[width=0.32\textwidth]{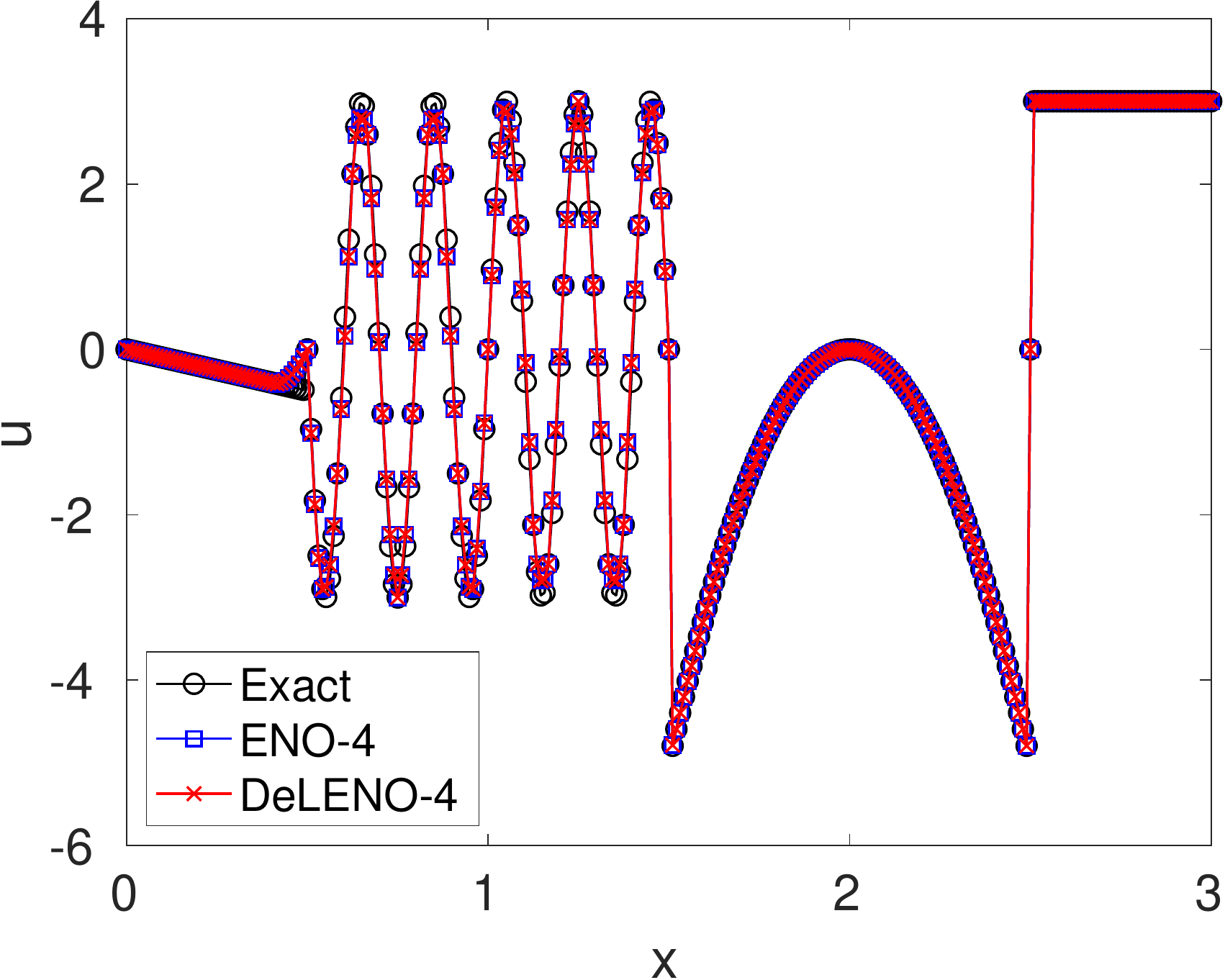}}
\subfigure[$\epsilon=1, p=4$]{\includegraphics[width=0.32\textwidth]{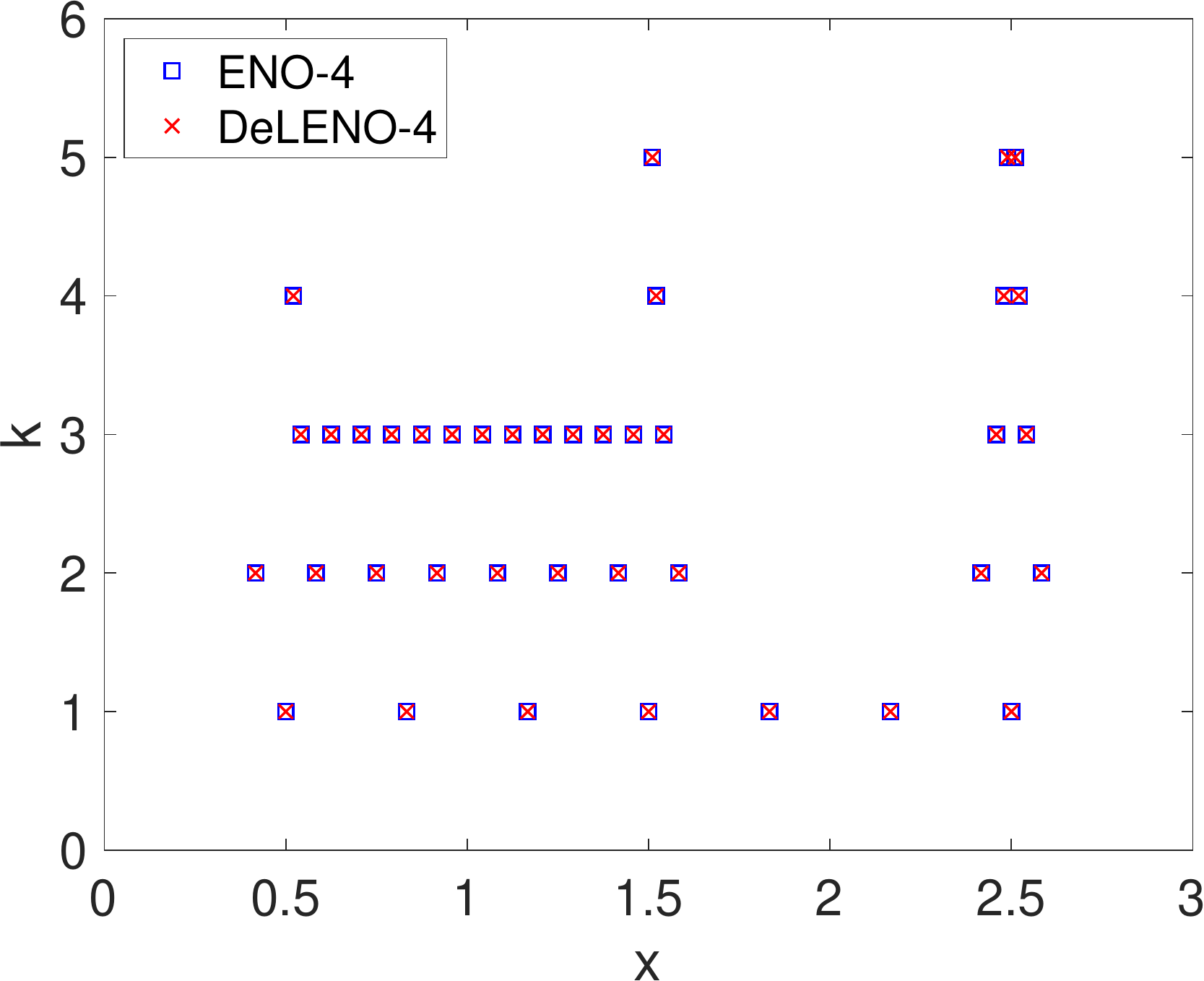}}\\
\caption{Data compression of \eqref{eqn:f_int} using ENO and DeLENO with $N_0$, $L=5$ and $t=0.5$. Comparison of thresholded decompressed data with the actual data on the finest level (left); Non-zero coefficients $\widehat{d}^k$ at each level (right).}
\label{fig:1d_comp_soln_dk}
\end{center}
\end{figure}

\begin{table}[!htbp]
\centering
\begin{tabular}{|c|c|c|c|c|c|c|c|}
\hline
\multirow{2}{*}{$p$} & \multirow{2}{*}{$\epsilon$} & \multicolumn{2}{c|}{$\|q^K - \widehat{q}^K\|_1$} & \multicolumn{2}{c|}{$\|q^K - \widehat{q}^K\|_2$} & \multicolumn{2}{c|}{$\|q^K - \widehat{q}^K\|_\infty$} \\ \cline{3-8} 
                     &                             & ENO                     & DeLENO                 & ENO                     & DeLENO                 & ENO                       & DeLENO                    \\ \hline
\multirow{2}{*}{3}   & 0.5         & 5.125e-2              & 5.125e-2                & 8.701e-2               & 8.701e-2                & 3.281e-1                & 3.281e-1                  \\ \cline{2-8} 
                                  & 1.0         & 2.072e-1              & 2.072e-1                & 2.421e-1               & 2.421e-1                & 4.102e-1                & 4.102e-1                  \\ \hline
\multirow{2}{*}{4}   & 0.5         & 1.032e-1              & 1.038e-1                & 1.268e-1               & 1.274e-1                & 3.027e-1                & 3.027e-1                  \\ \cline{2-8} 
                                  & 1.0         & 1.122e-1              & 1.122e-1                & 1.356e-1               & 1.356e-1                & 3.947e-1                & 3.947e-1                  \\ \hline                                  
\end{tabular}
\caption{1D compression errors for \eqref{eqn:f_int}.}
\label{tab:1D_comp_err}
\end{table}

The compression ideas used for one-dimensional problems can be easily extended to handle functions defined on two-dimensional tensorized grids. We consider a sequence of grids $\Tau^k$ with $(N^x_k+1) \times (N^y_k + 1)$ nodes, where 
$N^x_k = 2^k N^x_0$ and $N^y_k = 2^k N^x_0$, for $0\leq k \leq K$.
Let $\q^k$ be the data on grid $\Tau^k$ and denote by $\widehat{\q}^{k+1}$ the compressed interpolation on grid $\Tau^{k+1}$. 
To obtain $\widehat{\q}^{k+1}$, we first interpolate along the $x$-coordinate direction to obtain an intermediate approximation $\widetilde{\q}^{k+1}$ of size $(N^x_{k+1}+1) \times (N^y_k + 1)$. Then we use $\widetilde{\q}^{k+1}$ to interpolate along the $y$-coordinate direction to obtain the final approximation $\widehat{\q}^{k+1}$. 

Next, we use ENO and DeLENO to compress an image with $705 \times 929$ pixels, shown in Figure \ref{fig:image_original}. We set $K=5$, $\epsilon = 1$, $t=0.2$ in equation \eqref{eqn:threshold}. Once again, ENO and DeLENO give similar results, as can be seen from the decompressed images in Figure \ref{fig:image} and the relative errors in Table \ref{tab:image_err}. In this table we additionally listed the compression rate
\begin{equation}\label{eqn:compression_rate}
    c_r = 1 - \frac{\#\left\{ d^k_{i,j}  | d^k_{i,j} > \epsilon^k, \ 1\leq k \leq K \right\}}{(N^x_L + 1)(N^y_L + 1) - (N^x_0 + 1)(N^y_0 + 1)},
\end{equation}
which represents the fraction of coefficients set to null. 

\begin{table}[!htbp]
\centering
\begin{tabular}{|c|l|l|l|l|c|}
\hline
$p$                & Scheme & Rel. $L^1$ & Rel. $L^2$ & Rel. $L^\infty$ & $c_r$ \\ \hline
\multirow{2}{*}{3} & ENO    & 5.346e-2   & 8.368e-2   & 5.194e-1        & 0.996 \\ \cline{2-6} 
                   & DeLENO          & 5.343e-3   & 8.365e-2   & 5.194e-1        & 0.996 \\ \hline
\multirow{2}{*}{4} & ENO    & 5.422e-2   & 8.485e-2   & 5.581e-1        & 0.996 \\ \cline{2-6} 
                   & DeLENO          & 5.422e-2   & 8.492e-2   & 5.581e-1        & 0.996 \\ \hline
\end{tabular}
\caption{Image compression errors. }
\label{tab:image_err}
\end{table}
\begin{figure}[!htbp]
\begin{center}
\subfigure[Original]{\includegraphics[width=0.19\textwidth]{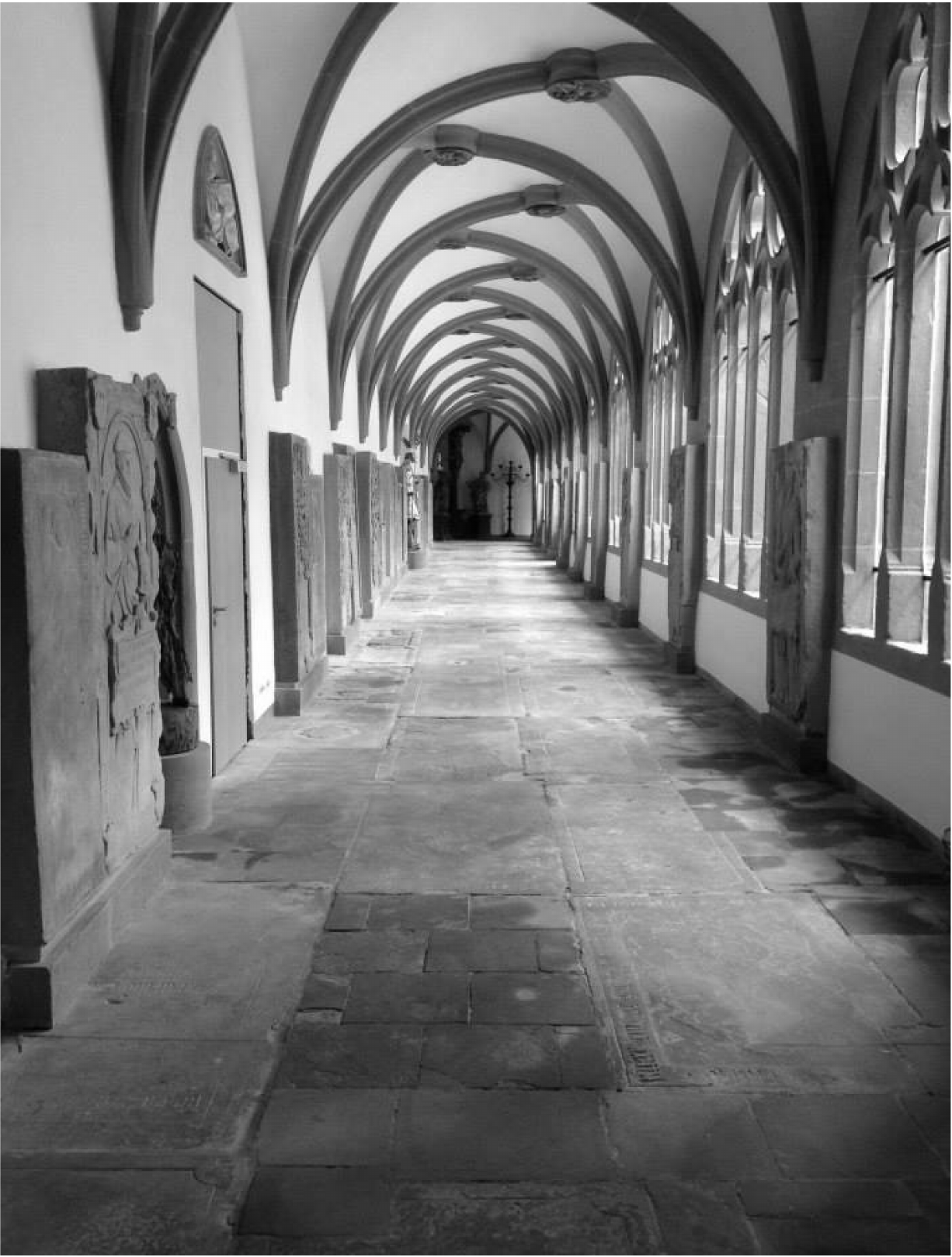}\label{fig:image_original}}
\subfigure[ENO-3]{\includegraphics[width=0.19\textwidth]{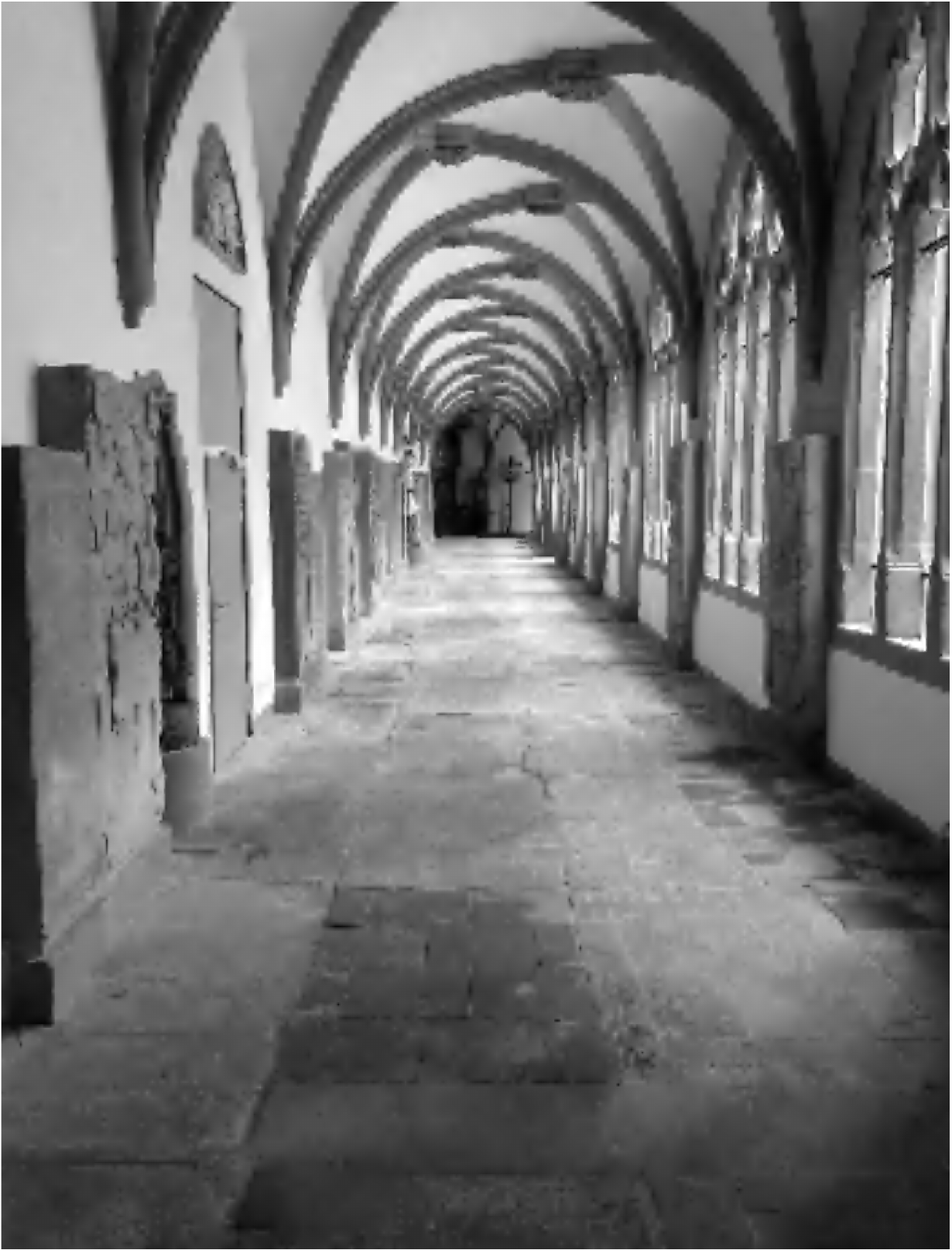}}
\subfigure[DeLENO-3]{\includegraphics[width=0.19\textwidth]{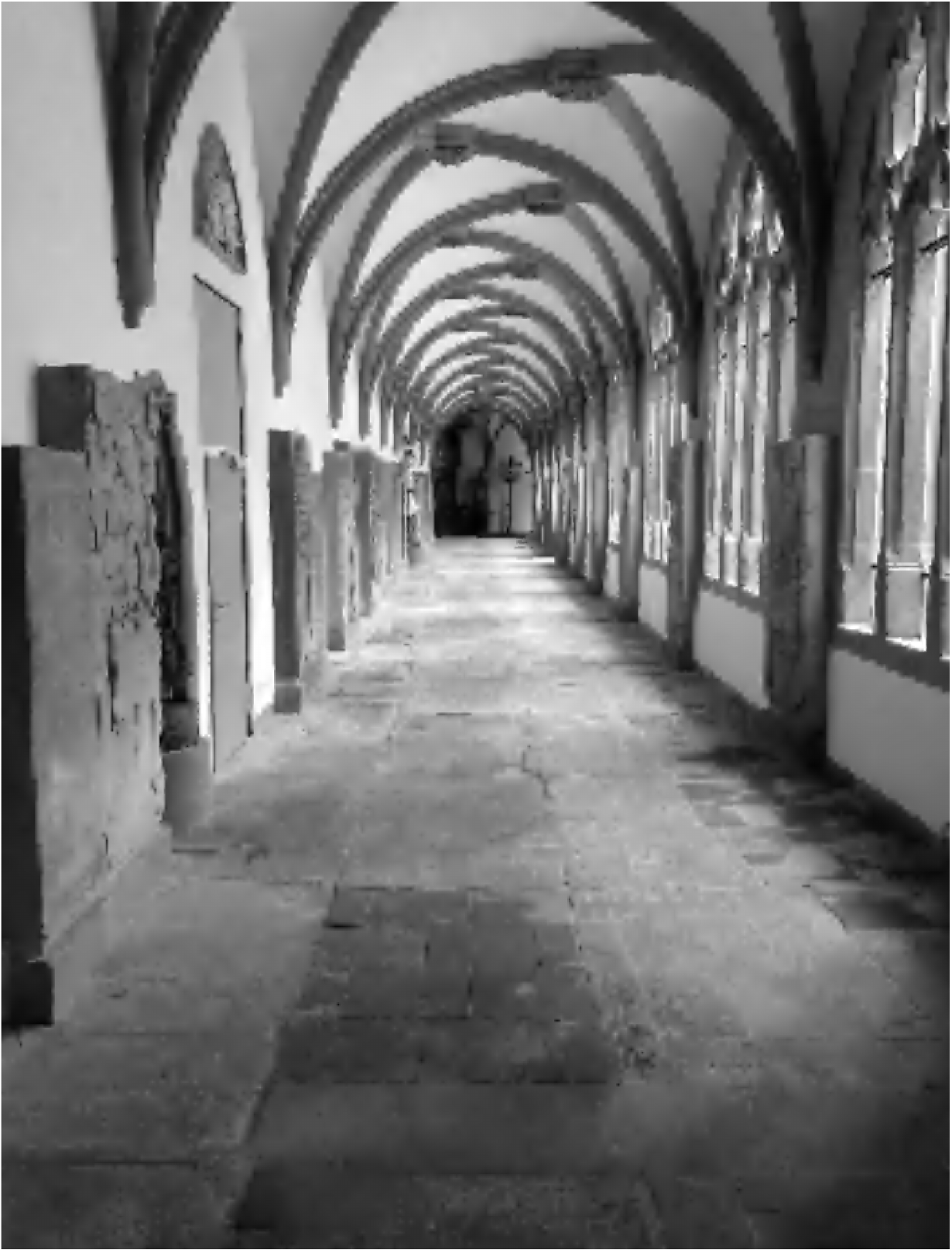}}
\subfigure[ENO-4]{\includegraphics[width=0.19\textwidth]{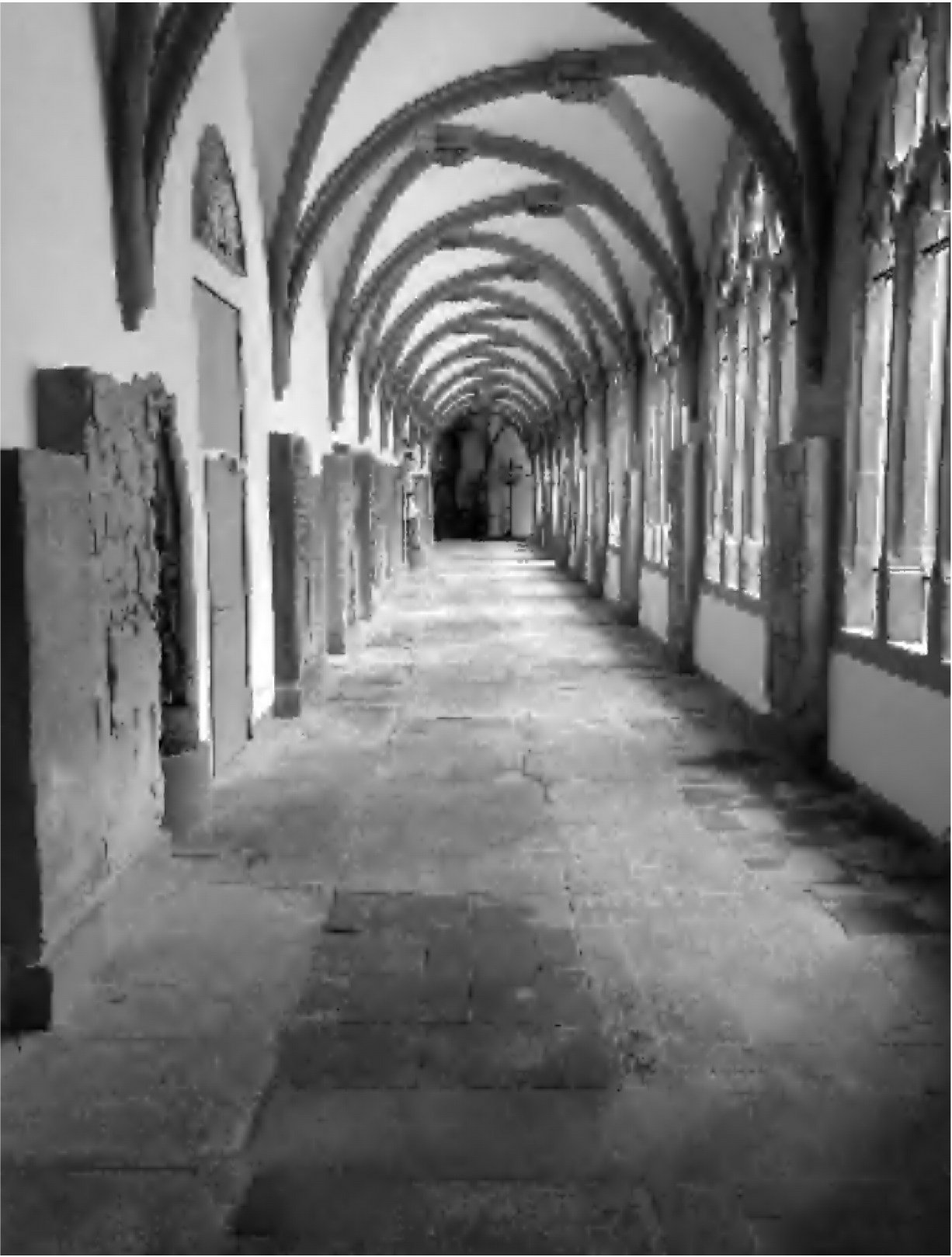}}
\subfigure[DeLENO-4]{\includegraphics[width=0.19\textwidth]{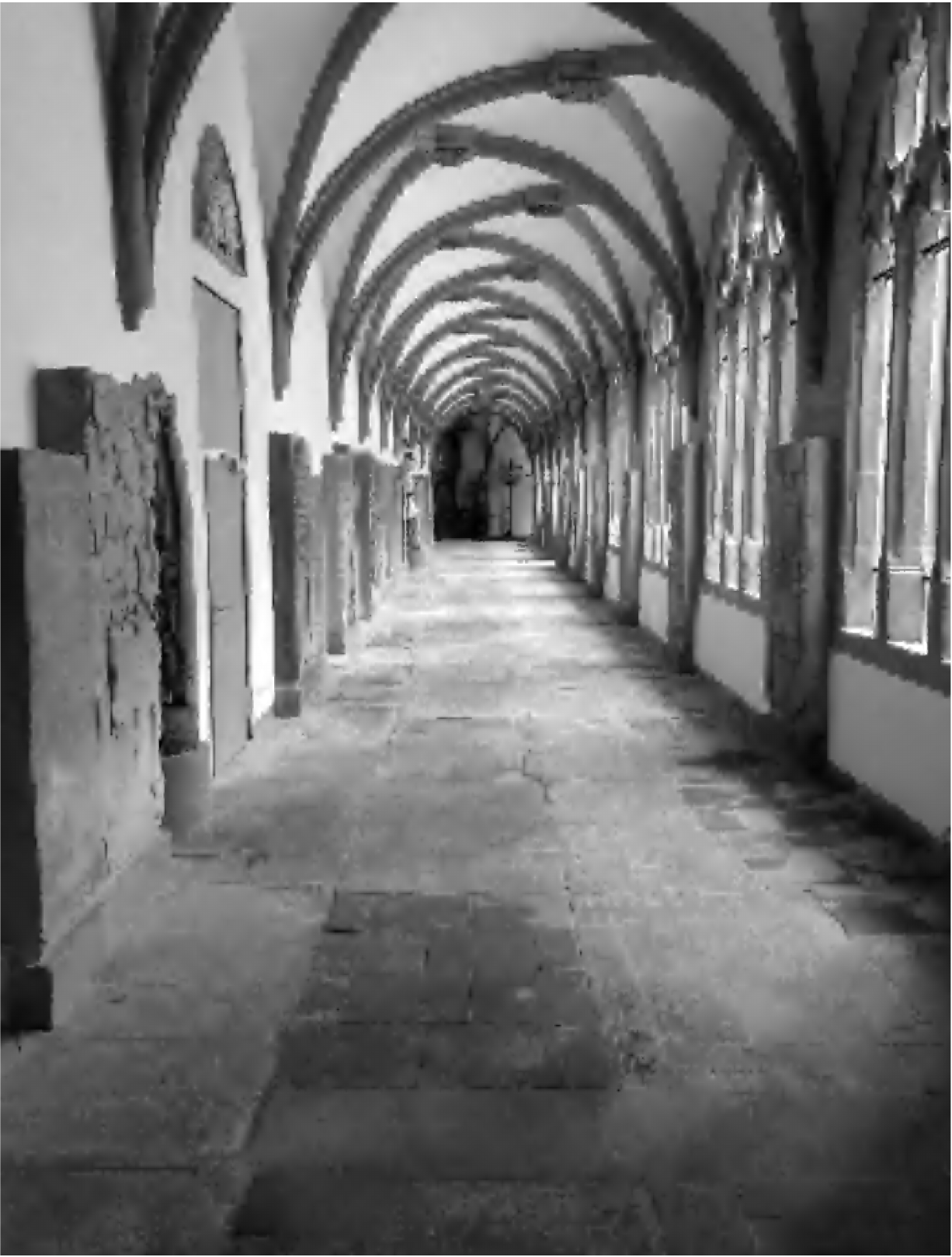}}
\caption{Image compression. }
\label{fig:image}
\end{center}
\end{figure}

As an additional example of two-dimensional data compression, we consider the function
\begin{equation}\label{eqn:2d_func}
q(x,y) = \begin{cases}
           -10 & \quad \text{if } (x - 0.5)^2 + (y-0.5)^2 < 0.0225\\
            30 & \quad \text{if } |x-0.5|>0.8 \text{ or } |y-0.5|>0.8\\
           40 & \quad \text{otherwise }\\
	\end{cases},
\end{equation}
where $(x,y) \in [0,1] \times [0,1]$, and generate a sequence of meshes by setting $K=4$, $N_0^x = 16$ and $N_0^y = 16$. The threshold for data compression is chosen according to \eqref{eqn:threshold}, with $\epsilon=10$ and $t=0.5$. The non-zero $\widehat{\db}^k$ coefficients are plotted in Figure \ref{fig:2d_comp_dk}, while the errors and compression rate \eqref{eqn:compression_rate} are listed in Table \ref{tab:2D_comp_err}. Overall, ENO and DeLENO perform equally well, with DeLENO giving marginally smaller errors.

\begin{table}[!htbp]
\centering
\begin{tabular}{|c|l|l|l|l|c|}
\hline
$p$                & Scheme & Rel. $L^1$ & Rel. $L^2$ & Rel. $L^\infty$ & $c_r$ \\ \hline
\multirow{2}{*}{3} & ENO    & 3.341e-3   & 2.442e-2   & 4.302e-1        & 0.989 \\ \cline{2-6} 
                   & DeLENO          & 3.246e-3   & 2.367e-2   & 4.302e-1        & 0.989 \\ \hline
\multirow{2}{*}{4} & ENO    & 3.816e-3   & 3.237e-2   & 5.876e-1        & 0.989 \\ \cline{2-6} 
                   & DeLENO          & 3.681e-3   & 3.130e-2   & 5.876e-1        & 0.989 \\ \hline
\end{tabular}
\caption{2D compression errors for \eqref{eqn:2d_func}. }
\label{tab:2D_comp_err}
\end{table}

\begin{figure}[!htbp]
\begin{center}
\subfigure[ENO-3, $k=1$]{\includegraphics[width=0.24\textwidth]{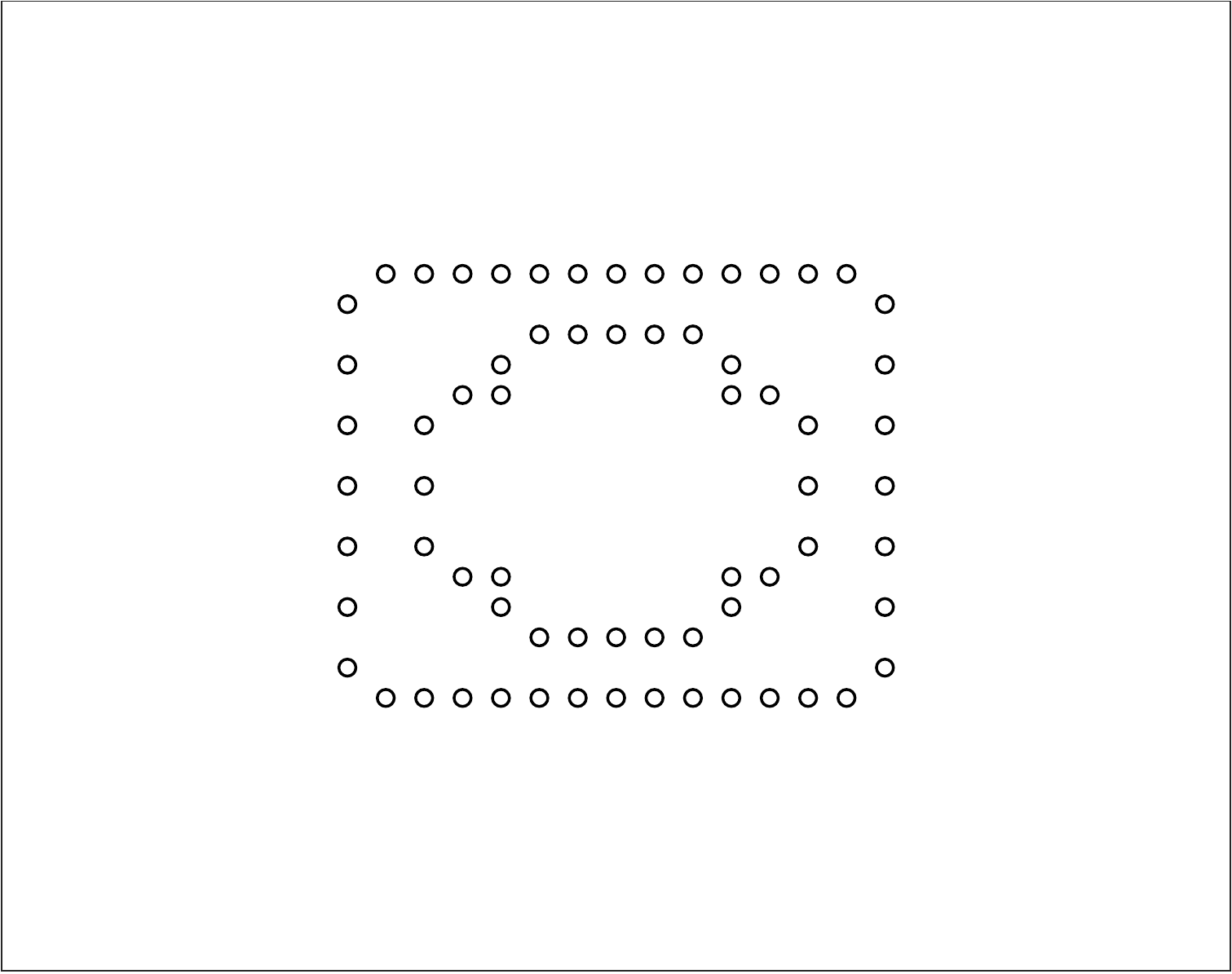}}
\subfigure[ENO-3, $k=2$]{\includegraphics[width=0.24\textwidth]{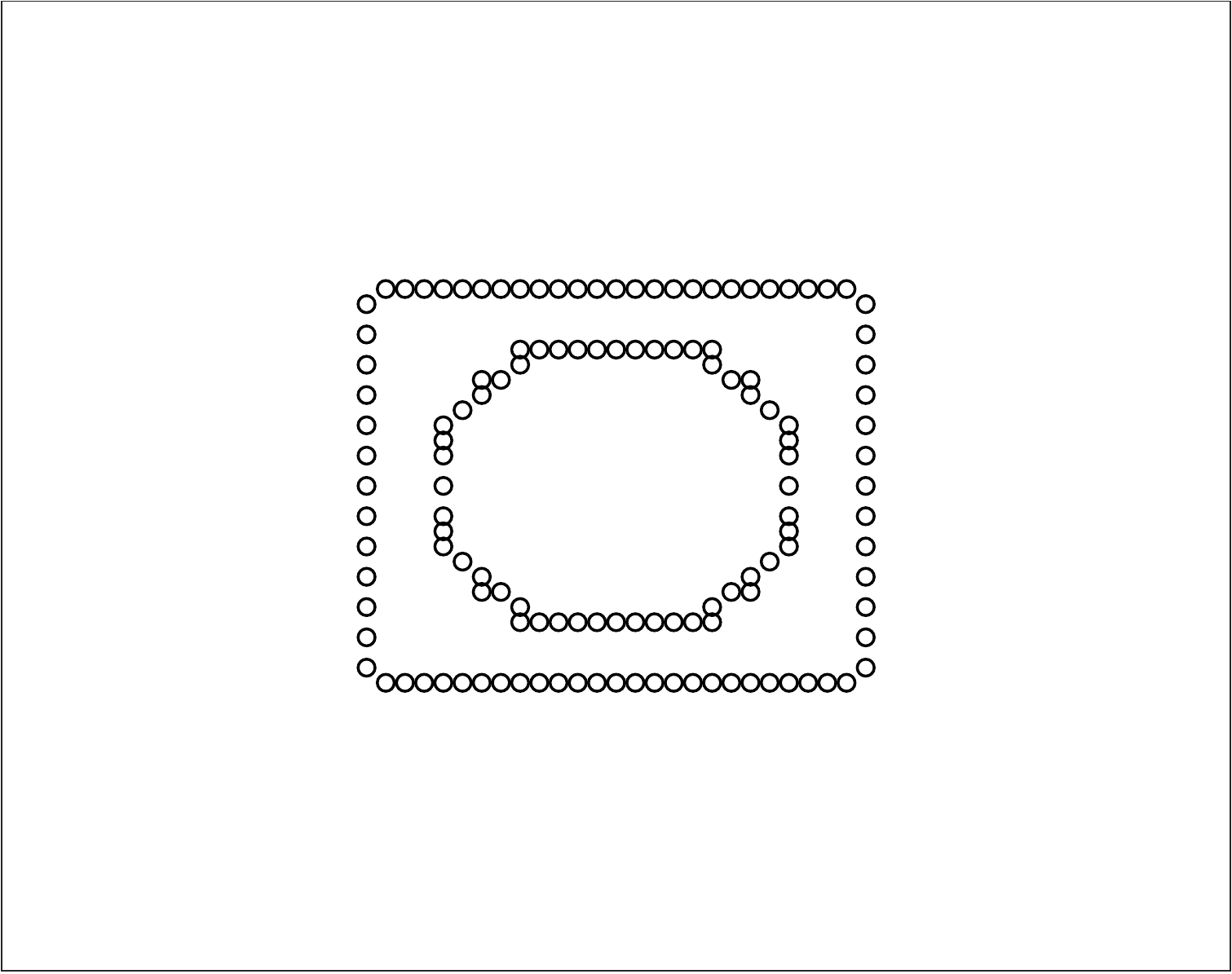}}
\subfigure[ENO-3, $k=3$]{\includegraphics[width=0.24\textwidth]{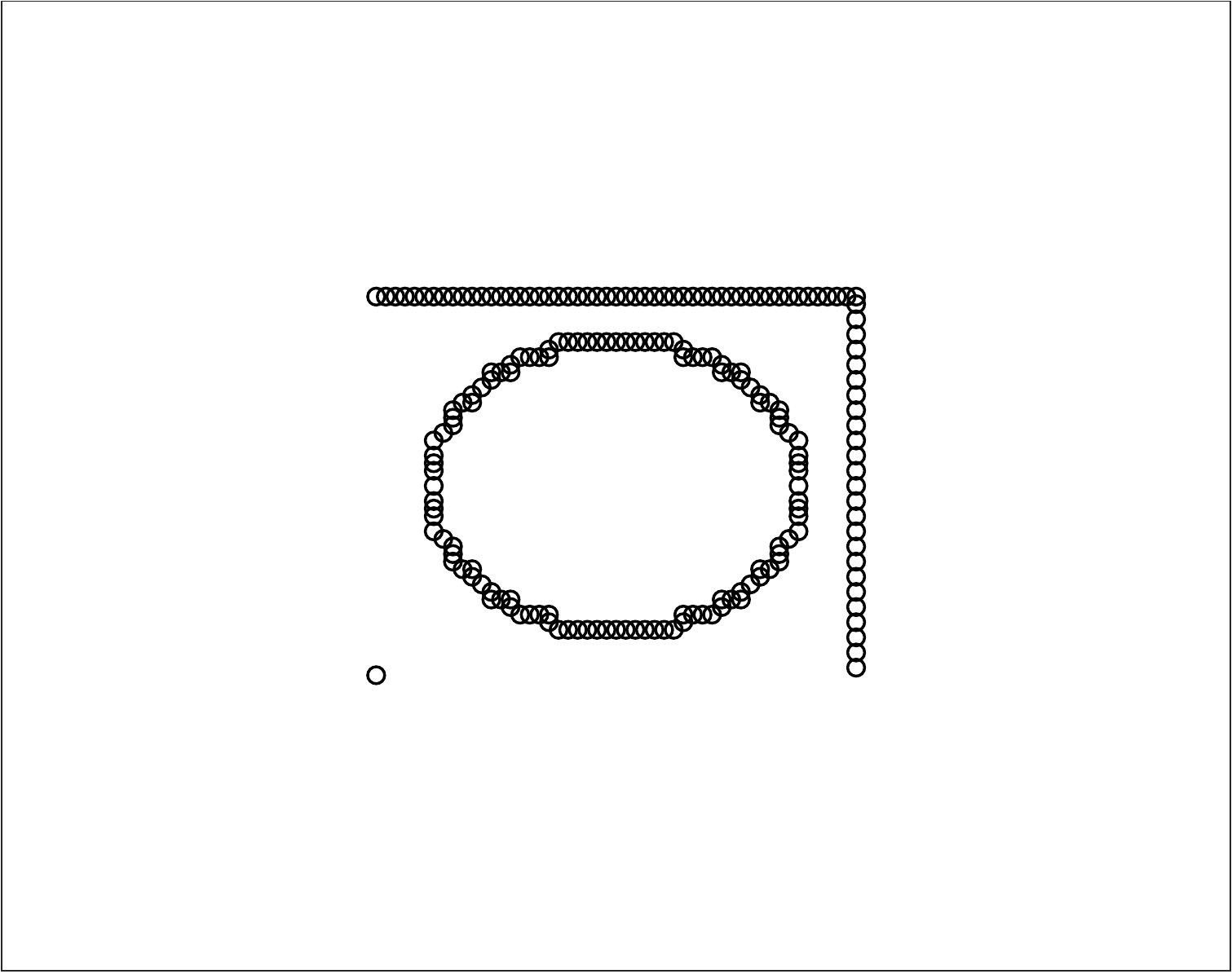}}
\subfigure[ENO-3, $k=4$]{\includegraphics[width=0.24\textwidth]{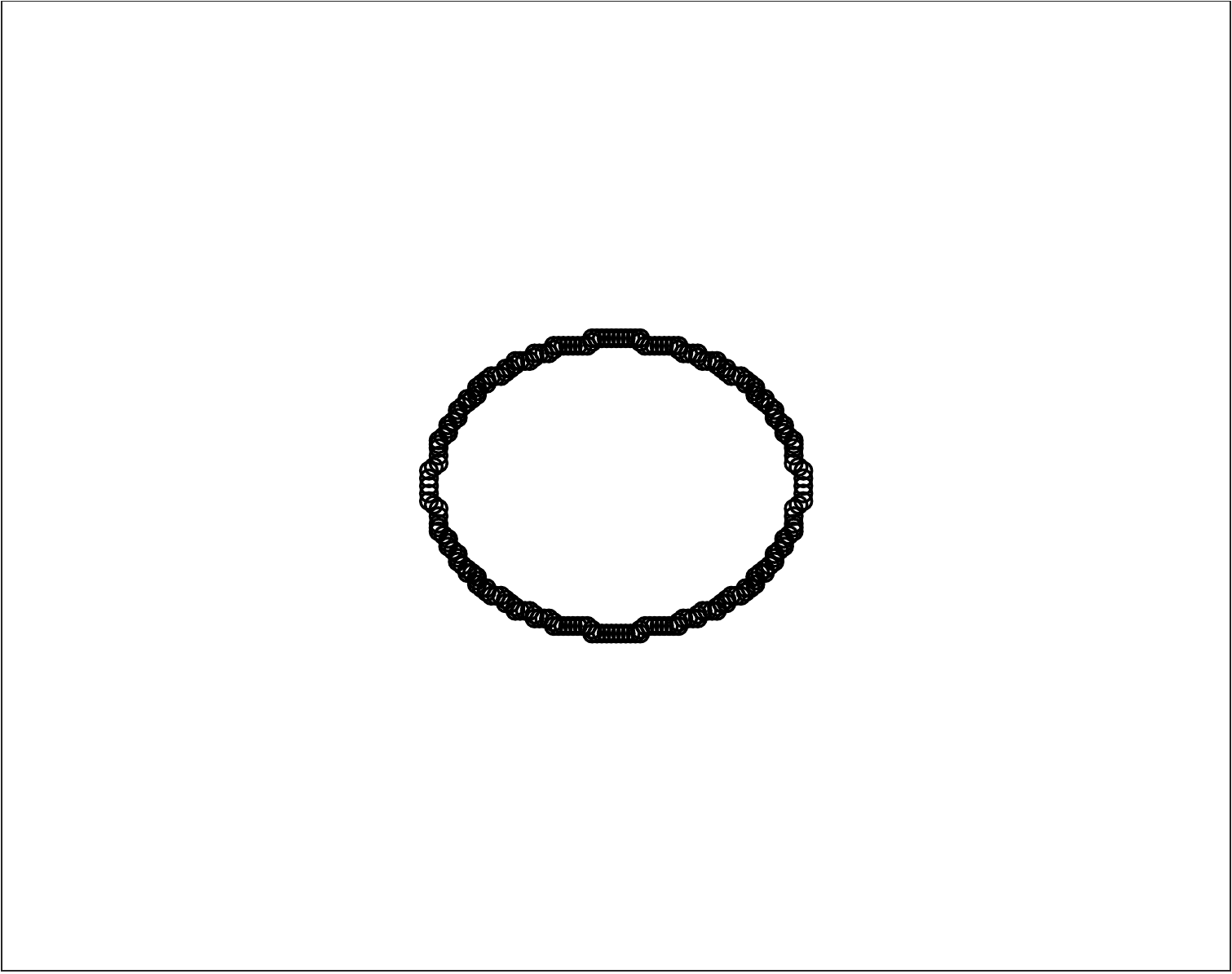}}\\
\subfigure[DeLENO-3, $k=1$]{\includegraphics[width=0.24\textwidth]{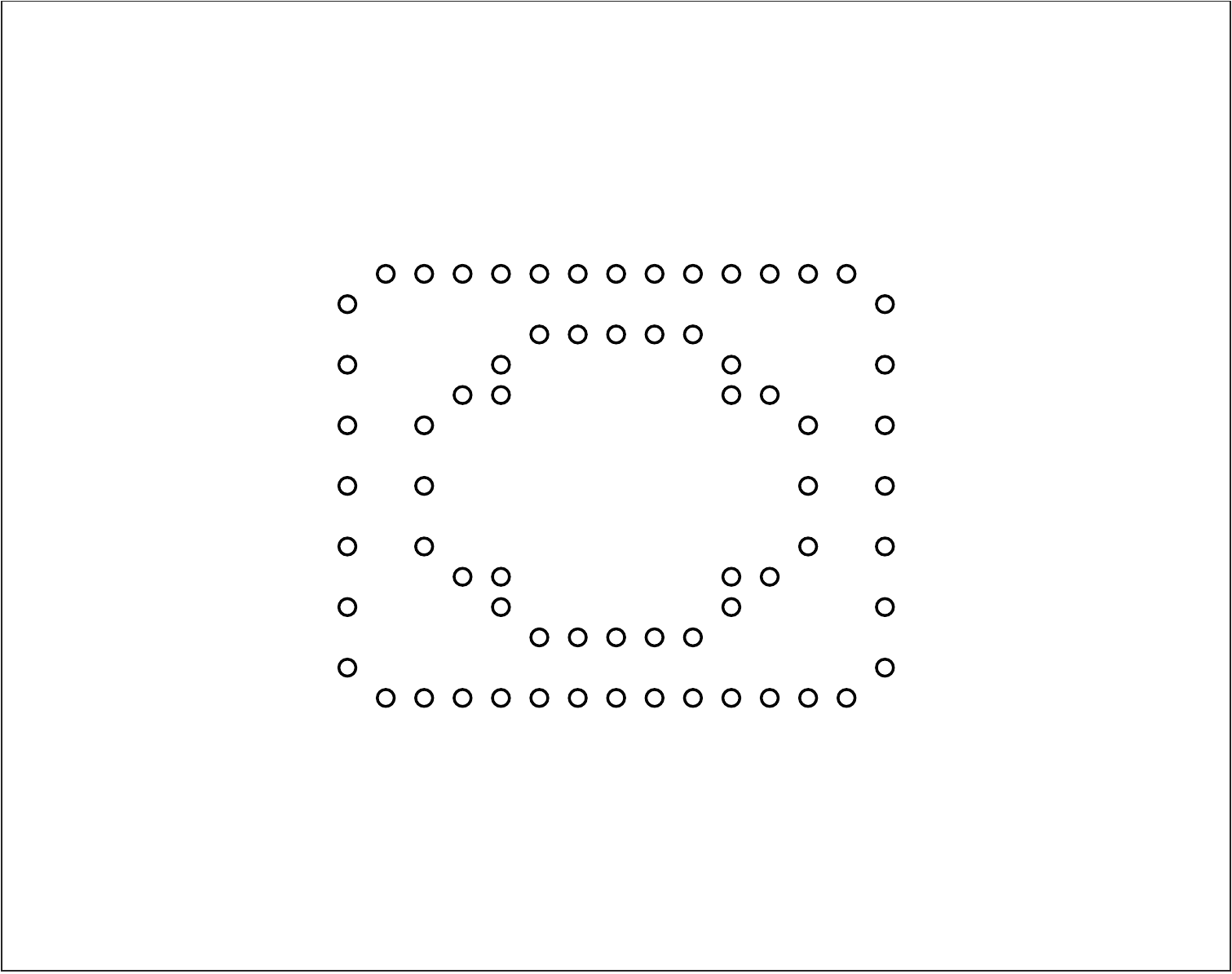}}
\subfigure[DeLENO-3, $k=2$]{\includegraphics[width=0.24\textwidth]{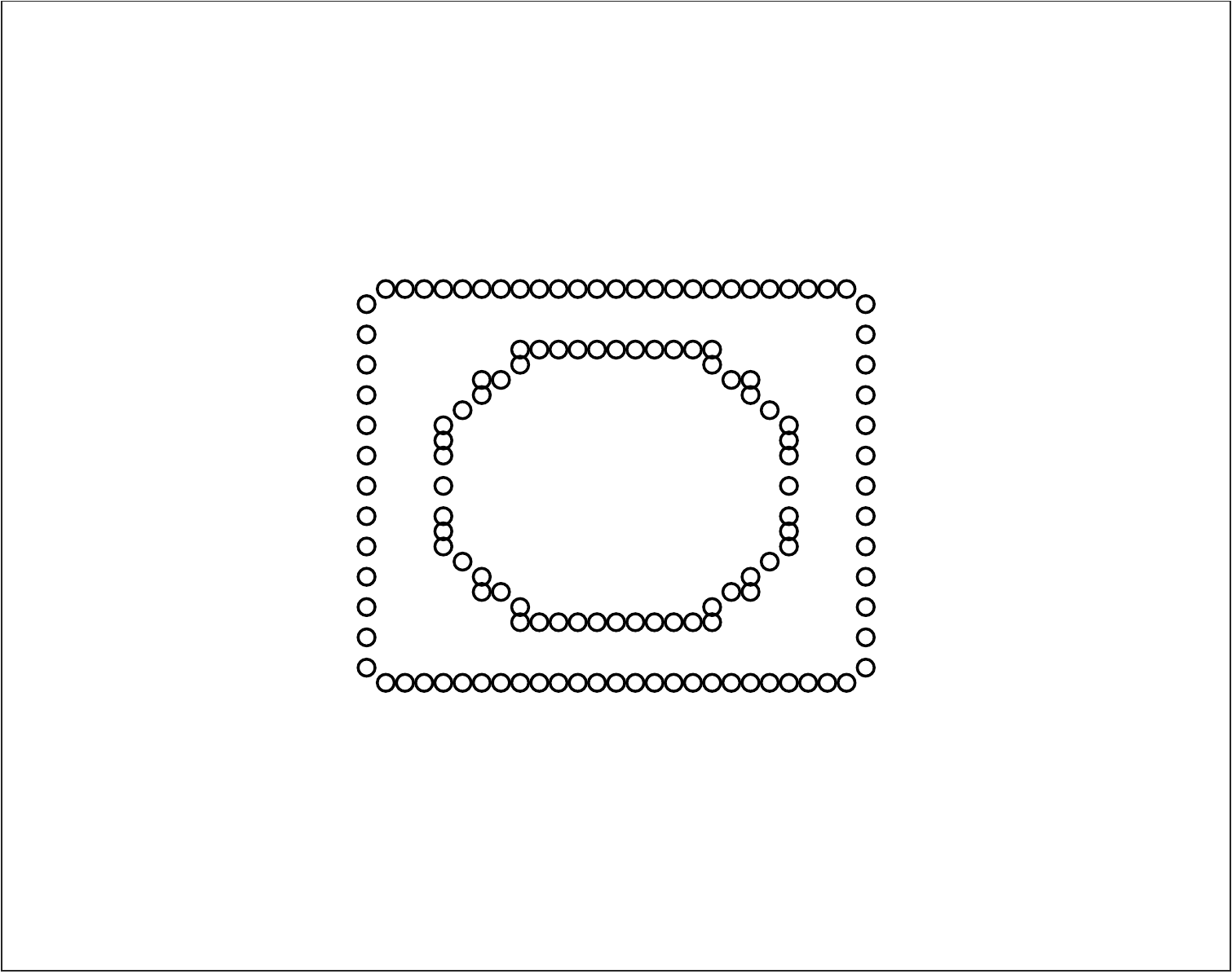}}
\subfigure[DeLENO-3, $k=3$]{\includegraphics[width=0.24\textwidth]{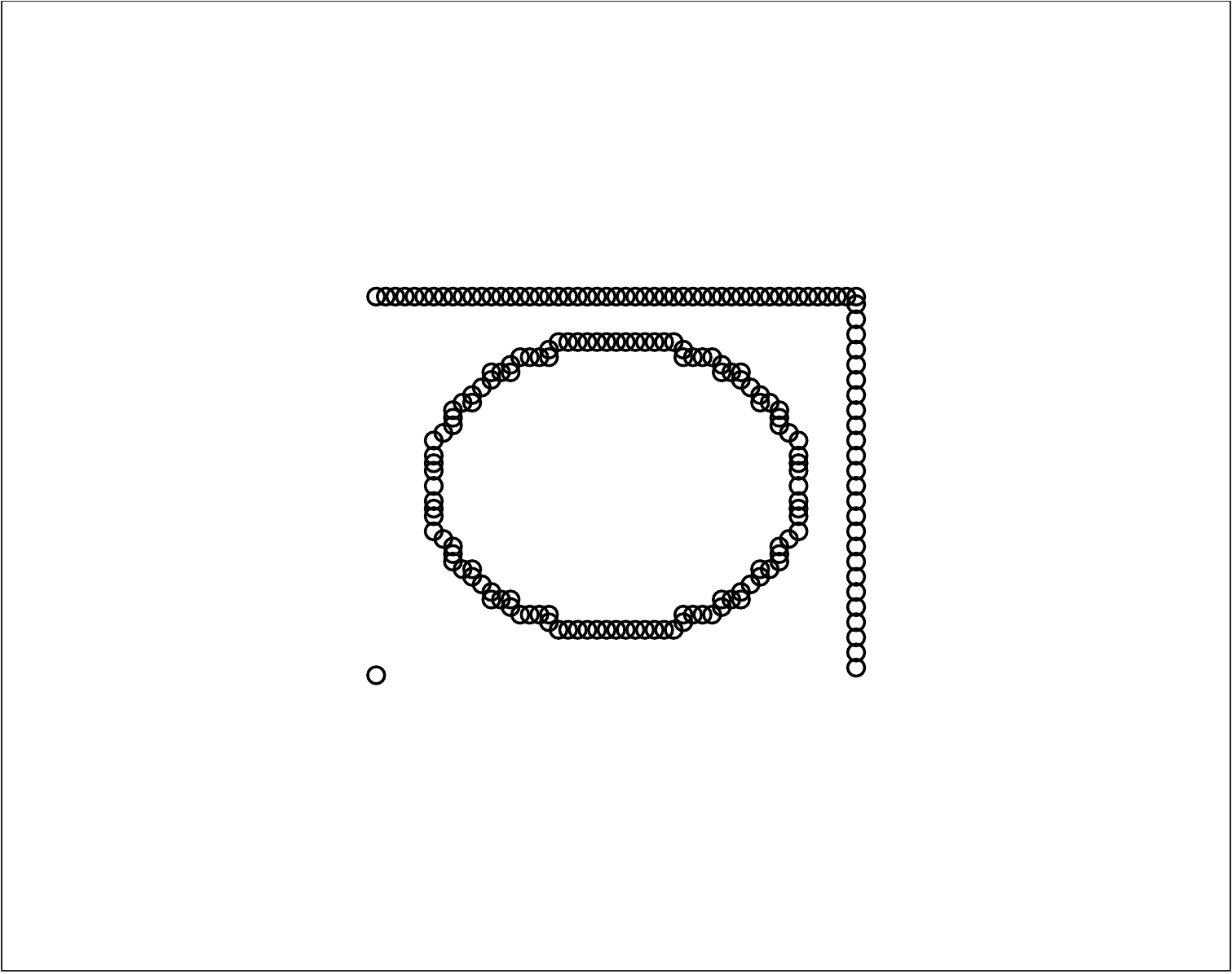}}
\subfigure[DeLENO-3, $k=4$]{\includegraphics[width=0.24\textwidth]{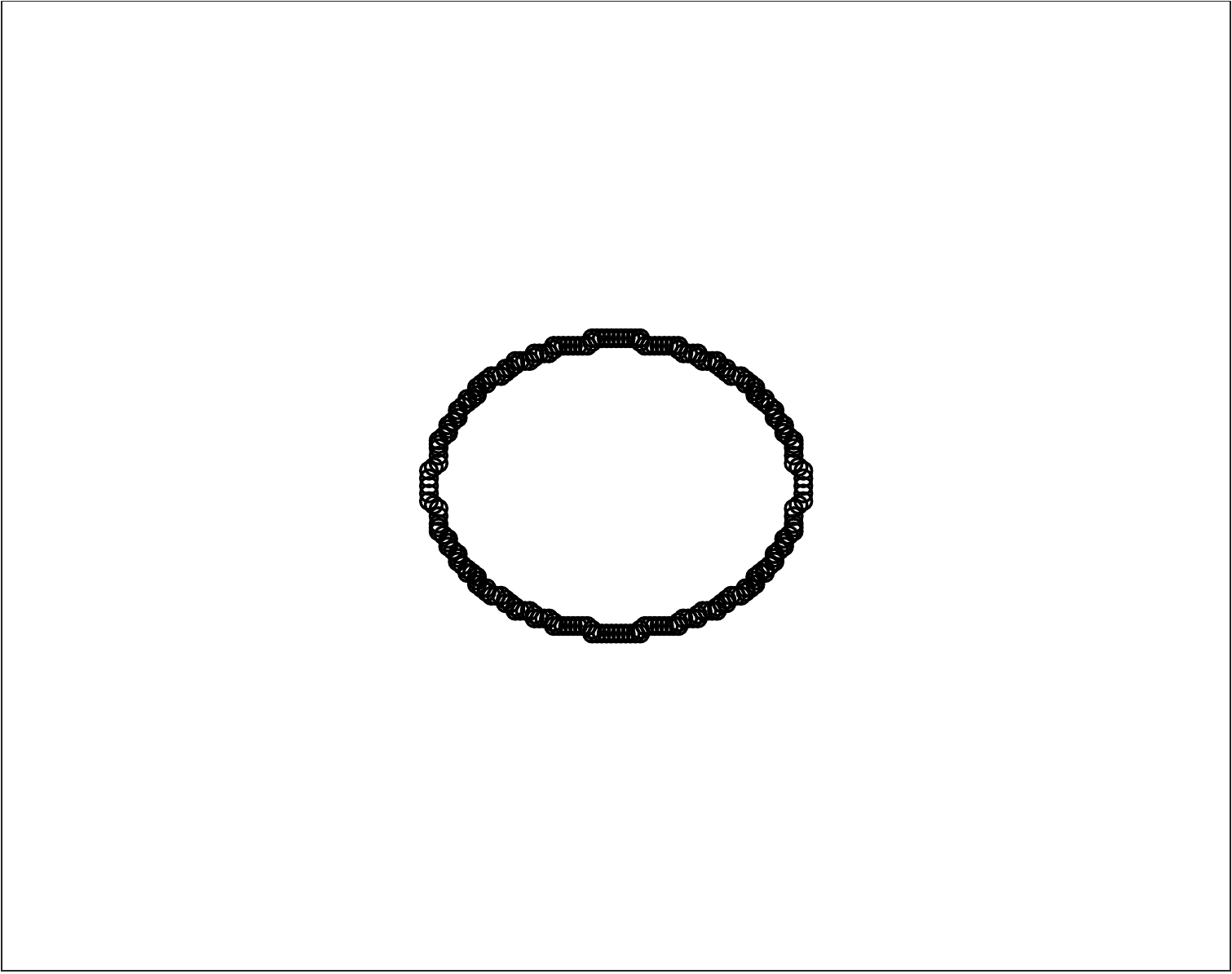}}\\
\subfigure[ENO-4, $k=1$]{\includegraphics[width=0.24\textwidth]{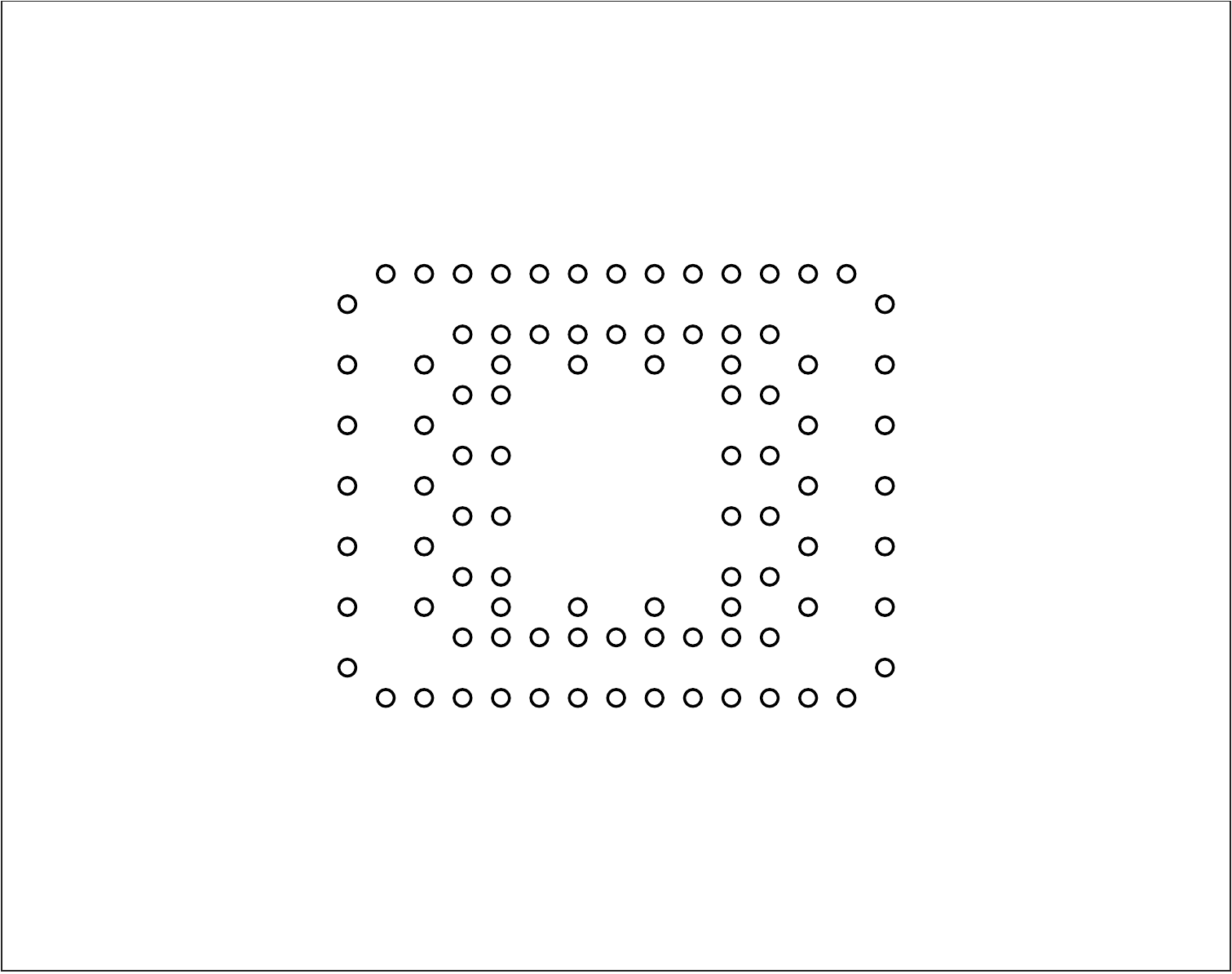}}
\subfigure[ENO-4, $k=2$]{\includegraphics[width=0.24\textwidth]{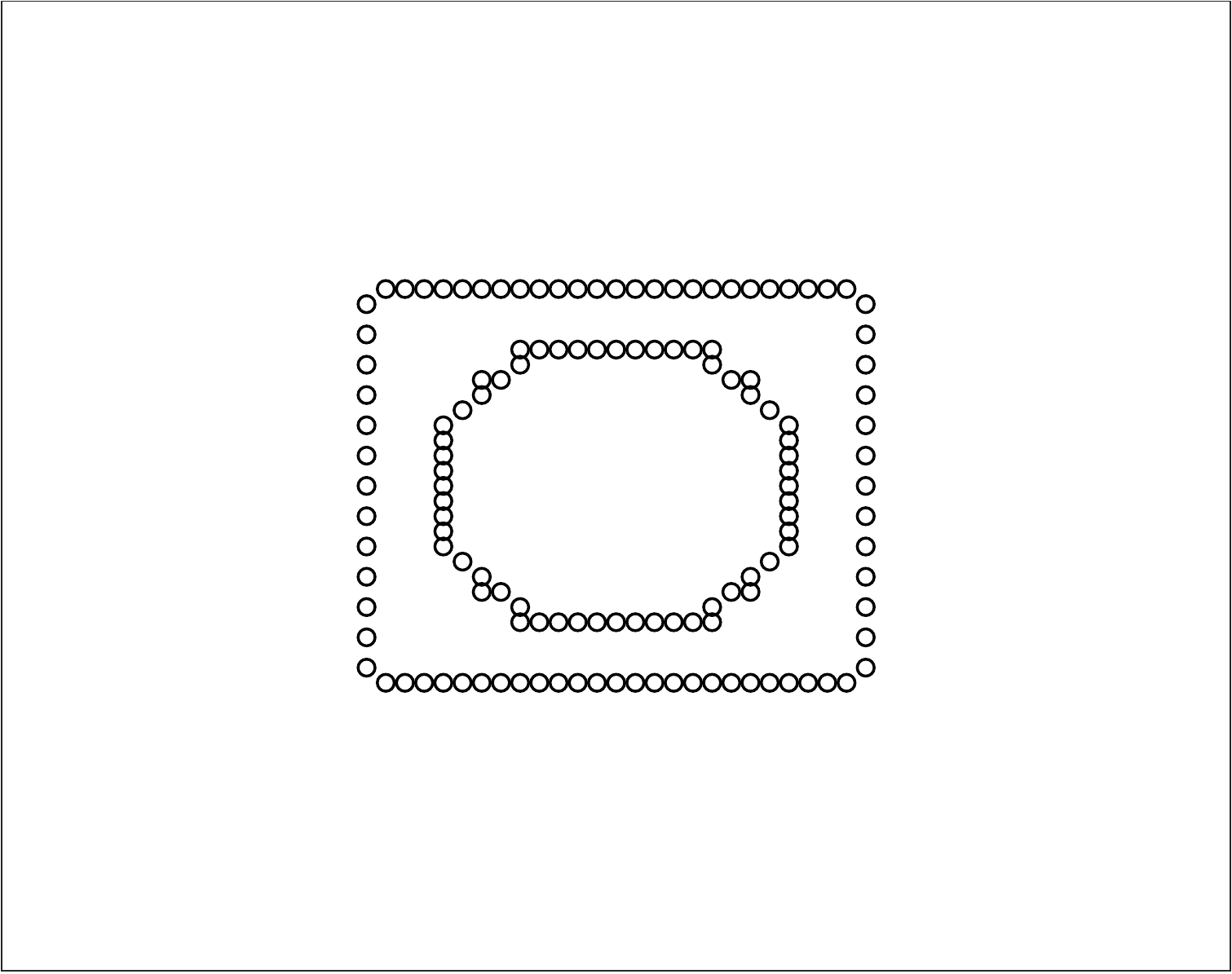}}
\subfigure[ENO-4, $k=3$]{\includegraphics[width=0.24\textwidth]{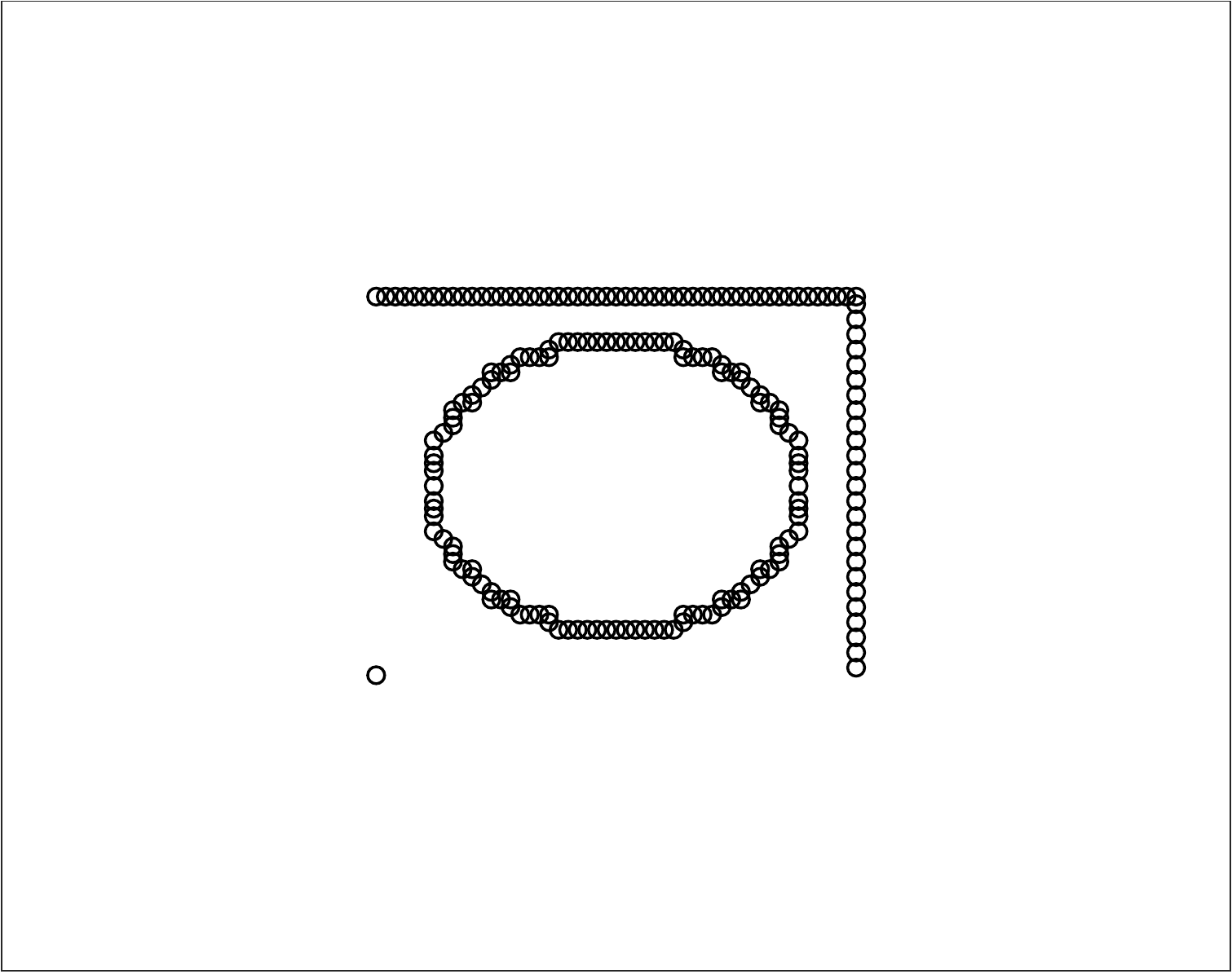}}
\subfigure[ENO-4, $k=4$]{\includegraphics[width=0.24\textwidth]{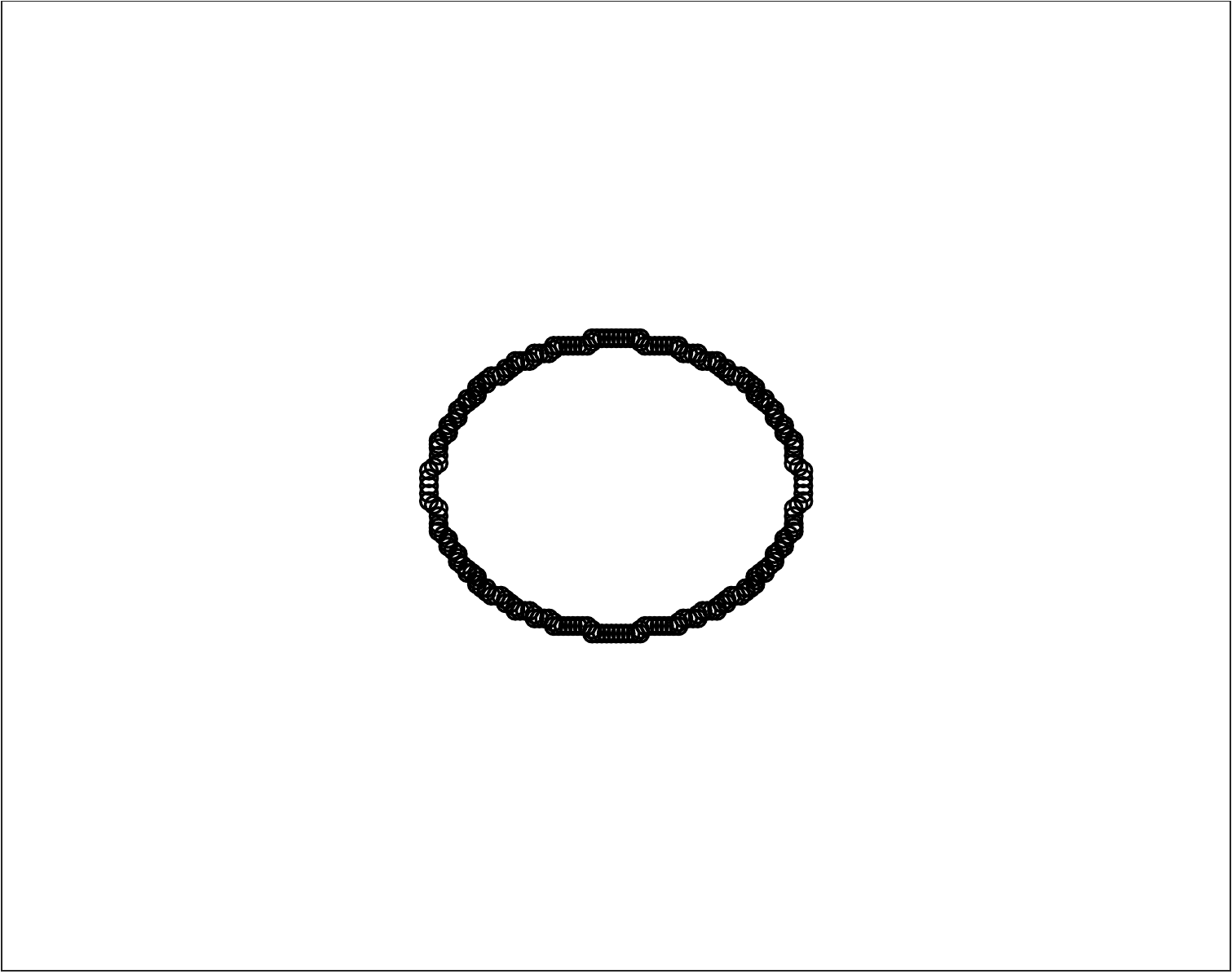}}\\
\subfigure[DeLENO-4, $k=1$]{\includegraphics[width=0.24\textwidth]{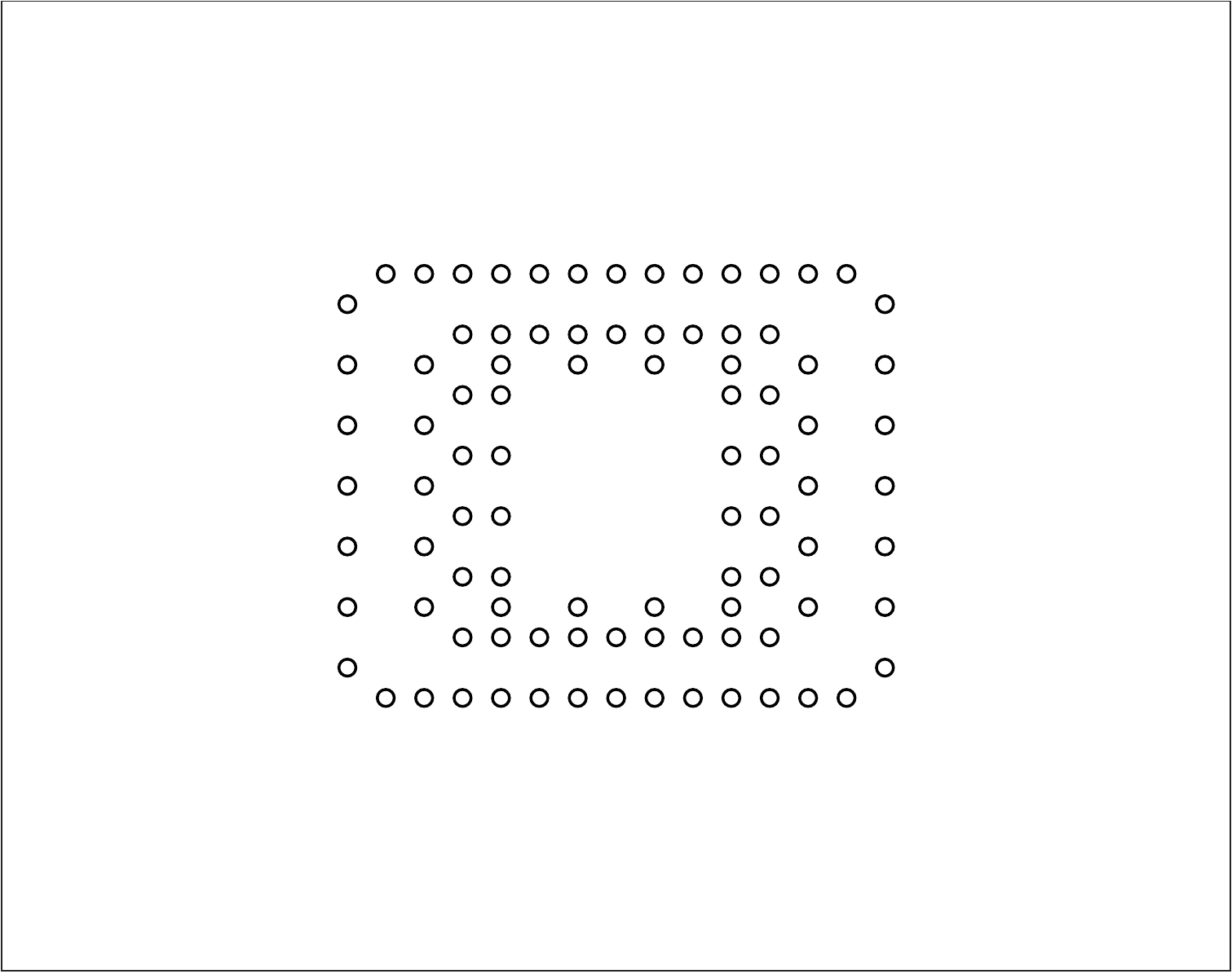}}
\subfigure[DeLENO-4, $k=2$]{\includegraphics[width=0.24\textwidth]{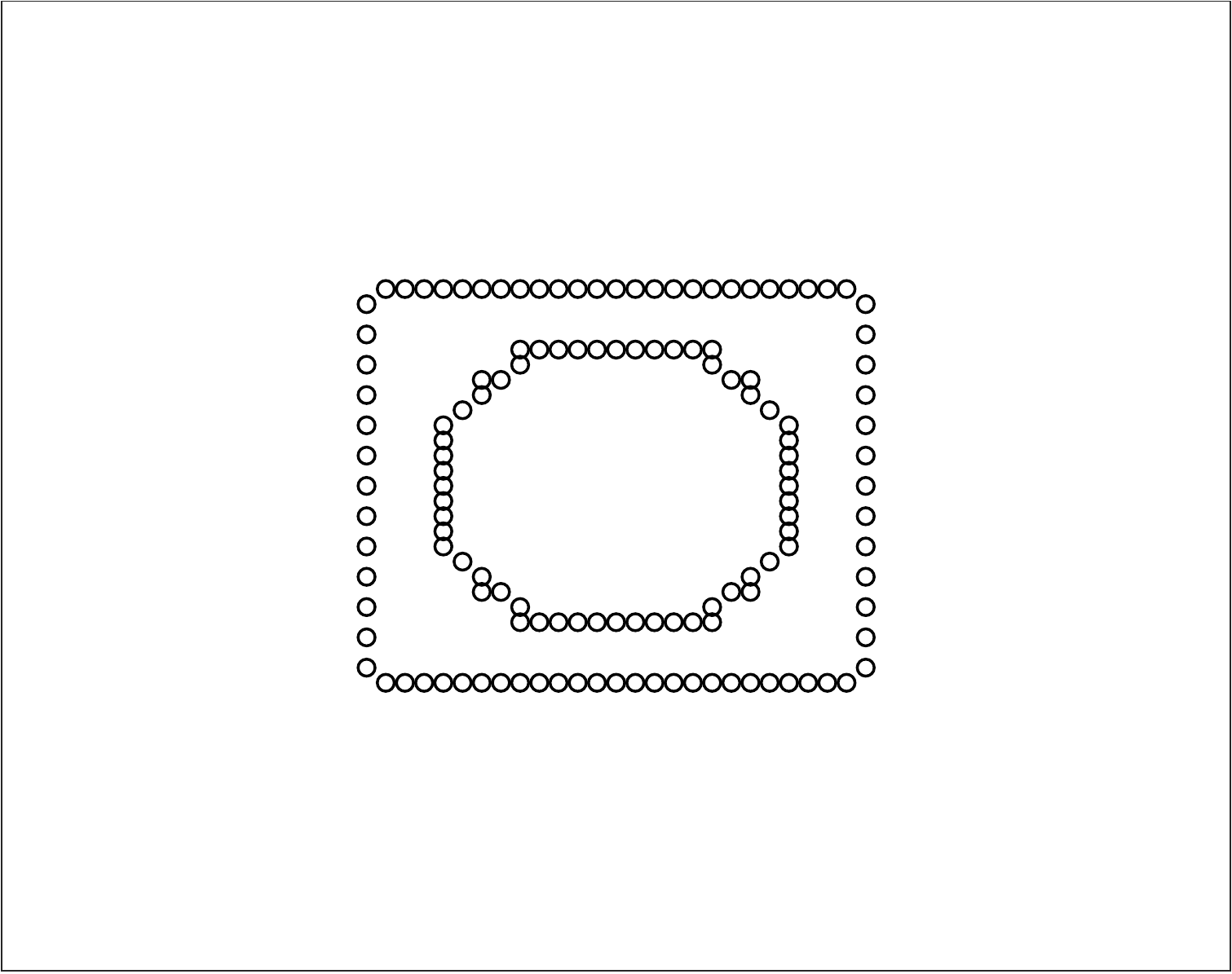}}
\subfigure[DeLENO-4, $k=3$]{\includegraphics[width=0.24\textwidth]{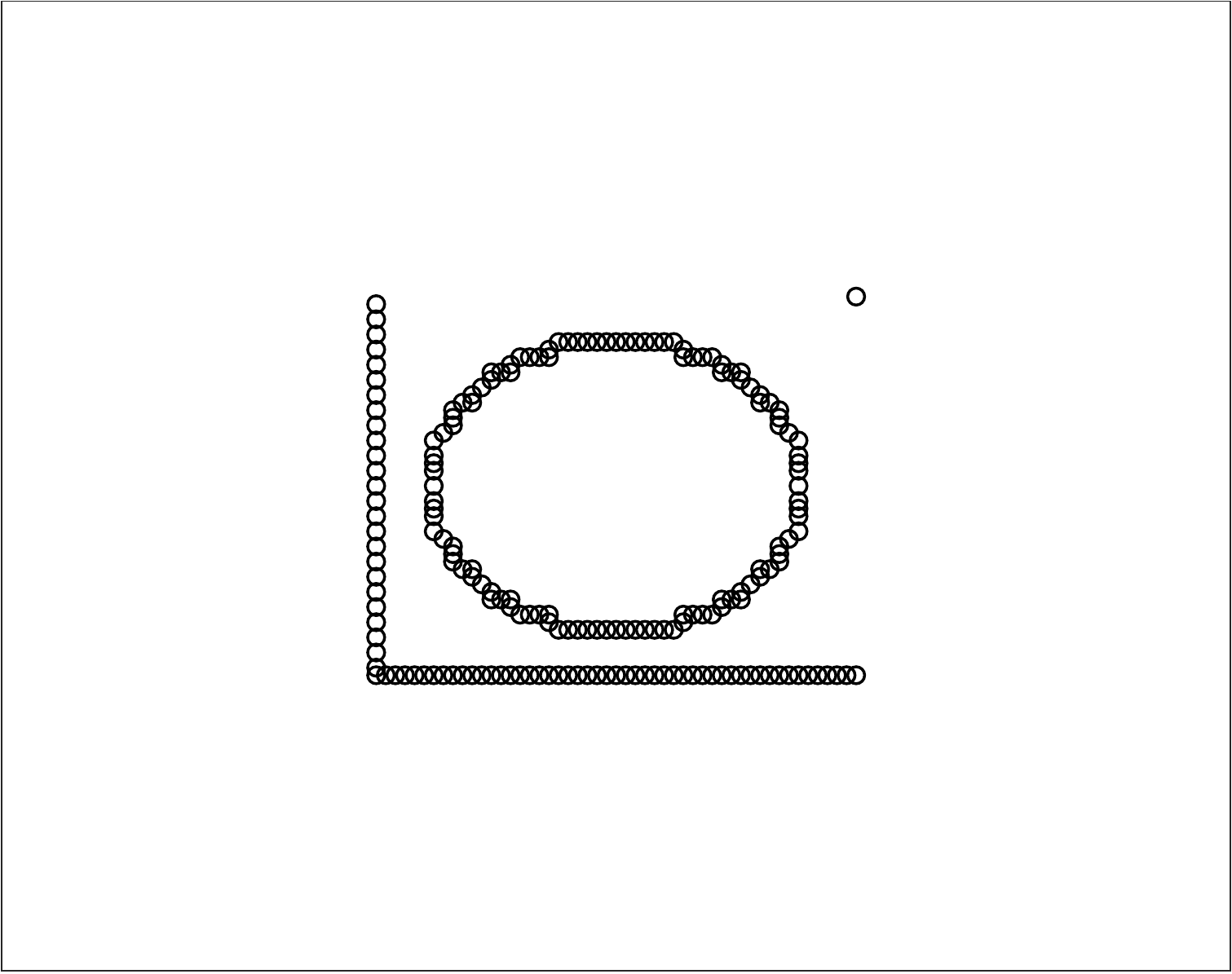}}
\subfigure[DeLENO-4, $k=4$]{\includegraphics[width=0.24\textwidth]{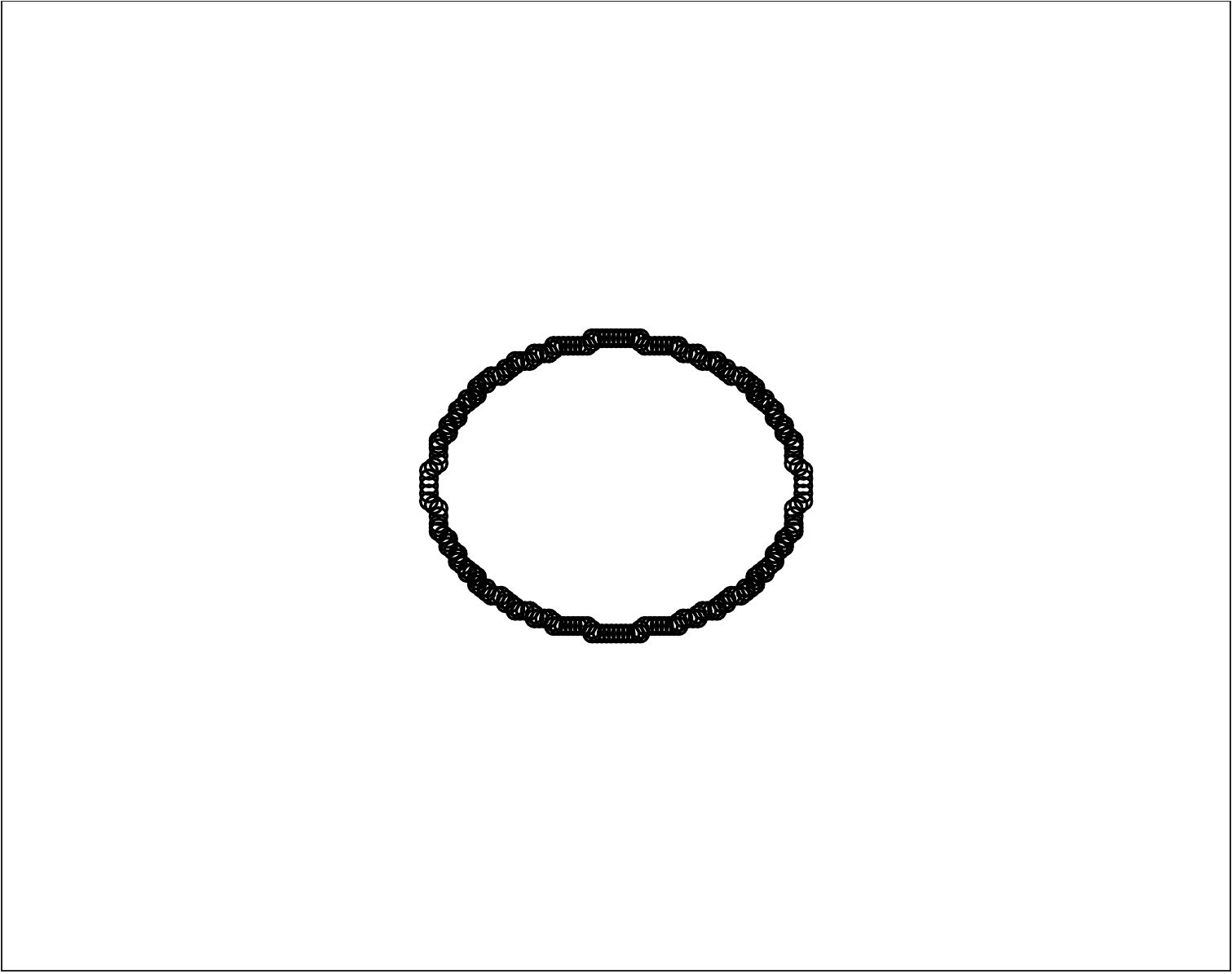}}\\
\caption{Non-zero coefficients $\widehat{d}^k$ for data compression of \eqref{eqn:2d_func} using ENO and DeLENO for mesh level $1\leq k \leq 4$. }
\label{fig:2d_comp_dk}
\end{center}
\end{figure}

\subsubsection{Conservation laws}
We compare the performance of ENO and DeLENO reconstruction, when used to approximate solutions of conservation laws. 
We work in the framework of  high-order finite difference schemes with flux-splitting and we use a fourth-order Runge-Kutta scheme for the time integration. 

As an example, we consider the system of conservation laws governing compressible flows given by
\begin{equation*}
   \partial_t \begin{pmatrix} \rho \\ v \\ p \end{pmatrix} + \partial_x \begin{pmatrix} \rho v \\ \rho v^2 + p \\ (E +p) v \end{pmatrix} = 0, \qquad  E = \frac{1}{2} \rho v^2 + \frac{p}{\gamma -1},
\end{equation*}
where $\rho, v$ and $p$ denote the fluid density, velocity and pressure, respectively. The quantity $E$ represents the total energy per unit volume 
where $\gamma=c_p/c_v$ is the ratio of specific heats, chosen as $\gamma=1.4$ for our simulations.
We consider the shock-entropy problem \cite{SHU89b}, which describes the interaction of a right moving shock with smooth oscillatory waves. The initial conditions for this test case are prescribed as
\begin{equation*}
    (\rho, \ v, \ p ) = \begin{cases} (3.857143,\ 2.629369,\ 10.33333) & \quad \text{if } x < -4 \\ (1+0.2 \sin(5x),\ 0,\ 1) & \quad \text{if } x > -4 \end{cases},
\end{equation*}
on the domain $[-5,5]$. Due to the generation of high frequency physical waves, we solve the problem on a fine mesh with $N=200$ cells up to $T_f = 1.8$ with $\mathrm{CFL = 0.5}$. A reference solution is obtained with ENO-4 on a mesh with $N=2000$ cells.  As can be seen in Figure \ref{fig:se_soln}, ENO-$\pdeg$ and DeLENO-$\pdeg$ perform equally well depending on the order $\pdeg$.

\begin{figure}[!htbp] 
\begin{center}
\subfigure[$\pdeg=2$]{\includegraphics[width=0.32\textwidth]{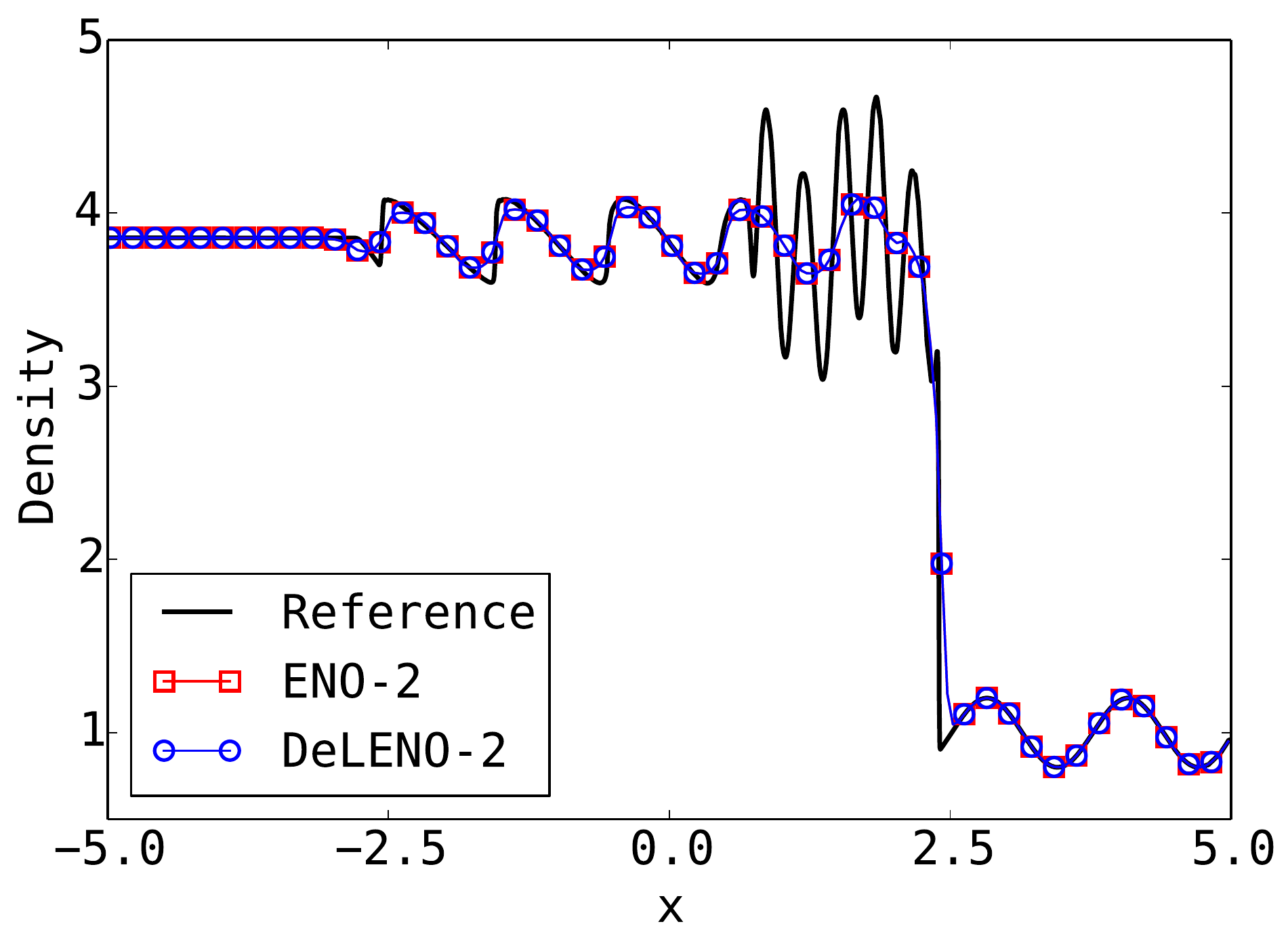}}
\subfigure[$\pdeg=2$]{\includegraphics[width=0.32\textwidth]{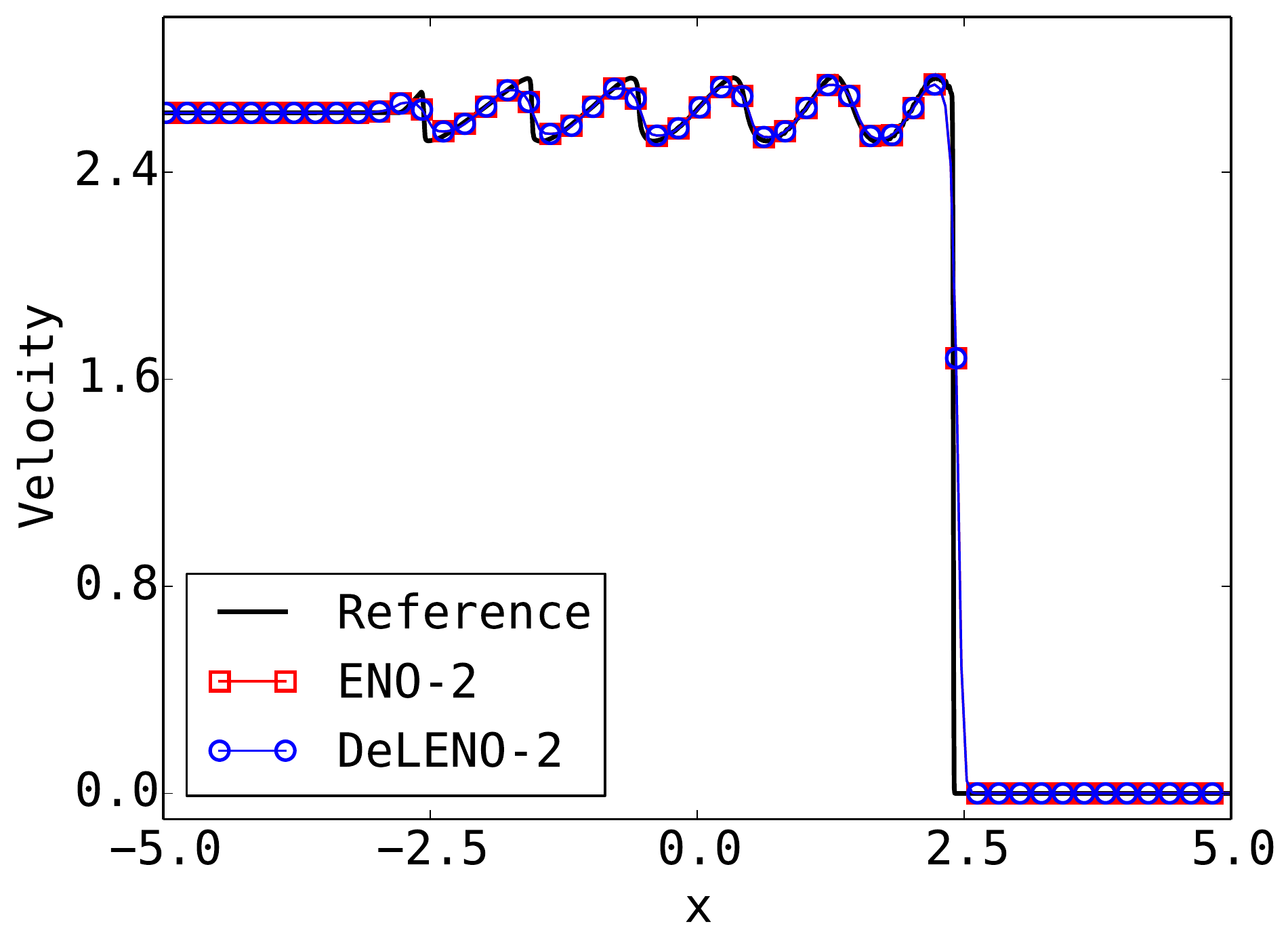}}
\subfigure[$\pdeg=2$]{\includegraphics[width=0.32\textwidth]{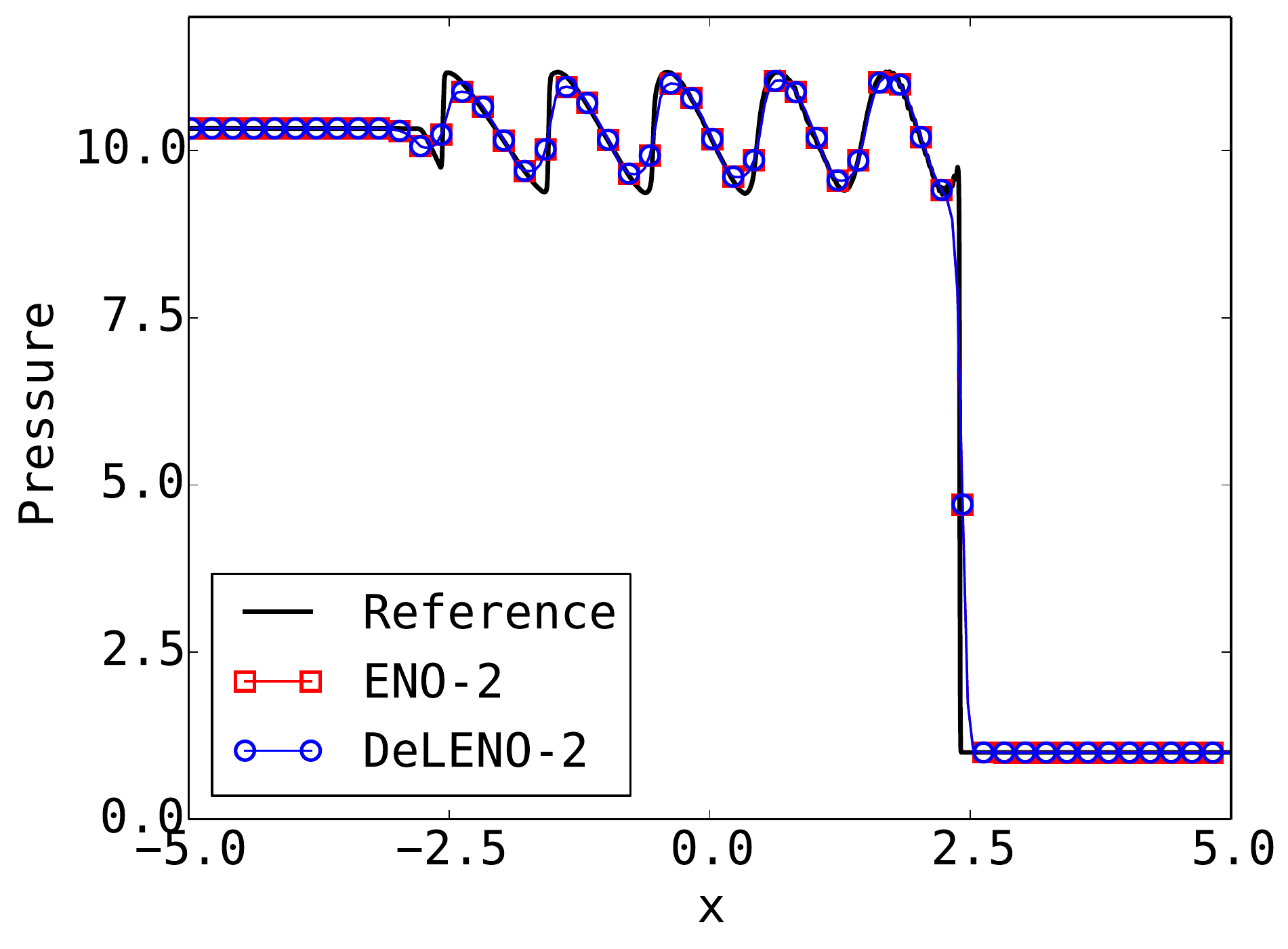}}\\
\subfigure[$\pdeg=3$]{\includegraphics[width=0.32\textwidth]{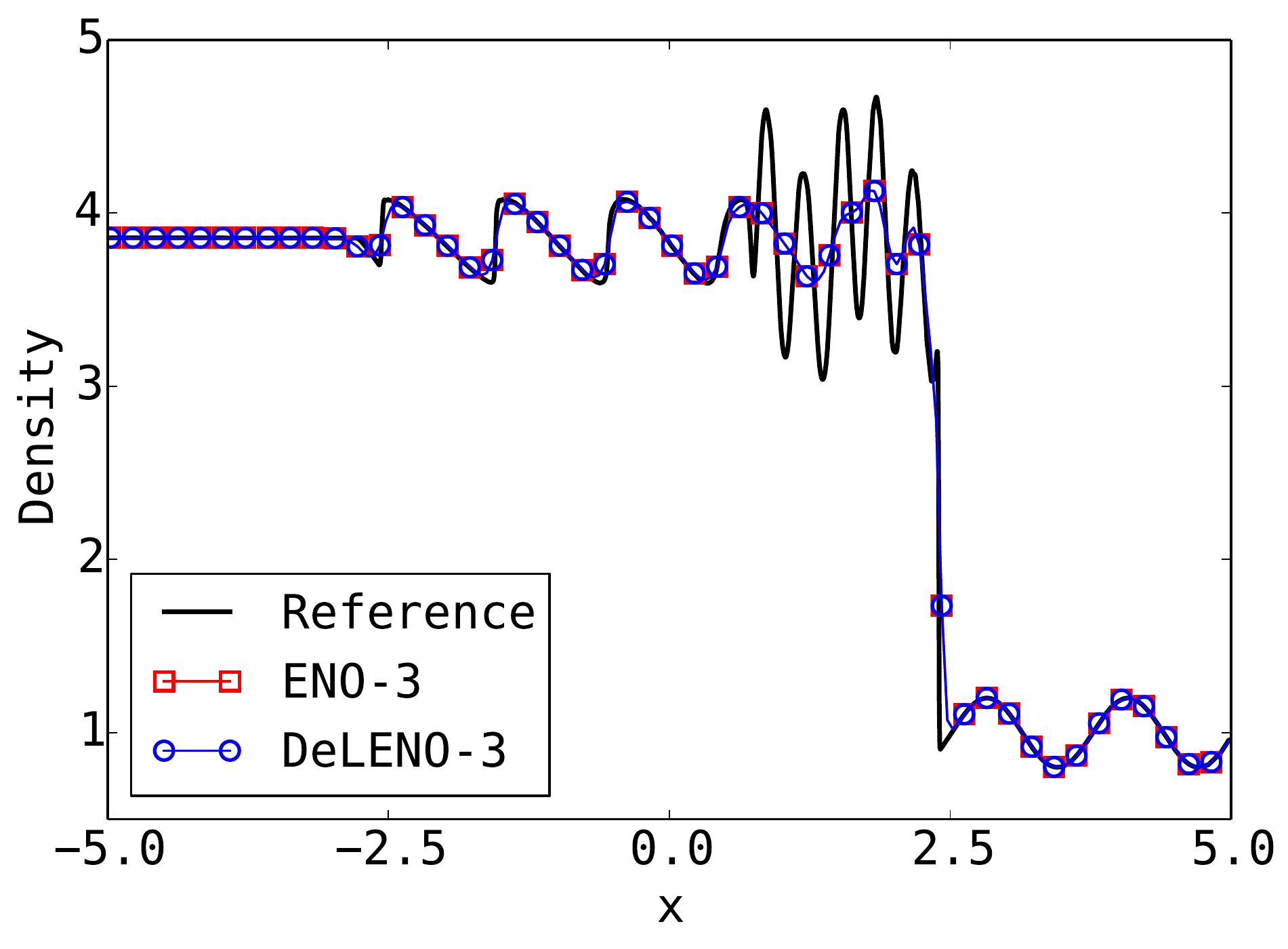}}
\subfigure[$\pdeg=3$]{\includegraphics[width=0.32\textwidth]{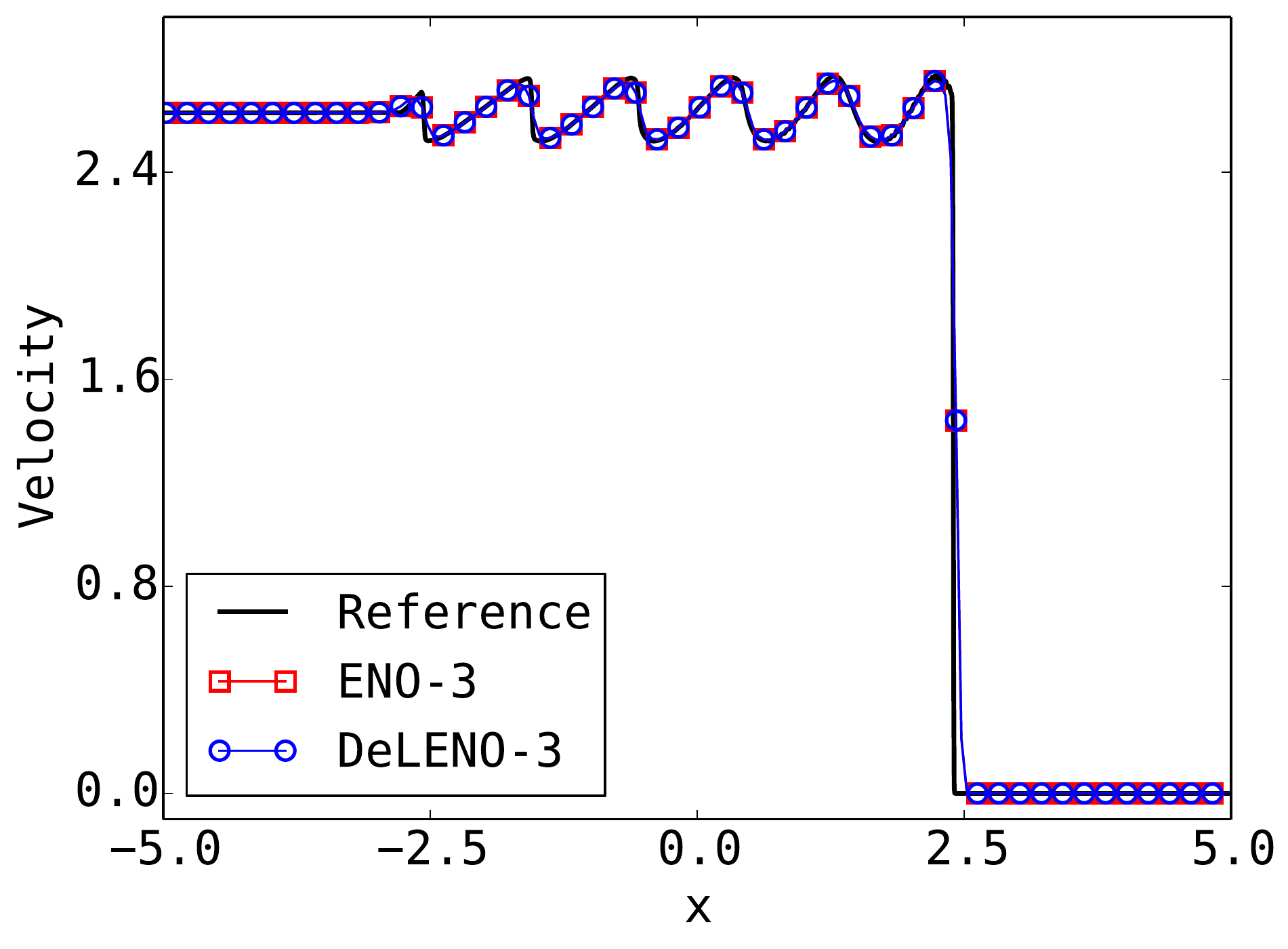}}
\subfigure[$\pdeg=3$]{\includegraphics[width=0.32\textwidth]{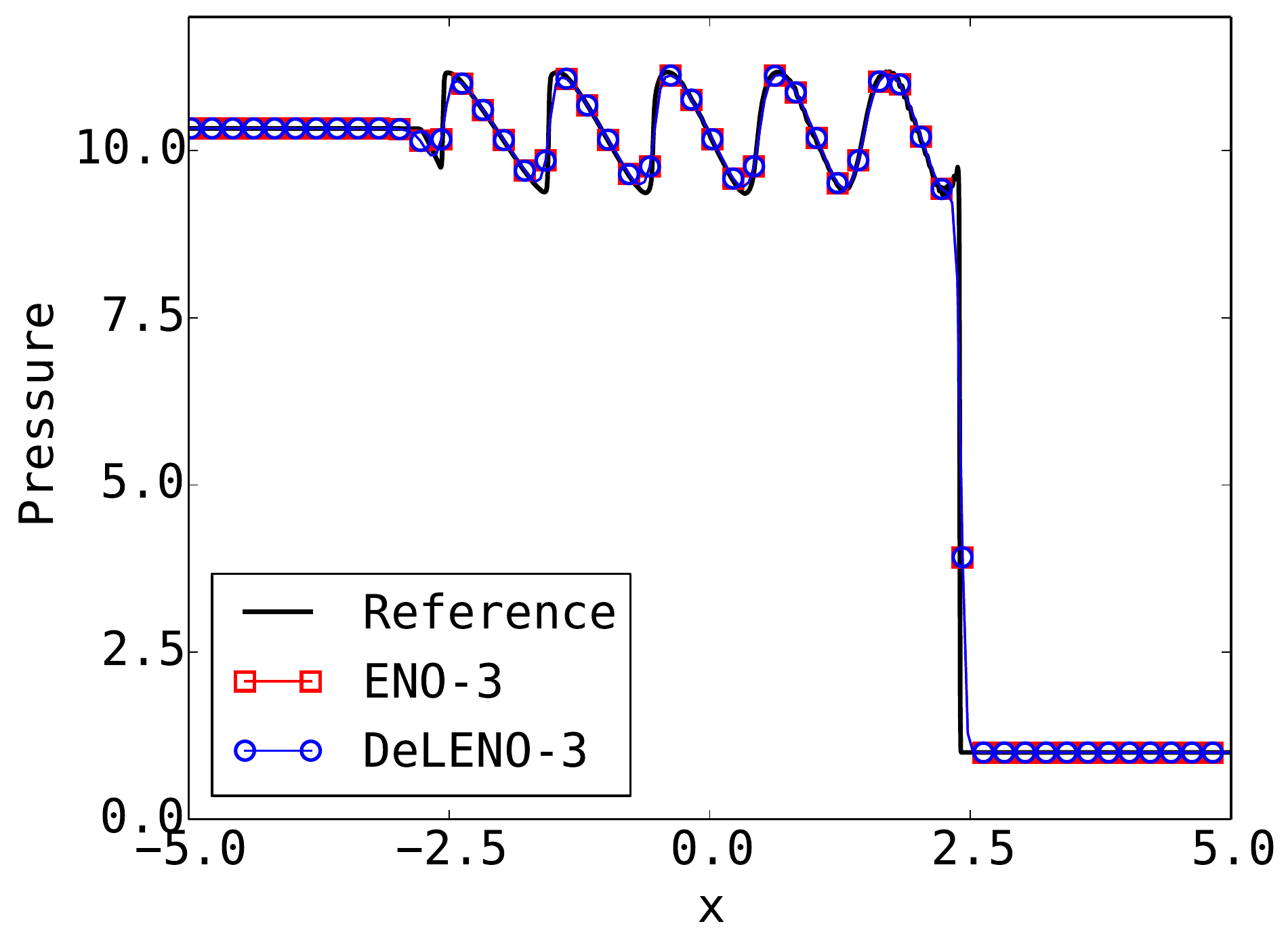}}
\caption{Solution for Euler shock-entropy problem with ENO-$\pdeg$ and DeLENO-$\pdeg$ on a mesh with $N=200$ cells.}
\label{fig:se_soln}
\end{center}
\end{figure}

Next we solve the Sod shock tube problem \cite{SOD78}, whose initial conditions are given by
\[
(\rho, \ v, \ p ) = \begin{cases} (1,\ 0,\ 1) & \quad \text{if } x < 0 \\ (0.125,\ 0,\ 0.1) & \quad \text{if } x > 0 \end{cases},
\]
on the domain $[-5,5]$. The solution consists of a shock wave, a contact discontinuity and a rarefaction. The mesh is descretized with $N=50$ cells and the problem is solved till $T_f = 2$ with a $\mathrm{CFL = 0.5}$. The solutions obtained with ENO-$\pdeg$ and DeLENO-$\pdeg$ are identical, as depicted in Figure \ref{fig:sod_soln}. 

\begin{figure}[!htbp] 
\begin{center}
\subfigure[$\pdeg=2$]{\includegraphics[width=0.32\textwidth]{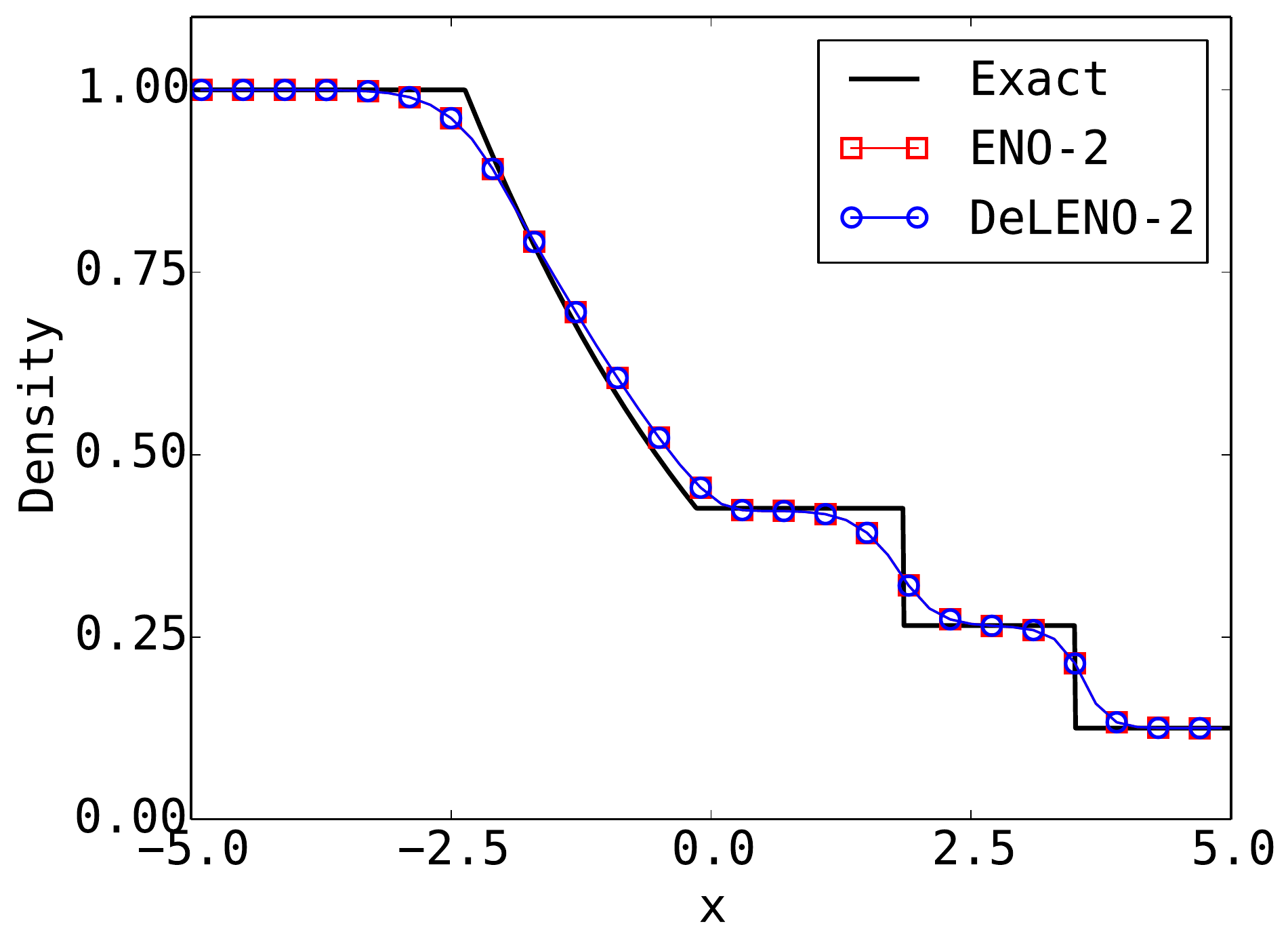}}
\subfigure[$\pdeg=2$]{\includegraphics[width=0.32\textwidth]{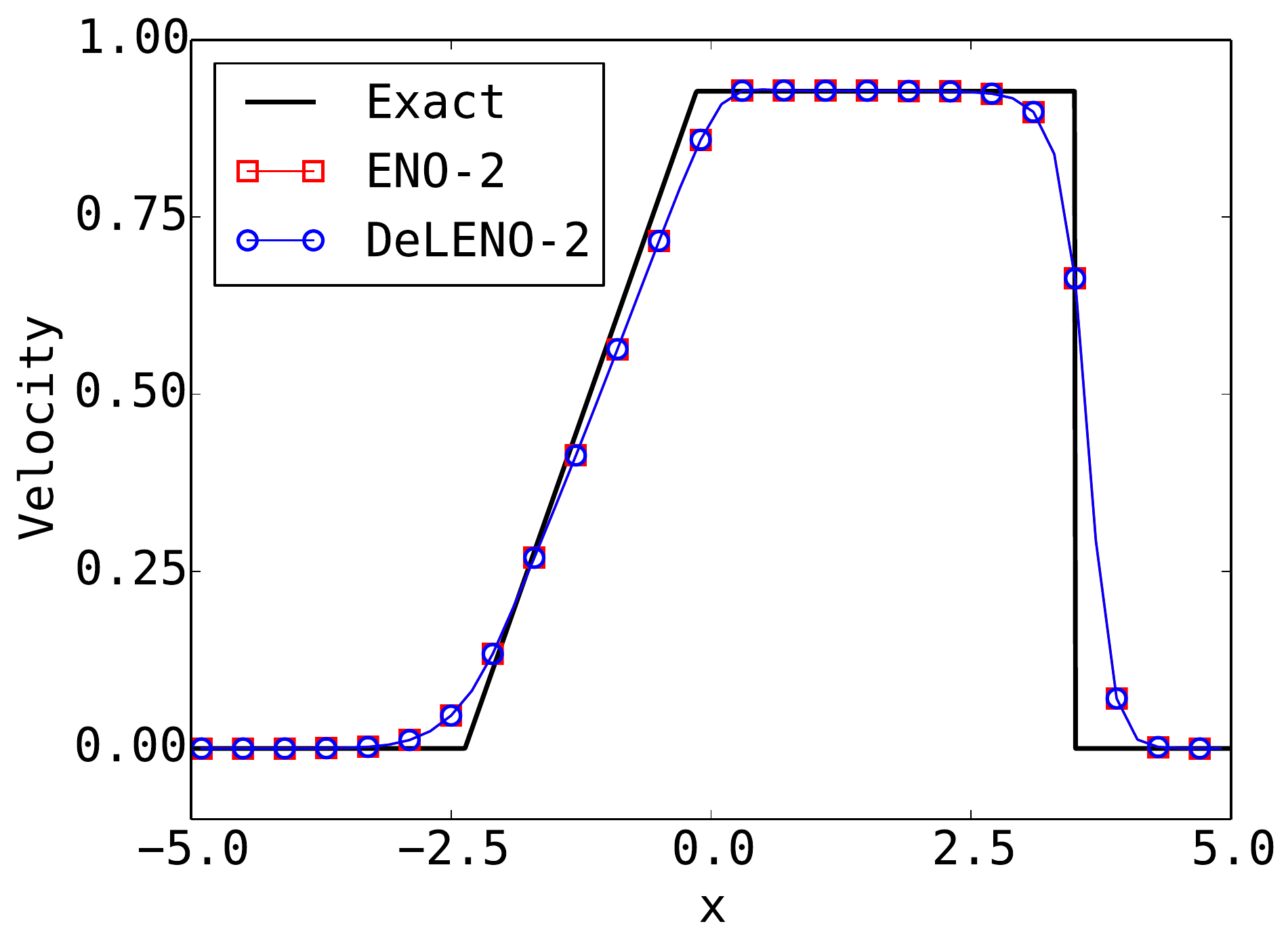}}
\subfigure[$\pdeg=2$]{\includegraphics[width=0.32\textwidth]{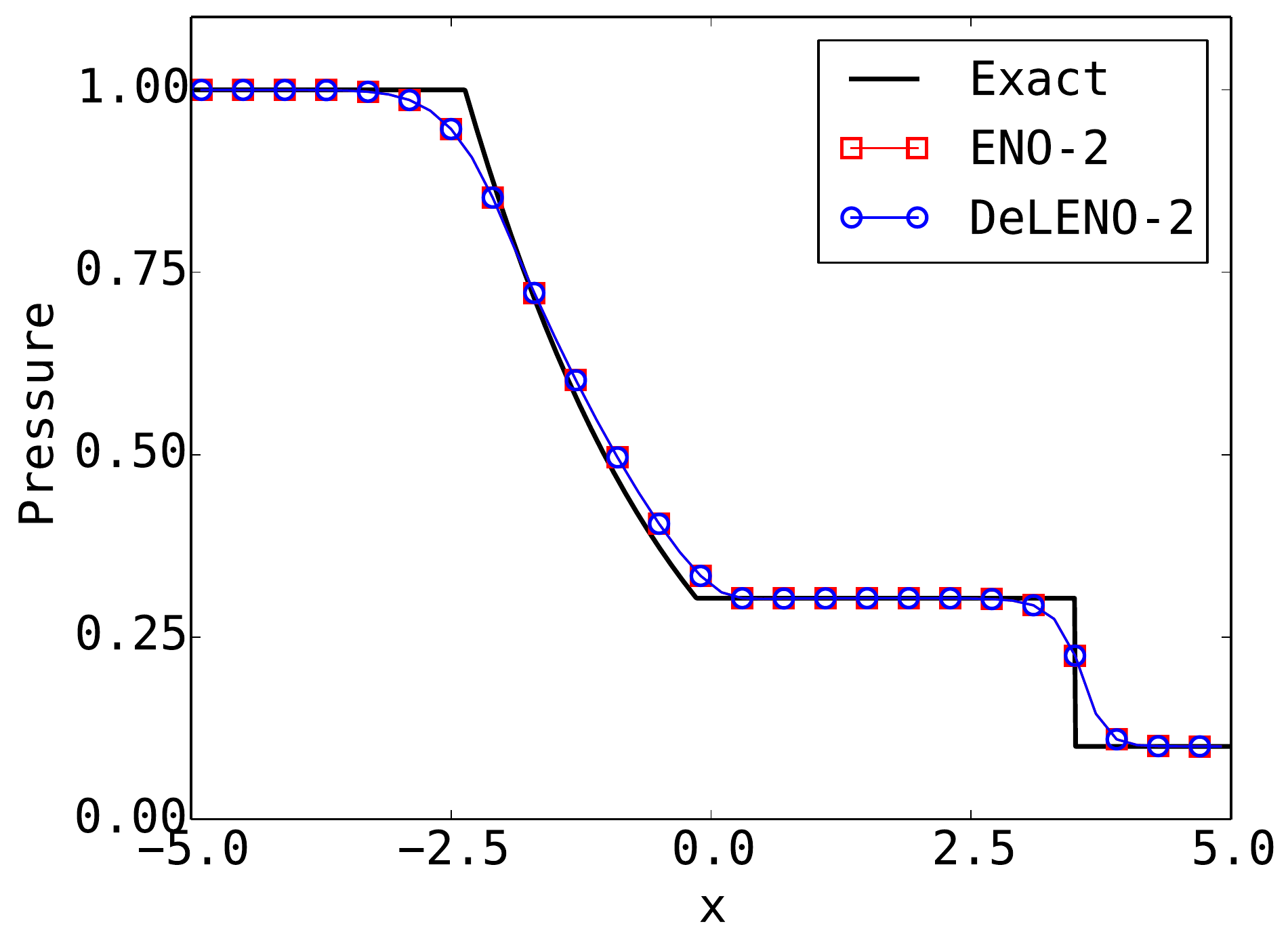}}\\
\subfigure[$\pdeg=3$]{\includegraphics[width=0.32\textwidth]{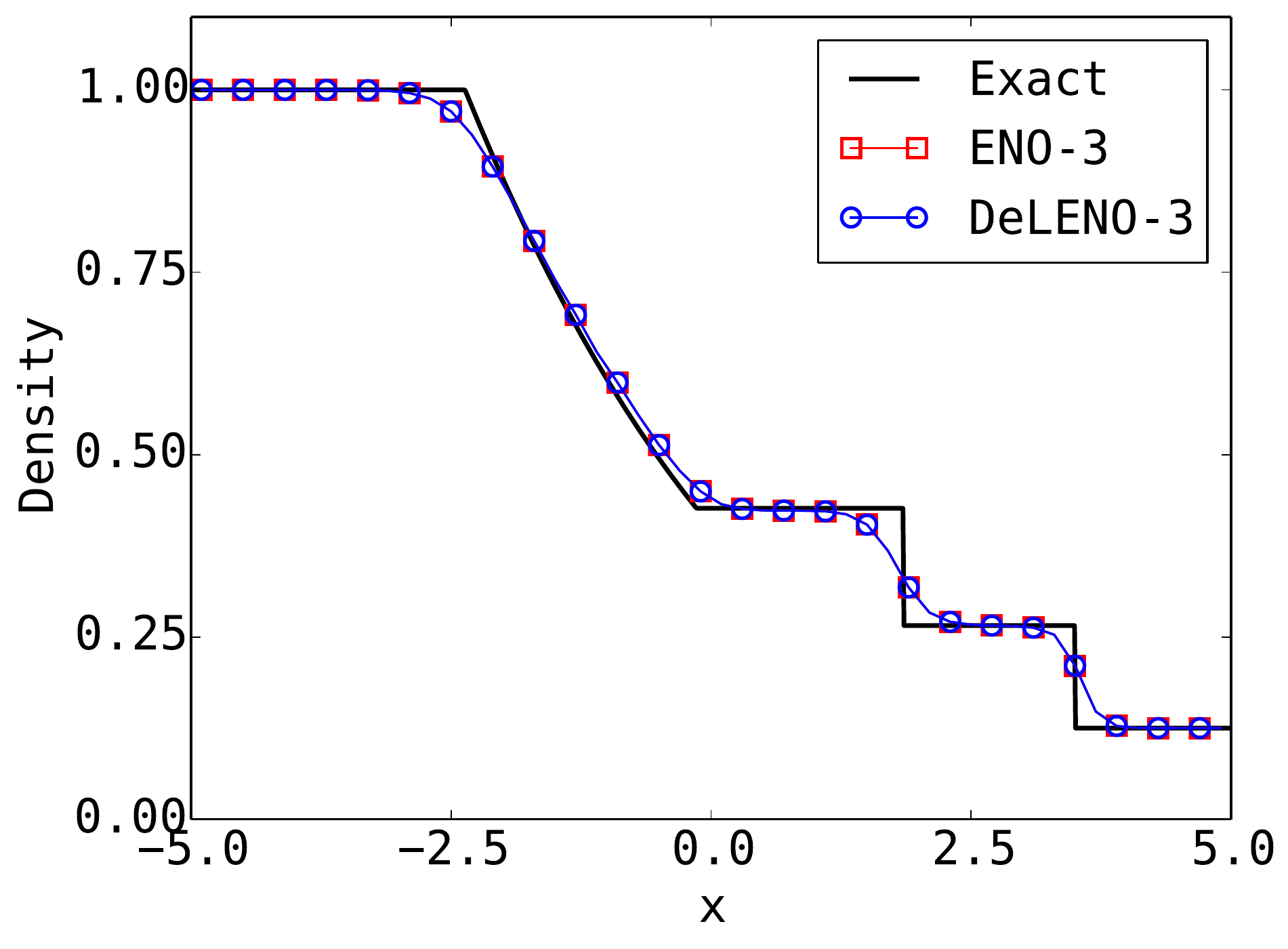}}
\subfigure[$\pdeg=3$]{\includegraphics[width=0.32\textwidth]{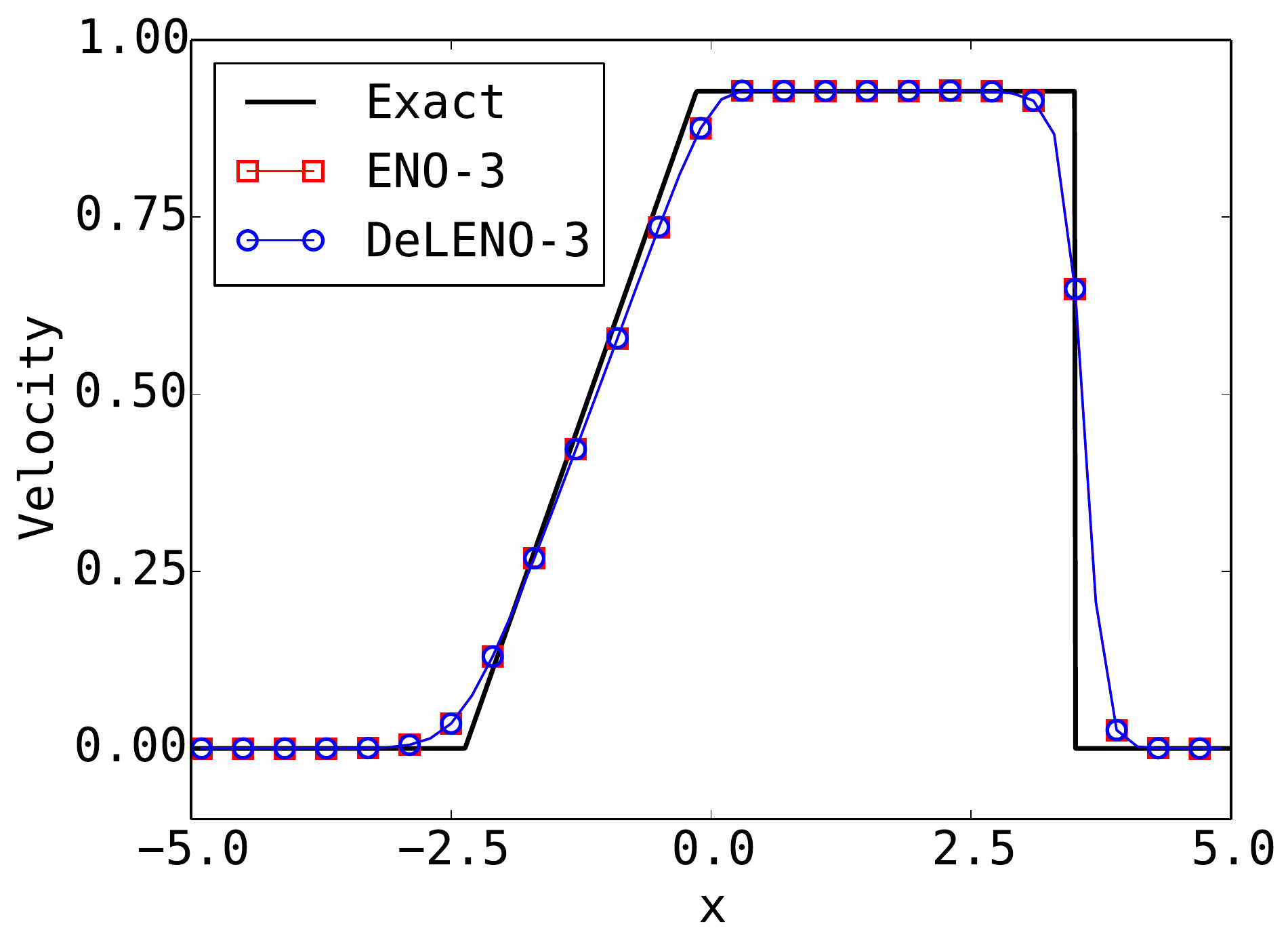}}
\subfigure[$\pdeg=3$]{\includegraphics[width=0.32\textwidth]{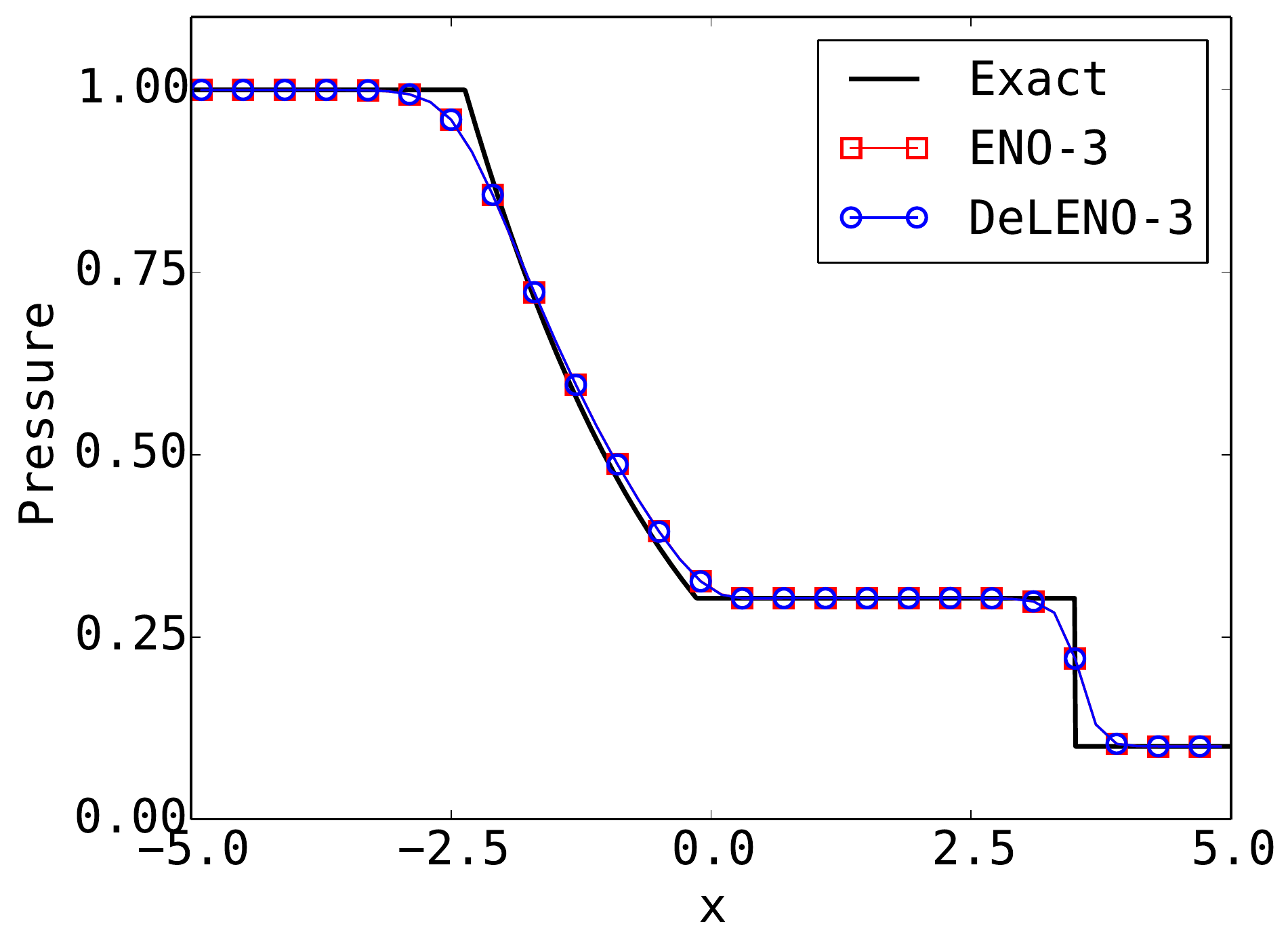}}
\caption{Solution for Euler Sod shock tube problem with ENO-$\pdeg$ and DeLENO-$\pdeg$ on a mesh with $N=50$ cells.}
\label{fig:sod_soln}
\end{center}
\end{figure}

\section{Discussion}
In this paper, we considered efficient interpolation of rough or piecewise smooth functions. A priori, both deep neural networks (on account of universality) and the well-known \emph{data dependent} ENO (and ENO-SR) interpolation procedure are able to interpolate rough functions accurately. We proved here that the ENO interpolation (and the ENO reconstruction) procedure as well as a variant of the second-order ENO-SR procedure can be cast as deep ReLU neural networks, at least for univariate functions. This equivalence provides a different perspective on the ability of neural networks in approximating functions and reveals their enormous expressive power as even a highly non-linear, data-dependent procedure such as ENO can be written as  a ReLU neural network. 

On the other hand, the impressive function approximation results, for instance of \cite{YAROTSKY17,YZ1}, might have limited utility for functions in low dimensions, as the neural network needs to be trained for every function that has to be interpolated. By interpreting ENO (and ENO-SR) as a neural network, we provide a natural framework for recasting the problem of interpolation in terms of pre-trained neural networks such as DeLENO, where the input vector of sample values are transformed by the network into the output vector of interpolated values. Thus, these networks are trained once and do not need to retrained for every new underlying function. This interpretation of ENO as a neural network allows us to possibly extend ENO type interpolations into several space dimensions on unstructured grids.

\bibliographystyle{abbrv}

\bibliography{ref}

\appendix

\section{ENO reconstruction}\label{sec:ENO-reconstruction}

In this section we present ENO reconstruction, which is very similar to ENO interpolation as presented in Section \ref{sec:ENO}. Let $V$ be a function on $[c,d]$. For the sake of completeness, we repeat the main steps of the algorithm for reconstruction purposes.

We define a uniform mesh $\Tau$ on $[c,d]$ with $N$ cells, 
\begin{equation}\label{eqn:rec_partitions}
\Tau = \{I_i\}_{i=1}^{N}, \: I_i = [x_{i-\frac{1}{2}}, x_{i+\frac{1}{2}}], \: \left\{x_i = c + (2i-1) \frac{h}{2}\right\}_{i=0}^{N}, \: h = \frac{(d-c)}{N}, 
\end{equation}
where $x_i$ and $x_{i\pm\frac{1}{2}}$ denote the cell center and cell interfaces of the cell $I_i$, respectively. We are given the cell averages
\begin{equation*}
    \overline{V}_i = \frac{1}{h} \int_{x_\imh}^{x_\iph} V(\xi) \ud \xi, \quad 1 \leq i \leq N
\end{equation*}
 and we define $\overline{V}_{-\pdeg+2}, ...,\overline{V}_{0}$ and $\overline{V}_{N+1},...,\overline{V}_{N + \pdeg-1}$ to be ghost  values that need to be suitably prescribed. 
The goal is to find a local interpolation operator $\mathcal{I}^{h}_i$ such that
\begin{equation*}
    \norm{\mathcal{I}^{h}_i V-V}_{\infty,I_i}=O(h^p) \textrm{ for } h\to 0.
\end{equation*}
For this purpose, let $\hat{V}$ be the primitive function of $V$ and note that we have access to the value of $\hat{V}$ at the cell interfaces,
\begin{equation*}
    \hat{V}(x_{i+\frac{1}{2}}) = h\sum_{j=0}^i \overline{V}_j \quad \textrm{where} \quad \hat{V}(x) = \int_c^x V(\xi) \ud \xi.
\end{equation*}
Next let $P_i$ be the unique polynomial of degree $p$ that agrees with $\hat{V}$ on a chosen set of $p+1$ cell interfaces that includes $x_{i-\frac{1}{2}}$ and $x_{i+\frac{1}{2}}$. 
The ENO reconstruction procedure considers the stencil sets 
\begin{equation*}
    \mathcal{S}^r_i = \{x_{i-\frac{1}{2}-r+j}\}_{j=0}^{p}, \quad 0\leq r \leq p-1,
\end{equation*}
where $r$ is called the (left) stencil shift. The smoothest stencil is then selected based on the local smoothness of $f$ using Newton's undivided differences. Algorithm \ref{alg:eno_fd_select} describes how the stencil shift $r_i$ corresponding to this stencil can be obtained. Note that $r_i$ uniquely defines the polynomial $P_i$. We then define $\mathcal{I}^{h}_i V$ to be the first derivative of $P_i$, one can check that this polynomial is indeed a $p$-th order accurate approximation. Note that the interpolants on two adjacent intervals do not need to agree on the mutual cell interface. 

\begin{algorithm}
\KwIn{ENO order $\pdeg$, input array ${\bf{\Delta^0}} = \{ \overline{V}_{i+j}\}_{j=-p+1}^{\pdeg-1}$, for any $1 \leq i \leq N$.}
\KwOut{Stencil shift $r$.}
\textit{Evaluate Newton undivided differences:} 


\For{$j=1$ \KwTo $\pdeg-1$}{
${\bf{\Delta^j}} = {\bf{\Delta^{j-1}}}[2:\textrm{end}] - {\bf{\Delta^{j-1}}}[1:\textrm{end}-1]$
}
\textit{Find shift:} 

$r=0$

\For{$j=1$ \KwTo $\pdeg-1$}{
\If{$|{\bf{\Delta^j}}[\pdeg-1-r] | < |{\bf{\Delta^j}}[\pdeg-r] |$}{ $r = r+1$}
}
\Return $r$
\caption{ENO reconstruction stencil selection}
\label{alg:eno_fd_select}
\end{algorithm}

In order to implement an ENO scheme, one only needs the values of $\mathcal{I}^{h}_i V$ at the cell interfaces $x_{i-\frac{1}{2}}$ and $x_{i+\frac{1}{2}}$. Analogously to equation (\ref{eqn:eno_int}), these can be directly obtained by calculating
\begin{equation}\label{eqn:eno_rec}
    \mathcal{I}^{h}_i V(x_{i+\frac{1}{2}}) = \sum_{j=0}^{p-1} \Tilde{C}_{r_i,j}^p \overline{V}_{i-r_i+j} \quad \textrm{and} \quad  \mathcal{I}^{h}_i V(x_{i-\frac{1}{2}}) = \sum_{j=0}^{p-1} \Tilde{C}_{r_i-1,j}^p \overline{V}_{i-r_i+j},
\end{equation}
where $r_i$ is stencil shift corresponding to the smoothest stencil for interval $I_i$ and with the coefficients $\Tilde{C}_{r,j}^p$ listed in Table \ref{tab:coef_rec} in Appendix \ref{app:ENO-coeff}. 

\section{ENO coefficients}\label{app:ENO-coeff}

Table \ref{tab:coef_int} and Table \ref{tab:coef_rec} respectively list the ENO coefficients used for ENO interpolation (Section \ref{sec:ENO}) and ENO reconstruction (Appendix \ref{sec:ENO-reconstruction}). More information on the calculation of these coefficients can be found in \cite{SHU89}. 

\begin{table}[!htbp]
\centering
\begin{tabular}{|c|c|c|c|c|c|}
\hline
$r$                  & $s$ & $j=0$    & $j=1$   & $j=2$    & $j=3$ \\ \hline
\multirow{2}{*}{3} & 0 & 3/8  & 3/4  & -1/8   &  -      \\ \cline{2-6} 
                   & 1 & -1/8  & 3/4 & 3/8 & -     \\ \hline
\multirow{3}{*}{4} & 0 & 5/16  & 15/16  & -5/16   & 1/16     \\ \cline{2-6} 
                   & 1 & -1/16  & 9/16  & 9/16   & -1/16    \\ \cline{2-6} 
                   & 2 & 1/16   & -5/16 & 15/16  & 5/16   \\ \hline
\end{tabular}\caption{Coefficients for ENO interpolation for $p>2$ used in \eqref{eqn:eno_int}.}
\label{tab:coef_int}
\end{table}

\begin{table}[!htbp]
\centering
\begin{tabular}{|c|c|c|c|c|c|}
\hline
$\pdeg$                & $s$ & $j=0$           & $j=1$            & $j=2$            & $j=3$           \\ \hline
\multirow{3}{*}{2} & -1  & 3/2   & -1/2   &       -           &        -         \\ \cline{2-6} 
                   & 0   & 1/2   & 1/2   &       -            &       -          \\ \cline{2-6} 
                   & 1   & -1/2  & 3/2    &       -           &        -         \\ \hline
\multirow{4}{*}{3} & -1  & 11/6  & -7/6   & 1/3    &       -          \\ \cline{2-6} 
                   & 0   & 1/3  & 5/6    & -1/6   &       -          \\ \cline{2-6} 
                   & 1   & -1/6  &  5/6     & 1/3     &     -            \\ \cline{2-6} 
                   & 2   & 1/3   & -7/6   & 11/6   &        -         \\ \hline
\multirow{5}{*}{4} & -1  & 25/12 & -23/12 & 13/12  & -1/4  \\ \cline{2-6} 
                   & 0   & -1/4  & 13/12  & -5/12  & 1/12  \\ \cline{2-6} 
                   & 1   & -1/12 & 7/12   & 7/12  & -1/12 \\ \cline{2-6} 
                   & 2   & 1/12   & -5/12  & 13/12  & -1/4   \\ \cline{2-6} 
                   & 3   &-1/4  & 13/12  & -23/12 & 25/12 \\ \hline
\end{tabular}
\caption{Coefficients for ENO reconstruction used in \eqref{eqn:eno_rec}.}
\label{tab:coef_rec}
\end{table}

\section{Multi-resolution representation of functions for data compression}\label{app:multi-res}

To describe the multi-resolution representation of functions, we use notations and operators similar to those introduced in \cite{ARAND2000}. We define a sequence of nested uniform meshes $\{\Tau^k\}_{k=0}^K$ on $\Omega = [c,d]$, where
\begin{equation}\label{eqn:partitions}
\Tau^k = \{I_i^k\}_{i=1}^{N_k}, \quad I_i^k = [x^k_{i-1}, x^k_i], \quad \{x_i^k = c + i h_k\}_{i=0}^{N_k}, \quad h_k = \frac{(d-c)}{N_k}, \quad N_k = 2^k N_0, 
\end{equation}
for $0\leq k\leq K$ and where $N_0$ is some positive integer. We call $\{x_i^k\}_{i=0}^{N_k}$ the nodes of the mesh $\Tau^k$. Let $\mathcal{B}_\Omega$ be the set of bounded functions on $\Omega$ and $\mathcal{V}^n$ the space of real-valued finite sequences of length $n$. We define the following operators associated with the various meshes:
\begin{itemize}
\item The {\it discretizer} $D_k: \mathcal{B}_\Omega \mapsto\mathcal{V}^{N_k + 1}$ defined by 
\[
D_k f = \q^k :=  \{q_i^k\}_{i=0}^{N_k} = \{q(x_i^k)\}_{i=0}^{N_k}, \quad  \forall \ q \in \mathcal{B}_\Omega.
\]
\item The {\it reconstructor} $R_k: \mathcal{V}^{N_k + 1} \mapsto \mathcal{B}_\Omega$ satisfying $D_k R_k \q^k = \q^k$ for $\q^k \in \mathcal{V}^{N_k + 1} $. Thus, $(R_k \q^k)(x)$ interpolates the members of $\q^k$ at the nodes of $\Tau^k$.
\item The {\it decimator} $D^{k-1}_k : \mathcal{V}^{N_k + 1}  \mapsto \mathcal{V}^{N_{k-1} + 1} $ defined by $D_k^{k-1} := D_{k-1} R_{k}$. For $q \in \mathcal{B}_\Omega$, we have
\begin{equation}\label{eqn:decimator}
q^{k-1}_i = (D_k^{k-1} \q^k)_i = q_{2i}^k, \quad 0\leq i\leq N_{k-1}.
\end{equation}
In other words, the decimator helps in extracting the function values on a given mesh from a finer one. 
\item The {\it predictor} $P_{k-1}^k : \mathcal{V}^{N_{k-1} + 1}  \mapsto \mathcal{V}^{N_{k} + 1} $ defined by $P^k_{k-1} := D_k R_{k-1}$. The predictor tries to recover the function values $\q^k$ from the coarser data $\q^{k-1}$, for $q \in \mathcal{B}_\Omega$. 
\end{itemize}
The prediction error is given by
\[
e^k_i = q^k_i - (P^k_{k-1} q^{k-1})_i, \quad 0\leq i \leq N_k.
\]
Clearly $e^k_{2i} = 0$ for $0 \leq i \leq N_{k-1} = N_k/2$. Thus, the interpolation error is essentially evaluated at the nodes in $\Tau^k \setminus \Tau^{k-1}$, which we denote as
\begin{equation}\label{eqn:int_err}
d_i^k = e_{2i-1}^k =  q^k_{2i-1} - (P^k_{k-1} \q^{k-1})_{2i-1}, \quad 1 \leq i \leq N_{k-1} .
\end{equation}
Given $\q^{k-1}$ and $\db^k$, we can recover $\q^k$ using \eqref{eqn:decimator} and \eqref{eqn:int_err}. By iteratively applying this procedure, the data $\q^k$ on the finest mesh can be fully encoded using the multi-resolution representation 
\begin{equation}\label{eqn:mr_rep}
\{ \q^0, \db^1, \db^2,...,\db^K\} .
\end{equation}
This multiresolution representation \eqref{eqn:mr_rep} for a function $f \in \mathcal{B}_\Omega$ is convenient to perform data compression. The easiest compression strategy \cite{ARAND2000} corresponds to setting the coefficients $d_i^k$ in \eqref{eqn:int_err} to zero based on a suitable threshold $\epsilon^k \geq 0$:
\[
\widehat{d}_i^k = 
 \mathcal{G}(d_i^k;\epsilon^k) =
\begin{cases} 0 & \quad \text{if } |d_i^k| \leq \epsilon^k\\
   										 d_i^k & \quad \text{otherwise}.
					                  \end{cases}
\]
The compressed representation is then given by
\begin{equation}\label{eqn:cmr_rep}
\{ \f^0, \widehat{\db}^1, \widehat{\db}^2,...,\widehat{\db}^K\} .
\end{equation}
The procedures for compressed encoding and decoding are listed in Algorithms \ref{alg:encode} and \ref{alg:decode}.

\begin{algorithm}
\KwIn{Highest resolution data $\f^K$, number of levels $K$, number of points $N_0$ on coarsest mesh, ENO order $\pdeg$, threshold parameters $\epsilon$ and $t$.}
\KwOut{Multi-resolution representation $\{ \f^0, \widehat{\db}^1, ..., \widehat{\db}^K \}$.}
\For{$k=K$ \KwTo $1$}{
$\f^{k-1} = D^{k-1}_k \f^k$}
$\hat{\f}^0 = \f^0$\\
\For{$k=1$ \KwTo $K$}{
$\widehat{f}^k_0 = f^K_0$\\
Construct $P_{k-1}^k$ using Algorithm \ref{alg:eno_int_select} and equation  \eqref{eqn:eno_int}\\
$\widetilde{\f}^k = P_{k-1}^k \widehat{\f}^{k-1}$\\
$N = N_0 2^{k-1}$\\
\For{$i=1$  \KwTo $N$}{
$d_i^k = f_{2i-1}^k - \widetilde{f}_{2i-1}^k$\\
$\epsilon^k = \epsilon t^{K-k}$\\
$\widehat{d}_i^k = \mathcal{G}(d_i^k; \epsilon^k )$\\
$\widehat{f}^k_{2i-1} = \widetilde{f}_{2i-1}^k + \widehat{d}_i^k$\\
$\widehat{f}^k_{2i}  = \widehat{f}^{k-1}_{i}$}
}
\Return $\{ \f^0, \widehat{\db}^1, ..., \widehat{\db}^K \}$
\caption{Compressed encoding \cite{ARAND2000}}
\label{alg:encode}
\end{algorithm}

\begin{algorithm}
\KwIn{Multi-resolution representation $\{ \f^0, \widehat{\db}^1, ..., \widehat{\db}^K \}$, number of levels $K$, number of cells $N_0$ on coarsest mesh, ENO order $\pdeg$. }
\KwOut{Decoded function $\widehat{f}^K$.}
$\hat{\f}^0 = \f^0$\\
\For{$k=1$ \KwTo $K$}{
Construct $P_{k-1}^k$ using Algorithm \ref{alg:eno_int_select} and equation  \eqref{eqn:eno_int}\\
$\widehat{\f}^k = P_{k-1}^k \widehat{\f}^{k-1} + \widehat{\db}^k$
}
\Return $\widehat{\f}^K$
\caption{Decoding multi-resolution data \cite{ARAND2000}}
\label{alg:decode}
\end{algorithm}

The following result is known on the error bounds for the compressed encoding in the form \eqref{eqn:cmr_rep}, the proof can be found in \cite{ARAND2000}.

\begin{proposition}
Let $\{\Omega^k\}_{l=0}^K$ be a sequence of nested uniform meshes discretizing the interval $[c,d]$ generated according to \eqref{eqn:partitions} for some positive integer $N_0 > 1$. Assume that some $f \in \mathcal{B}[c,d]$ is encoded using thresholds
\begin{equation}\label{eqn:threshold}
\epsilon^k = \epsilon t^{K-k}, \quad 0 < t < 1.
\end{equation}
to give rise to the multi-resolution representation of the form \eqref{eqn:cmr_rep}. If $\widehat{\f}^K$ is the decoded data, then
\begin{equation}\label{eqn:err_estimate}
\|\f^K - \widehat{\f}^K\|_n \leq C_n \epsilon \quad \text{for } n =\infty, 1,2,
\end{equation}
where $C_\infty=(1-t)^{-1}$, $C_1 = (b-a)(1-t)^{-1}$ and $C_2 = \sqrt{(b-a)(1-t^2)^{-1} }$. This estimate is independent of the interpolation procedure used to encode and decode the data.
\end{proposition}

\section{Proof of Theorem \ref{thm:ENOSR2}}\label{app:proof-ENOSR2}

We present some properties of our adaptation of the ENO-SR algorithm. To prove the second-order accuracy, we state some results due to \cite{ACDD2005} in a slightly adapted form. 

\begin{lemma}\label{lem:groupsize-bis}
The groups of adjacent $B$ intervals are at most of size 2. They are separated by groups of adjacent $G$ intervals that are at least of size 2. 
\end{lemma}
\begin{proof}
Note that our detection algorithm is the same as the one in \cite{ACDD2005} for $m=3$. The result then follows from their Lemma 1. 
\end{proof}
\begin{lemma}\label{lem:L2-cohen}
Let $f$ be a globally continuous function with a bounded second derivative on $\mathbb{R}\backslash \{z\}$ and a discontinuity in the first derivative at a point $z$. Define the critical scale
\begin{equation}\label{eqn:hc}
    h_c:=\frac{\abs{[f']}}{4\sup_{x\in\mathbb{R}\backslash \{z\}}\abs{f''(x)}},
\end{equation}
where $[f']$ is the jump of the first derivative $f'$ at the point $z$. Then for $h<h_c$, the interval that contains $z$ is labelled $B$. 
\end{lemma}
\begin{proof}
See Lemma 2 in \cite{ACDD2005}. 
\end{proof}
\begin{lemma}\label{lem:L3-cohen}
There exist constants $C>0$ and $0<K<1$ such that for all continuous $f$ with uniformly bounded second derivative on $\mathbb{R}\backslash \{z\}$ and for $h<Kh_c$ with $h_c$ defined by equation (\ref{eqn:hc}), the following holds: 
\begin{enumerate}
    \item The singularity $z$ is contained in an isolated $B$ interval $I_i^k$ or in a $B$-pair $(I_i^k,I^k_{i+1})$. 
    \item The two polynomials $p^k_{i-2}$ and $p^k_{i+2}$ (or $p^k_{i-1}$ and $p^k_{i+3}$) have only one intersection point $y$ inside $I_i^k$ or $I_i^k\cup I^k_{i+1}$, respectively. 
    \item The distance between $z$ and $y$ is bounded by 
    \begin{equation}\label{eqn:distance-zy}
        \abs{z-y}\leq \frac{C\sup_{x\in\mathbb{R}\backslash \{z\}}\abs{f''(x)}h^2}{\abs{[f']}}.
    \end{equation}
\end{enumerate}
\end{lemma}
\begin{proof}
This is a light adaptation of Lemma 3 in \cite{ACDD2005}. The proof remains the same, after one minor change. We take $I=[b,c]$ to be equal to $I_{-1}\cup I_0\cup I_1$, which does not affect equation (38) in the proof. In fact, all other steps of their proof remain valid. It only must be noted that the constant $C$ in equation (\ref{eqn:distance-zy}) of this paper and equation (37) in \cite{ACDD2005} do not necessarily agree. 
\end{proof}

We now prove Theorem \ref{thm:ENOSR2} of Section \ref{sec:ENO-SR-DNN}, based on the proof of Theorem 1 in \cite{ACDD2005}. 
Let $h_c$ be as in Lemma \ref{lem:L2-cohen} and let $K$ be as in Lemma \ref{lem:L3-cohen}. Note that we can write 
\begin{equation*}
    f=f_- \mathbbm{1}_{(-\infty,z]}+f_+\mathbbm{1}_{(z,+\infty)}
\end{equation*}
where $f_-,f_+$ are $C^2$ on $\mathbb{R}$ such that  
\begin{equation*}
    \sup_{\mathbb{R}\backslash \{z\}}\abs{f''_{\pm}} \leq \sup_{\mathbb{R}\backslash \{z\}}\abs{f''}.
\end{equation*}

Let us consider some interval $I_0=[b,c]=[b,b+h]$. First consider the case $0<h<Kh_c$. Suppose it was labelled as good. Lemma \ref{lem:L2-cohen} then guarantees that $I_0$ does not contain $z$. It then follows directly from the theory of Lagrange interpolation that \begin{equation}\label{eqn:second-order-I0}
    \abs{f(x)-\mathcal{I}^hf(x)}\leq Ch^2 \sup_{\mathbb{R}\backslash \{z\}}\abs{f''}
\end{equation}
for all $x\in I_0$. Now suppose that $I_0$ was labelled bad. As a consequence of Lemma \ref{lem:groupsize-bis}, $I_{-2}$ and $I_2$ are good intervals and therefore do not contain the discontinuity. If in addition $z\not\in I_{-1}\cup I_0 \cup I_1$, then equation (\ref{eqn:second-order-I0}) holds again for all $x\in I_0$ since $\mathcal{I}^hf(x)$ is either equal to $p_{-2}(x)$, $p_0(x)$ or $p_2(x)$. On the other hand, if $z\in I_{-1}\cup I_0 \cup I_1$ then Lemma \ref{lem:L3-cohen} guarantees the existence of a single intersection point $y\in I_{-1}\cup I_0 \cup I_1$ of $p_{-2}$ and $p_2$. Assume now without loss of generality that $z \leq y$. In this case, equation (\ref{eqn:second-order-I0}) holds for all $x\in[b,z]\cup[y,c]$. It thus remains to treat the case $z<x<y$. For such $x$, we have
\begin{equation*}
     \abs{f(x)-\mathcal{I}^hf(x)} = \abs{f_+(x)-p_{-2}(x)} \leq \abs{f_+(x)-f_-(x)}+\abs{f_-(x)-p_{-2}(x)}
\end{equation*}
where the second term is again bounded by $Ch^2 \sup_{\mathbb{R}\backslash \{z\}}\abs{f''}$. We can use a second-order Taylor expansion for the first term to derive
\begin{equation*}
    \abs{f_+(x)-f_-(x)}\leq (y-z)([f']+2h\sup_{\mathbb{R}\backslash \{z\}}\abs{f''})\leq \frac{3}{2}\abs{[f']}(y-z)
\end{equation*}
where in the last inequality we used that $h<h_c$. By invoking the third part of Lemma \ref{lem:L3-cohen}, we find indeed that equation (\ref{eqn:second-order-I0}) holds again. This concludes the proof for the case $h<Kh_c$. 

Now suppose that $h\geq Kh_c$. First define 
\begin{equation*}
    f_2(x) = f(x) - [f'](x-z)_+
\end{equation*}
for $x\in\mathbb{R}$. Furthermore, by the definition of $h_c$ in Lemma \ref{lem:L2-cohen}, we find that for $h \geq Kh_c$,
\begin{equation} \label{eqn:above-hc}
    [f'] = 4h_c\sup_{\mathbb{R}\backslash \{z\}}\abs{f''} \leq C_0h \sup_{\mathbb{R}\backslash \{z\}}\abs{f''},
\end{equation}
where $C_0>0$ does not depend on $f$. We distinguish two cases. 

\textbf{Case 1: } $\mathcal{I}^h(x)=p_0(x)$ for all $x\in I_0$. If $z\not \in I_0$, second-order accuracy as in equation (\ref{eqn:second-order-I0}) is immediate. If not, more work is needed. Define 
\begin{equation*}
    g_1(x) = \frac{[f'](x_0-z)_+}{h}(x-x_{-1})
\end{equation*}
and note that $p_0-g_1$ is the linear interpolation between $(x_{-1},f_2(x_{-1}))$ and $(x_0,f_2(x_0))$. Since $f_2$ is $C^2$ we know that $p_0-g_1$ is a second-order accurate approximation of $f_2$ on $I_0$, such that equation (\ref{eqn:second-order-I0}) holds. We then calculate for $x\in I_0$,
\begin{equation*}
\begin{split}
     \abs{f(x)-p_0(x)} &\leq \abs{f_2(x)-(p_0(x)-g_1(x))} + \abs{[f'](x-z)_+-g_1(x)}\\
     &\leq C_1h^2\sup_{\mathbb{R}\backslash \{z\}}\abs{f''} + [f']\left(\abs{(x-z)_+}+\frac{(x_0-z)_+}{h}\abs{x-x_{-1}}\right)\\
     & \leq C_1h^2\sup_{\mathbb{R}\backslash \{z\}}\abs{f''} + C_0h \sup_{\mathbb{R}\backslash \{z\}}\abs{f''}(h+h)\\
     &= Ch^2\sup_{\mathbb{R}\backslash \{z\}}\abs{f''},
\end{split}
\end{equation*}
where we used inequality (\ref{eqn:above-hc}). 

\textbf{Case 2: } $\mathcal{I}^h(x)=p_{-2}(x)\mathbbm{1}_{(-\infty,y]}(x)+p_{2}(x)\mathbbm{1}_{(y,+\infty)}(x)$ for $x\in I_0$, where $y$ is the intersection point of $p_{-2}$ and $p_{2}$. If $z\not\in \cup_{q=-2}^2 I_{q}$, then inequality (\ref{eqn:second-order-I0}) follows immediately for $x\in I_0$. Consider now the case that $z\in \cup_{q=-2}^2 I_{q}$ and assume without loss of generality $z\leq y$. Let $x\in I_0$ be arbitrary. It follows that equation (\ref{eqn:second-order-I0}) also holds immediately for this $x$ if $y\leq x$ or $x \leq z$. It suffices to find a bound for when $x_{-3}\leq z\leq x \leq y$. Define
\begin{equation*}
    g_2(x) = \frac{[f'](x_{-2}-z)_+}{h}(x-x_{-3}). 
\end{equation*}
Note that $p_{-2}-g_2$ is an affine function through $(x_{-3},f_2(x_{-3}))$ and $(x_{-2},f_2(x_{-2}))$. It follows that 
\begin{equation*}
\begin{split}
     \abs{f(x)-p_{-2}(x)} &\leq \abs{f_2(x)-(p_{-2}-g_2(x))} + \abs{[f'](x-z)_+-g_2(x)}\\
     &\leq C_1h^2\sup_{\mathbb{R}\backslash \{z\}}\abs{f''} + [f']\left(\abs{(x-z)_+}+\frac{(x_{-2}-z)_+}{h}\abs{x-x_{-3}}\right)\\
     & \leq C_1h^2\sup_{\mathbb{R}\backslash \{z\}}\abs{f''} + C_0h \sup_{\mathbb{R}\backslash \{z\}}\abs{f''}(3h+3h)\\
     &= Ch^2\sup_{\mathbb{R}\backslash \{z\}}\abs{f''},
\end{split}
\end{equation*}
where we used again inequality (\ref{eqn:above-hc}) and the bounds $\abs{x-x_{-3}}\leq 3h$ and $x_{-3}\leq z$. This concludes the proof of Theorem \ref{thm:ENOSR2}. 

\section{Proof of Theorem \ref{thm:approx-enosr-bound}}\label{app:acc-approx-enosr}

In what follows, we let $x^*=x^{k+1}_{2i-1}$, $f^*=\mathcal{I}^{h_k}_{i}f(x^{k+1}_{2i-1})$,  $\hat{f}_{\epsilon} = \hat{f}^{k+1}_{2i-1,\epsilon}$ and $X^3_{l} = X^3_{n_{i,l}}$ for $1\leq l \leq 4$ (where $X^3_{n_{i,l}}$ is as in the proof of Theorem \ref{thm:eno-sr-ann}). We assume without loss of generality that $[c,d]\subset[0,\infty)$.  Furthermore we simplify the notation by dropping the index $k$ and setting $i=0$. It follows from the proof of Theorem \ref{thm:ENOSR2} that the results holds for $h\geq Kh_c$, since $\hat{f}_{\epsilon}$ is a convex combination of $p_{-2}(x^*)$, $p_{0}(x^*)$ and $p_{2}(x^*)$ for any value of $X^3$. We therefore assume in the following that $h<Kh_c$. The proof consists of an extensive case study, visualized in Figure \ref{fig:case-study}. 

\begin{figure}[h!]
    \centering
\tikzstyle{block} = [rectangle, draw, 
    text centered, rounded corners, minimum width=6em,minimum height=5em]

\tikzstyle{block2} = [rectangle, draw, 
    text centered, rounded corners,  minimum width=12em,minimum height=5em]

\tikzstyle{block3} = [rectangle, draw, 
    text centered, rounded corners, minimum height=10em]

\tikzstyle{block4} = [rectangle, draw, 
    text centered, rounded corners, minimum height=7em]
    
\tikzstyle{ellipse} = [circle, draw, 
    text centered, rounded corners, minimum height=5em]
\tikzstyle{line} = [draw, -latex']
\centering
\resizebox{\textwidth}{!}{
\begin{tikzpicture}[node distance = 2em, auto]
\node [block,align=center] (layer0a) {$X_1^3=0$};
\node [ellipse,below=of layer0a,align=center] (layer0b) {Case 7};
\node [block,right=of layer0a, align=center] (layer1a) {$0<X_1^3<\epsilon$};
\node [ellipse,below=of layer1a,align=center] (layer1b) {Case 6};
\node [block,right=of layer1a, align=center] (layer2a) {$X_1^3=0$};
\node [ellipse,below=of layer2a,align=center] (layer2b) {Case 5};
\node [block,right=of layer2a, align=center] (layer3a) {$0<X_2^3<\epsilon$};
\node [ellipse,below=of layer3a,align=center] (layer3b) {Case 4};
\node [block,right=of layer3a, align=center] (layer4a) {$X_3^3=0$};
\node [ellipse,below=of layer4a,align=center] (layer4b) {Case 3};
\node [block,right=of layer4a, align=center] (layer5a) {$0<X_3^3<\epsilon$};
\node [ellipse,below=of layer5a,align=center] (layer5b) {Case 2};
\node [ellipse,right =of layer5a,align=center] (layer5c) {Case 1};

\path [line] (layer0a) -- node{no}(layer1a);
\path [line] (layer1a) -- node{no}(layer2a);
\path [line] (layer2a) -- node{no}(layer3a);
\path [line] (layer3a) -- node{no}(layer4a);
\path [line] (layer4a) -- node{no}(layer5a);

\path [line] (layer0a) -- node{yes}(layer0b);
\path [line] (layer1a) -- node{yes}(layer1b);
\path [line] (layer2a) -- node{yes}(layer2b);
\path [line] (layer3a) -- node{yes}(layer3b);
\path [line] (layer4a) -- node{yes}(layer4b);
\path [line] (layer5a) -- node{yes}(layer5b);
\path [line] (layer5a) -- node{no}(layer5c);
\end{tikzpicture}
}
\caption{Overview of the case study done in the proof of Theorem \ref{thm:approx-enosr-bound}.}\label{fig:case-study}
\end{figure}
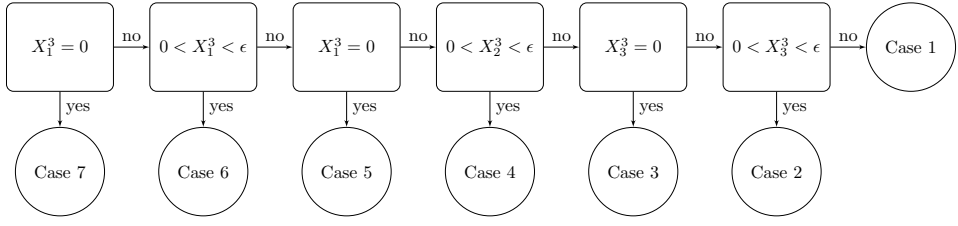

\textbf{Case 1: } In this case $\alpha=1$ and $\beta=1$, therefore $\hat{f}_{\epsilon} = p_{-2}(x^*) = f^*$. 

\textbf{Case 2: } We have that $\alpha=1$ and $0<X^3_3<\epsilon$, therefore $\hat{f}_{\epsilon} = (1-\beta)p_{2}(x^*)+\beta p_{-2}(x^*)$ where $\beta$ can take any value in $(0,1)$. From \eqref{eqn:X^3-bis}, it follows that the condition  $0<X^3_3<\epsilon$ corresponds to
\begin{equation}
    0 < \abs{b_{-2}-b_{2}}-x^*\abs{a_{-2}-a_{2}} < \epsilon,
\end{equation}
where $a_{-2}$ and $a_2$ are the slopes of $p_{-2}$ and $p_2$, respectively. Recall that in \eqref{eqn:X^3-bis} the assumption that $x_i^k=i$ was made. It can however be seen that $X^3_3$ is invariant to grid translations and scaling. Indeed, when the grid size is scaled by $h$ one needs to replace $x^*$ by $hx^*$ and $a_{\pm2}$ by $a_{\pm2}/h$; this leaves $X^3_3$ unchanged. If we now observe that from $\alpha=1$ follows that $\abs{a_{-2}-a_2}\neq0$, then we can define $y$ as the unique intersection point of $p_{-2}$ and $p_2$, and obtain
\begin{equation}
    0<y-x^*<\frac{\epsilon}{\abs{a_{-2}-a_2}}.
\end{equation}
Furthermore note that we can write $p_{\pm2}(x)=a_{\pm2}(x-y)+p_{\pm2}(y)$ where by definition $p_{-2}(y)=p_{2}(y)$. This leads to the estimate
\begin{equation}
    \abs{p_{-2}(x^*)-p_{2}(x^*)} = \abs{a_{-2}-a_2}(y-x^*) < \epsilon. 
\end{equation}

\textbf{Case 3: }  In this case $\alpha=1$ and $\beta=0$, therefore $\hat{f}_{\epsilon} = p_{2}(x^*) = f^*$. 

\textbf{Case 4: } By looking at the definition of $X^3_2$ in \eqref{eqn:X^3-bis}, we see that $0<h\abs{(a_{-2}-a_2)}<\epsilon$ (where the factor $h$ was added to remove the assumption that $x_i^k=i$). Furthermore we have that $X^3_1>0$. If $z\not\in I_{-1}\cup I_0 \cup I_1$ then Lemma \ref{lem:groupsize-bis} and Lemma \ref{lem:L3-cohen} guarantee that $\hat{f}_{\epsilon}$ is a second-order accurate approximation of $f^*$. If $z\in I_{\pm1}$, then $p_0$ is a second-order accurate approximation of $p_{\mp2}$ on $I_0$. In addition, this leads to the bound 
\begin{equation}
\begin{split}
       \abs{p_{\pm2}(x^*)-p_{0}(x^*)} &\leq \abs{p_{\pm2}(x^*)-p_{\mp2}(x^*)}+\abs{p_{\mp2}(x^*)-p_{0}(x^*)}\\
       &= \abs{(a_{2}-a_{-2})(x^*-y)} +\abs{p_{\mp2}(x^*)-p_{0}(x^*)}\\
      &\leq \frac{3}{2}\epsilon + Ch^2\sup\abs{f''}. 
\end{split}
\end{equation}
Finally we treat the case where $z\in I_0$. Define $p^*_0$ as the affine function through $(x_{-1},p_{-2}(x_{-1}))$ and $(x_{0},p_{2}(x_{0}))$. It then follows from $h\abs{(a_{-2}-a_2)}<\epsilon$ that
$h\abs{(a_{\pm2}-a^*_0)}<\epsilon$ where $a^*_0$ is the slope of $p^*_0$. 
We also have that $p^*_0$ is a second-order accurate approximation of $p_0$ on the interval $I_0$. This then leads to 
\begin{equation}
\begin{split}
    \abs{p_{\pm2}(x^*)-p_{0}(x^*)} &\leq \abs{p_{\pm2}(x^*)-p_{0}^*(x^*)}+\abs{p_{0}^*(x^*)-p_{0}(x^*)}\\
       &\leq \abs{(a_{\pm2}-a_{0}^*)\frac{h}{2}} +\abs{p_{0}^*(x^*)-p_{0}(x^*)}\\
      &\leq \frac{\epsilon}{2} + Ch^2\sup\abs{f''}.
      \end{split}
\end{equation}

\textbf{Case 5: } In this case $\alpha=0$ and therefore $\hat{f}_{\epsilon}=p_0(x^*)=f^*$. 

\textbf{Case 6: } In this case we only know that $I_0$ is a bad interval, since $X^3_1>0$. It may or may not contain the discontinuity. If $z\not\in I_{-1}\cup I_0 \cup I_1$ then Lemma \ref{lem:groupsize-bis} and Lemma \ref{lem:L3-cohen} guarantee that $\hat{f}_{\epsilon}$ is a second-order accurate approximation of $f^*$. We therefore assume in the following that $z\in I_{-1}\cup I_0 \cup I_1$. 

In the proof of Theorem \ref{thm:eno-sr-ann}, we introduced the quantity $X^2_q$ as a smoothness indicator for the interval $I_q$. We first investigate how the quantities $X^2_q$ for $f$ are related to the same quantities for the piecewise linear function defined by $g(x) = f(z)+f'(z-){(x-z)}+[f'](x-z)_+$, which we will denote by $\widehat{X}^{2}_q$. Let $\Delta_q=\Delta^2_h f(x_q)$ and $\widehat{\Delta}_q = \Delta^2_h g(x_q)$, as defined in  \eqref{eqn:second-diff}, and define $\Phi_q = \Delta_q - \widehat{\Delta}_q$. Since $g$ is a second-order accurate approximation of $f$, we obtain that 
\begin{equation}
    \abs{\Phi_q}\leq C_q h^2 \sup\abs{f''},
\end{equation}
where the constant $C_q$ is independent of $f$ and $h$. Using the triangle inequality and its reverse, we find that for $m,n\in\mathbb{N}_0$,
\begin{equation}
\begin{split}
    \abs{\widehat{\Delta}_m}-\abs{\widehat{\Delta}_n} 
    &\leq\abs{\widehat{\Delta}_m+\Phi_m}-\abs{\widehat{\Delta}_n+\Phi_n} +\abs{\Phi_m}+\abs{\Phi_n}\\
    &\leq \abs{\Delta_m}-\abs{\Delta_n} + (C_m+C_n)h^2 \sup\abs{f''}.
\end{split}
\end{equation}
Now assume that $0<X^2_q<\epsilon$ for some index $q$. We can then conclude that \begin{equation}\label{eqn:x-hat-bound}
    0\leq \widehat{X}^{2}_q \leq X^{2}_q + C h^2 \sup\abs{f''} < \epsilon + C h^2 \sup\abs{f''}.
\end{equation}
This allows us to restrict our further calculations to the piecewise linear function $g$. We assume that $z\in I_s$ for some index $s$, and that $[f']>0$. Furthermore we denote by $\overline{x}$ the midpoint of $I_s$. In Table \ref{tab:x-hat} we calculate  $\widehat{X}^2_{s-1},\widehat{X}^2_{s}$ and $\widehat{X}^2_{s+1}$. Note that for $s\not\in\{-1,0,1\}$, second-order accuracy is immediate. 


\begin{table}[h!]
    \centering
    \begin{tabular}{|c|c|c|c|c|c|c|} \hline
     $x_q$ & $x_{s-2}$ & \multicolumn{2}{c|}{$x_{s-1}$} &\multicolumn{2}{c|}{$x_{s}$} & $x_{s+1}$ \\\hline
        $\abs{\widehat{\Delta}_q}$ & $0$  & \multicolumn{2}{c|}{$\qquad\quad [f'](x_s-z)\qquad\quad$}&  \multicolumn{2}{c|}{$\qquad\quad[f'](z-x_{s-1})\qquad\quad$} &  $0$ \\\hline
        $(\abs{\widehat{\Delta}_q}-M_q)_+$ & $0$ &  \multicolumn{2}{c|}{$2[f'](\overline{x}-z)_+$} & \multicolumn{2}{c|}{$2[f'](z-\overline{x})_+$} &  $0$ \\\hline
        $(\abs{\widehat{\Delta}_q}-N_q^+)_+$ & $0$ &  \multicolumn{2}{c|}{$2[f'](\overline{x}-z)_+$} & \multicolumn{2}{c|}{$[f'](z-x_{s-1})$} & $0$ \\\hline
        $(\abs{\widehat{\Delta}_q}-N_q^-)_+$ & $0$ & \multicolumn{2}{c|}{$[f'](x_s-z)$} & \multicolumn{2}{c|}{$2[f'](z-\overline{x})_+$}  & $0$ \\\hline
        $\widehat{X}^2_{s-1},\widehat{X}^2_{s},\widehat{X}^2_{s+1}$ &  \multicolumn{2}{c|}{$2[f'](\overline{x}-z)_+$} &\multicolumn{2}{c|}{$[f'](2\abs{\overline{x}-z}+\min\{x_s-z,z-x_{s-1}\})$} &\multicolumn{2}{c|}{$2[f'](z-\overline{x})_+$} \\\hline 
    \end{tabular}
    \caption{Calculation of $\widehat{X}^2_{s-1},\widehat{X}^2_{s},$ and $\widehat{X}^2_{s+1}$.}
    \label{tab:x-hat}
\end{table}

First assume that $s=0$, in this case $x^*=\overline{x}$. From the table, it follows that $\widehat{X}^2_0\geq [f']h/2$. Combining this with \eqref{eqn:x-hat-bound} leads to the bound $[f']h/2<\epsilon + C h^2 \sup\abs{f''}$. We define again $p^*_0$ as the affine function through $(x_{-1},p_{-2}(x_{-1}))$ and $(x_{0},p_{2}(x_{0}))$, which is a second-order accurate approximation of $p_0$. We also introduce two new second-order accurate approximations of $p_0$ on $I_0$, 
\begin{equation}
\begin{split}
    p_0^-(x) &= p_{-2}(x)+[f']\frac{(x_0-z)_+}{h}(x-x_{-1}),\\
    p_0^+(x)    &= p_{ 2}(x)+[f']\frac{(z-x_{-1})_+}{h}(x-x_{0}).
\end{split}
\end{equation}
Indeed, $p_0^-(x_{-1}) = p_0^*(x_{-1}) = p_{-2}(x_{-1})$ and $p_0^-(x_{0}) = p_{-2}(x_{0})+[f'](x_0-z)_+$ are both second-order accurate on $I_0$, therefore $p_0^-$ is a second-order accurate approximation on $I_0$ of $p^*_0$ and hence of $p_0$. A similar reasoning holds for $p^+_0$. This then leads to 
\begin{equation}\label{eqn:calc-proof1}
    \begin{split}
        \abs{p_{\pm2}(x^*)-p_{0}(x^*)} &\leq \abs{p_{\pm2}(x^*)-p_{0}^{\pm}(x^*)}+\abs{p_{0}^{\pm}(x^*)-p_{0}(x^*)}\\
       &\leq [f']h/2 +\abs{p_{0}^{\pm}(x^*)-p_{0}(x^*)}\\
      &\leq \epsilon + Ch^2\sup\abs{f''}.
    \end{split}
\end{equation}

Next we assume that $s=1$. Since $X^3_1>0$, we have that $x_{0}\leq z < \overline{x}$. We distinguish two subcases. First, if $x_{0}\leq z \leq \overline{x}/2$ then $X^3_1 = 2[f'](\overline{x}-z)_+\geq [f']h/2$. Equation \eqref{eqn:x-hat-bound} is still valid and a calculation as in \eqref{eqn:calc-proof1} leads to the bound 
\begin{equation}
    \abs{p_{\pm2}(x^*)-p_{0}(x^*)} \leq \frac{5}{4}\epsilon + C h^2 \sup\abs{f''}
\end{equation}
where we used that $(z-x_{-1})_+\leq 5h/4$. Second, if $\overline{x}/2 < z < \overline{x}$ then the intersection point $y$ of $p_{-2}$ and $p_2$ will be located in $I_1$. This result can be deduced from the proof of Lemma 3 in \cite{ACDD2005}, by taking $I=I_s$ in the beginning of the proof. The second-order accuracy then follows from Case 1 or Case 2, depending on the value of $X^3_3$. The case $s=-1$ is completely analogous.

\textbf{Case 7: } In this case $\alpha=0$ and therefore $\hat{f}_{\epsilon}=p_0(x^*)=f^*$. 

We can now summarize all seven cases by the error bound
\begin{equation}
    \abs{\hat{f}_{\epsilon}-f^*} \leq Ch^2\sup\abs{f''}+\frac{3}{2}\epsilon, 
\end{equation}
which concludes the proof. 

\section{Weights of trained DeLENO networks}\label{sec:trained-weights}

We can now compare the weights and biases of the trained networks to the theoretical ones from Section \ref{sec:ENO-DNN}. As the networks do not have an accuracy of 100\%, it comes as no surprise that these do not agree. We list the obtained weights and biases for the trained third-order DeLENO interpolation network. 
\begin{equation}\label{DeLENO3int-trained}
    \begin{split}
            \W^{1} &= \begin{pmatrix}
    1.1951 &   2.0433&  -11.7410&    5.6383\\
    2.9216 &  -2.8703 &  -2.5077 &   2.4624\\
   -2.2775 &   7.6890  & -7.2667  &  2.4914\\
    3.2909 &  -5.8431   &-5.6085   & 3.4171\\
\end{pmatrix}, \quad \bb^{1}  = \begin{pmatrix}    -0.1069\\
   -0.3615\\
    0.0389\\
    0.0605\\
    \end{pmatrix},\\
       W^{2}  &= \begin{pmatrix}
  -11.6122 &   4.2986&   10.7356 &   8.1240\\
   11.5929 &  -4.2767 & -10.7357  & -8.1316\\
\end{pmatrix}, \quad \bb^{2}  = \begin{pmatrix}    -2.5193 \\
    2.4493\end{pmatrix}.
    \end{split}
\end{equation}
The theoretical counterparts can be found in equations \eqref{eqn:DeLENO3int-L1} and \eqref{eqn:DeLENO3int-L2}. Below we list the obtained weights and biases for the trained fourth-order DeLENO interpolation network. 
\begin{equation}\label{DeLENO4int-trained}
    \begin{split}
        \W^{1} &= \begin{pmatrix}
     -0.0559 &  -1.0026&    1.1115&    1.001&1   -1.0569&   -0.0001\\
  -0.3547&    0.3557&   -0.8777&    2.0707&   -1.1983&   -0.0051\\
  -0.0060 &  -0.6155 &   1.6342 &  -1.4526 &   0.4599 &  -0.0370\\
  0  & 0  &  0  &  0  &  0  &  0\\
   0.0011   &-0.1965   & 0.7964   &-1.1817   & 0.7805   &-0.1913\\
  -0.3324&    1.2088&   -1.6946&    1.0632   &-0.2479&   -0.0043\\
   0.1432 &  -0.6459 &   0.9448 &  -0.5434    &0.1271 &  -0.0154\\
  -0.0076  & -0.4239  &  0.8604  & -0.0248&   -0.8755  &  0.4517\\
   0.0196   &-0.1527   & 0.6286   &-0.9194 &   0.6069   &-0.1619\\
  -0.0088   &-0.2571    &1.0627  & -1.6288  &  1.0899  & -0.2714\\
\end{pmatrix},\bb^{1}  = \begin{pmatrix}    
   -0.0005 \\
   0.0029 \\
  -0.0005 \\
  -0.1217 \\
   0.0012 \\
   0.0007 \\
  -0.0010 \\
   0.0010 \\
   0.0027 \\
  -0.0005 \\
    \end{pmatrix},\\
         \W^{2} &= \begin{pmatrix}
    0&    1.6803&    0.2910&   -3.1738 &   0 &   1.5946\\
    0 &  -1.9935 &  -0.7975 &   3.2015  &  0  & -1.9124\\
   0   &-0.3125   & 2.7085   & 0.7264 &  0&   -0.2819\\
    0   & 0 &  0    &0   & 0   &0\\
   0&   -1.3886&    0.6976&    0.6071  & 0 &  -1.3514\\
    0&    2.1097&    0.8511&   -3.6655  &  0&    2.0568\\
    0 &   0.6645 &   0.1485 &  -1.2784   &0  &  0.5052\\
   0   &-0.0061   &-1.4376   &-0.0073 &  0&   -0.0082\\
   0&    0.4580 &  -1.0432&    0.0474  &  0&    0.4542\\
    0&    0.6597 &  -2.6379&   -0.4588  & 0 &   0.7357\\
\end{pmatrix}^T, \quad \bb^{2}  = \begin{pmatrix}   
  -0.0349 \\
   0.0993 \\
   0.0463 \\
   0.0284 \\
  -0.0570 \\
   0.0438 \\
    \end{pmatrix},
    \end{split}
\end{equation}
    \begin{equation}
\begin{split}
             \W^{3} &= \begin{pmatrix}
  0&0&0&0&0&0\\
   0&0&0&0&0&0\\
   0 &  8.0933  &  0.6463 &  -5.7734&    0  &  8.3502 \\
   0  &  2.0780  &-10.0148 &  -0.3565&   0   & 1.8241 \\
\end{pmatrix}, \quad \bb^{2}  = \begin{pmatrix}    
  -0.0316 \\
  -0.0432 \\
  -0.3800 \\
   0.4131 \\
    \end{pmatrix},\\   
             \W^{4} &= \begin{pmatrix}
 0&   0&    2.3377&  -11.5452 \\
  0&   0&    2.2965&   11.1520 \\
  0 &  0 & -12.6485 &  -0.8555 \\
\end{pmatrix}, \quad \bb^{4}  = \begin{pmatrix}    
   1.6841 \\
  -7.5807 \\
   8.2575 \\
    \end{pmatrix}.  
    \end{split}
\end{equation}
Note that these matrices and vectors again differ significantly from the theoretical weights and biases from equations (\ref{eqn:DeLENO4int-L1}-\ref{eqn:DeLENO4int-L3}). This shows that there are multiple neural networks which can approximate the ENO interpolation procedure very well.

\end{document}